\documentclass[a4paper, reqno,10pt]{amsart}
\usepackage{geometry}

\title[{}]{Orthosymplectic Feigin--Semikhatov duality}

\author[Justine Fasquel]{Justine Fasquel$^1$}\thanks{$^1$ Supported by a University of Melbourne Establishment Grant and Andrew Sisson Support Package 2023.}
\address[J.F.]{School of Mathematics and Statistics, University of Melbourne, Parkville, Australia, 3010}
\email{justine.fasquel@unimelb.edu.au}

\author[Shigenori Nakatsuka]{Shigenori Nakatsuka$^2$}\thanks{$^2$ Supported by JSPS Overseas Research Fellowships Grant Number 202260077.}
\address[S.N.]{Department Mathematik, FAU Erlangen–Nürnberg, Cauerstraße 11, 91058, Erlangen, Germany}
\email{shigenori.nakatsuka@fau.de}



\usepackage{latexsym}
\usepackage{amsmath}
\usepackage{amsfonts}
\usepackage{amssymb}
\usepackage{amsxtra}
\usepackage{mathrsfs}
\usepackage{eucal}
\usepackage[all]{xy}
\usepackage{color}
\usepackage{braket}
\usepackage{graphics, setspace}
\usepackage{etoolbox, textcomp}
\usepackage{comment}
\usepackage{changepage}
\usepackage{xcolor}
\definecolor{rouge}{rgb}{0.85,0.1,.4}
\definecolor{bleu}{rgb}{0.1,0.2,0.9}
\definecolor{violet}{rgb}{0.7,0,0.8}
\usepackage[colorlinks=true,linkcolor=bleu,urlcolor=violet,citecolor=rouge,breaklinks]{hyperref}
\usepackage{cleveref}
\usepackage{lscape}
\usepackage{caption}
\usepackage{subcaption}
\usepackage{enumitem}

\usepackage{tikz}			
\usetikzlibrary{matrix}
\usetikzlibrary{automata, arrows.meta, positioning,cd}
\usepackage{dynkin-diagrams}
\usetikzlibrary{cd}			

\newtheorem{definition}{Definition}[section]
\newtheorem{proposition}[definition]{Proposition}
\newtheorem{theorem}[definition]{Theorem}
\newtheorem{corollary}[definition]{Corollary}
\newtheorem{lemma}[definition]{Lemma}
\newtheorem{conjecture}[definition]{Conjecture}

\newtheorem{MainThm}{Main Theorem}
\theoremstyle{remark}
\newtheorem{remark}[definition]{Remark}
\newtheorem{example}[definition]{Example}
\numberwithin{equation}{section}

\newcommand{\Z}{\mathbb{Z}}
\newcommand{\R}{\mathbb{R}}
\newcommand{\C}{\mathbb{C}}

\newcommand{\Hom}{\operatorname{Hom}}
\newcommand{\Aut}{\operatorname{Aut}}

\newcommand{\Com}{\operatorname{Com}}
\newcommand{\id}{\operatorname{id}}
\newcommand{\Ker}{\operatorname{Ker}}

\newcommand{\ch}{\operatorname{ch}}

\newcommand{\cA}{\mathcal{A}}

\newcommand{\lenght}[1]{\ell(#1)}

\newcommand{\Q}{\mathbb{Q}}

\newcommand{\rP}{\mathrm{P}}
\newcommand{\rQ}{\mathrm{Q}}
\newcommand{\g}{\mathfrak{g}}
\newcommand{\h}{\mathfrak{h}}

\newcommand{\gl}{\mathfrak{gl}}
\newcommand{\sll}{\mathfrak{sl}}
\newcommand{\osp}{\mathfrak{osp}}
\newcommand{\symp}{\mathfrak{sp}}
\newcommand{\so}{\mathfrak{so}}

\newcommand{\Top}{\mathrm{top}}
\newcommand{\ind}{\mathrm{Ind}}

\newcommand{\Mod}{\text{-}\mathrm{mod}}
\newcommand{\Irr}{\mathrm{Irr}}

\newcommand{\W}{\mathcal{W}}

\newcommand{\VA}{V^k(\g)}
\newcommand{\WA}{\mathcal{W}^k(\g,f)}

\newcommand{\Wsub}{\mathcal{W}^k_{D^+}(n,1)}
\newcommand{\Wpr}{\mathcal{W}^k(\so_{2n+1})}
\newcommand{\sWsub}{\mathcal{W}_k^{D^+}(n,1)}
\newcommand{\sWsuper}{\mathcal{W}_{\ell}^{D^-}(n,1)}
\newcommand{\Wsuper}{\mathcal{W}^{\ell}_{D^-}(n,1)}
\newcommand{\subW}{\mathcal{W}^{k}_{D^+}}
\newcommand{\sprW}{\W_{D^-}^\ell}
\newcommand{\ssubW}{\mathcal{W}_{k}^{D^+}}
\newcommand{\ssprW}{\W^{D^-}_\ell}

\newcommand{\NO}[1]{:\!#1\!:}
\newcommand{\ssqrt}[1]{\operatorname{\sqrt{\smash[b]{#1}}}}

\newcommand{\KL}{\mathbf{KL}}
\renewcommand{\Pr}{\mathrm{Pr}}
\newcommand{\weyl}{\mathbb{V}}
\newcommand{\parity}[1]{\overline{\mathstrut #1}}
\newcommand{\fsub}{\ensuremath{f_{\mathrm{sub}}}}
\newcommand{\fpr}{\ensuremath{f_{\mathrm{prin}}}}
\newcommand{\hv}[1]{\mathbf{e}^{#1}}
\newcommand{\dz}{\ensuremath{\mathrm{d}z}}
\newcommand{\poch}[1]{(#1;q)_{\infty}}
\newcommand{\half}{\frac{1}{2}}

\newcommand\doi[2]{\href{http://dx.doi.org/#1}{#2}}
\newcommand{\arxiv}[2]{\href{https://arxiv.org/abs/#1}{#2}}

\newcommand{\red}[1]{\textcolor{red}{#1}}

\begin{document}
\maketitle

\begin{abstract}
We study the representation theory of the subregular $\W$-algebra $\W^k(\so_{2n+1},\fsub)$ of type $B$ and the principal $\W$-superalgebra $\W^\ell(\osp_{2|2n})$, which are related by an orthosymplectic analogue of Feigin--Semikhatov duality in type $A$. 
We establish a block-wise equivalence of weight modules over the $\W$-superalgebras by using the relative semi-infinite cohomology functor and spectral flow twists, which generalizes the result of Feigin--Semikhatov--Tipunin for the $\mathcal{N}=2$ superconformal algebra. In particular, the correspondence of Wakimoto type free field representations is obtained. 
When the level of the subregular $\W$-algebra is exceptional, we classify the simple modules over the simple quotients $\W_k(\so_{2n+1},\fsub)$ and $\W_\ell(\osp_{2|2n})$ and derive the character formulae.
\end{abstract}

\section{Introduction}

Let $\g$ be a simple Lie superalgebra. To any nilpotent element $f\in\g$ of even parity and any complex number $k\in\C$, referred to as the level, one associates the $\W$-(super)algebra $\WA$. It is obtained by applying the quantum Drinfeld--Sokolov reduction (a.k.a BRST reduction) to the universal affine vertex superalgebra $\VA$ \cite{FF1,KRW03,KW} and can be regarded as an affinization of the finite $\W$-algebra $U(\g,f)$ introduced by Premet \cite{Pr}. Indeed, $U(\g,f)$ can be obtained from $\WA$ by taking its Zhu's algebra \cite{FZ,Z}.

The $\W$-algebras were first constructed by Feigin and Frenkel \cite{FF1} for principal nilpotent elements $f\in\g$. The corresponding principal $\W$-algebras $\W^k(\g,\fpr)$ -- usually denoted by $\W^k(\g)$ -- and their representation theories have been studied intensively \cite{A1,A2,AF,D,DR,FKW92,R}.
In particular, they enjoy the so-called Feigin--Frenkel duality, stating an isomorphism between $\W^k(\g)$ and $\W^\ell({^L}\g)$,
where ${}^L\g$ is the Langlands dual of $\g$, when the levels $k$ and $\ell$ satisfy the relation
\begin{align}\label{eq:duality_condition_levels}
r^\vee(k+h^\vee)(\ell+{}^Lh^\vee)=1,\quad (k,\ell)\neq (-h^\vee,-{}^Lh^\vee).
\end{align}
Here $r^\vee$ is the lacing number of $\g$,  $h^\vee$ and ${}^Lh^\vee$ are the dual Coxeter numbers of $\g$ and ${}^L\g$ respectively.
This duality can be seen as a vertex-algebraic refinement of the isomorphism between the centers of the enveloping algebras of $\g$ and ${}^L\g$ -- which corresponds to the isomorphism of their Zhu's algebras. 
Moreover, the Feigin--Frenkel duality plays a fundamental role in the quantum geometric Langlands program as $\W^k(\g)$ at the critical level $k=-h^\vee$ describes the center of the enveloping algebra of the affine Kac--Moody algebra at the corresponding level \cite{AF,Frenkel,Gai16}.

In the last decade, generalizations of the Feigin-Frenkel duality to other $\W$-superalgebras, called hook-type $\W$-superalgebras, were found by Gaiotto and Rap\v c\'ak in the connection to four-dimensional gauge theories in physics \cite{GR}.
They assert isomorphisms between cosets of the maximal affine vertex subalgebras inside hook-type $\W$-superalgebras.
When the affine vertex subalgebras are non-super (or of $\osp_{1|2n}$-type), these isomorphisms can be proven by using certain universal $\W_\infty$-algebras \cite{CL1,CL2}.
In addition, although the dualities assert isomorphisms between the cosets, one can relax the ``coset constraint'' to obtain a reconstruction theorem for hook-type $\W$-superalgebras from one side to the other \cite{CGN,CLNS}.
It opens up the possibility to understand the representation theory of $\W$-superalgebras, for which only a little is known so far.

The representation theory of the $\mathcal{N}=2$ superconformal algebra, that is $\W^\ell(\sll_{2|1})$, can be seen as a prototype of such an example, which attracted a lot of interest (see for instance \cite{Ad1,Ad2,CLRW,DVPYZ,FSST,FST,KaSu,S3}). 
The $\W$-superalgebra $\W^\ell(\sll_{2|1})$ is realized in terms of $V^k(\sll_2)$ \cite{KaSu} and vice versa \cite{FST}, which amounts to a block-wise equivalence of their weight module categories, including the natural isomorphisms of the spaces of intertwining operators and the fusion rules \cite{CGNS,FSST}.

The generalization of this construction for higher rank cases of type $A$ was originally suggested by Feigin and Semikhatov \cite{FS} and relates the subregular $\W$-algebras and the principal $\W$-superalgebras, namely 
\begin{align*}
    \W^k(\sll_n,\fsub),\quad \W^\ell(\sll_{n|1}).
\end{align*}
The Gaiotto--Rap\v c\'ak duality in this case is the isomorphism between their coset subalgebras by a rank one Heisenberg vertex subalgebra, which is proven in \cite{CGN,CL1} by different methods, and the block-wise equivalence of their weight module categories is established in \cite{CGNS} by one of the authors.

In this paper, we continue the study of the orthosymplectic analogue between
\begin{align*}
    \W_{D^+}^k=\W^k(\so_{2n+1},\fsub),\quad \W_{D^-}^\ell=\W^\ell(\osp_{2|2n}).
\end{align*}
In particular, we are interested in exceptional levels, for which the simple quotient of $\W^k(\so_{2n+1},\fsub)$ is lisse \cite{A3} and rational \cite{Mc}. 
It implies that the module category is finite, semisimple \cite{ABD04,DLM98,Z} and, moreover, admits a structure of modular tensor category \cite{H}. 
Although the rationality of vertex algebras is difficult to show in general, our understanding of the rationality of $\W$-algebras has been improved recently \cite{Ara13,A2,AvE,CL6,Fa,Mc}.
It is mainly based on the results of the admissible affine vertex algebras.
In contrast, the representation theory of affine vertex \emph{super}algebras is still quite mysterious, and thus the study of  $\W$-superalgebras is out of reach in general up to now.
The duality between $\W$-superalgebras and $\W$-algebras is very useful to reveal this mystery and provides a hint for a general pattern.
Indeed, the block-wise correspondence of module categories implies the rationality of the $\W$-superalgebra side as well.
We classify the simple modules of the subregular $\W$-algebra and of the principal $\W$-superalgebra as well through the equivalence and derive their character formulae.
Despite the development of the rationality results on exceptional $\W$-algebras, the classification of simple modules is not well-developed beyond the principal nilpotent case \cite{A2}.
In the literature, we find the type $A$ case and some simply-laced cases \cite{AvE,CGNS}. Our result is a generalization of $\W^k(\symp_{4},\fsub)$ established by one of the authors \cite{Fa} and seems to be the first result on the orthosymplectic type in the super side beyond the rank one case \cite{Ad3}.

\subsection{Results}

In the present paper, we start by proving an embedding between subregular and principal $\W$-algebras of type $B$.
These types of embeddings, called \emph{inverse Hamiltonian reductions}, have been investigated a lot in type $A$ with renewed interest in recent years \cite{Ada19,ACG,AKR21,Fe,Fe23,FR22,Sem94}.
Our construction is an analogue of \cite{Fe} in type $B$ and generalizes the rank $n=2$ case presented in \cite{BM21}.
As an application, we relate the Fegin--Semikhatov duality to the Feigin--Frenkel duality: 

\begin{MainThm}[Corollary \ref{FS vs FF}]\label{mainthm:FFFS}
For the levels $k,\ell$ satisfying \eqref{eq:duality_condition_levels}
the diagram below commutes:
    \begin{center}
		\begin{tikzcd}[row sep=huge, column sep = huge]
			\W^k(\so_{2n+1},\fsub)
			\arrow[d, leftrightarrow, "\text{F.-S.}"']
			\arrow[r, hook, "iHR"]&
			\Wpr\arrow[d, leftrightarrow, "\text{F.-F.}"]&
            \hspace{-3cm}\otimes \mathcal{F}_1 
			\\
			\W^\ell(\osp_{2|2n})
			\arrow[r, hook, "\text{Wakimoto real.}"]&
			\W^\ell(\symp_{2n})&
            \hspace{-3.4cm}\otimes \mathcal{F}_2. 
		\end{tikzcd}
	\end{center}
Here $\mathcal{F}_i$ $(i=1,2)$ are free field vertex superalgebras.
\end{MainThm} 

In the previous picture, the relation between the free field algebras $\mathcal{F}_i$ ($i=1,2$) is given by the relative semi-infinite cohomology functors $\mathrm{H}^\pm_\lambda$, ($\lambda\in\C$).
We study these functors and use them in order to establish a block-wise equivalence of module categories of $\W$-superalgebras. 
More precisely, we work over the categories $\KL_{D^+}^k$ and $\KL_{D^-}^\ell$ of weight modules over $\W^k(\so_{2n+1},\fsub)$ and $\W^\ell(\osp_{2|2n})$, i.e., the modules which decompose into direct sums of Fock modules when restricted to their Heisenberg vertex subalgebras. 
They naturally decompose into 
$$\KL_{D^+}^k=\bigoplus_{[\lambda]\in \C/\Z}\KL_{D^+}^{k,[\lambda]},\quad \KL_{D^-}^\ell=\bigoplus_{[\lambda]\in \C/\Z}\KL_{D^-}^{\ell,[\lambda]}$$
such that $\W^k(\so_{2n+1},\fsub)$ and $\W^\ell(\osp_{2|2n})$ lie in the $[0]$-blocks, see Sect.~\ref{subsection_Equivalence of categories}. 
Moreover, we prove that the different relative semi-infinite cohomology functors are related by the spectral flow twists.

\begin{MainThm}[Theorem \ref{categorical equivalence}/\ref{equivalence of categories}]
For $\theta\in \Z$, the functors 
		\begin{align*}
			\mathrm{H}_{\lambda_\theta^+}^+\colon \KL_{D^+}^{k,[\lambda\varepsilon]}\ \rightleftarrows \ \KL_{D^-}^{\ell,[\lambda_\theta^+/\varepsilon]} \colon \mathrm{H}_{\lambda_\theta^+}^-
		\end{align*}
		are quasi-inverse to each other and give equivalences of categories.
Moreover, for $\theta_1,\theta_2\in\Z$, there are natural isomorphisms between the functors 
\begin{align*}
\mathrm{H}_{\lambda_{\theta_1}^+}^+\simeq\mathrm{H}_{{\lambda'}_{\theta_1}^+}^+\circ S_{\theta_2 },\quad 
S_{\theta_2-\theta_1}\circ \mathrm{H}_{\lambda_{\theta_1}^+}^+\simeq \mathrm{H}_{\lambda_{\theta_2}^+}^+,
\end{align*}
where $\lambda'=\lambda+\varepsilon\theta_2$ and $\lambda_\theta^+=\lambda-\theta\varepsilon^{-1}$, $\varepsilon=(k+h^\vee)^{1/2}$.
\end{MainThm}

\noindent This generalizes the result of Feigin--Semikhatov--Tipunin for the $\mathcal{N}=2$ superconformal algebra \cite{FST}.

One of the classic techniques to study representation theory is to use resolutions by fundamental modules. Free field representations through Wakimoto realizations of $\W$-superalgebras play such a role \cite{A5,FF88,Gen,Wak}.
We introduce and study the correspondence between several Wakimoto-type free field representations, which we call ``thick" and ``thin" since the size of such modules over $\W^{k}(\so_{2n+1},\fsub)$ and $\W^{\ell}(\osp_{2|2n})$ are different. 
On the super-side, they serve as a higher rank analogue of massive and topological Verma modules defined for the $\mathcal{N}=2$ superconformal algebra.

In the second half of the paper, we focus on admissible levels
\begin{equation}\label{eq:exceptional_levels_subreg}
   k=-h^\vee+\frac{p}{q},\qquad (p,q)=1,\ p\geq q,\ q=2n-1,2n.
\end{equation}
The simple subregular $\W$-algebra $\W_k(\so_{2n+1},\fsub)$ at such a level is said exceptional.
A long-standing conjecture of Kac--Wakimoto and Arakawa, recently proved in \cite{Mc}, states that $\W_k(\so_{2n+1},\fsub)$ is rational.
Expanding the strategy introduced previously in \cite{Fa} to prove the rationality of $\W_k(\symp_4,\fsub)$, we classify the simple $\W_{k}(\so_{2n+1},\fsub)$-modules. 
These simple modules are obtained by applying spectral flow twists to the BRST reduction image $\mathbf{L}_k(\lambda)$ of the affine admissible representations $L_k(\lambda)$ with dominants integral highest weights $\lambda\in\Pr^k_\Z$.
On the super-side, the simple $\W_{\ell}(\osp_{2|2n})$-modules are obtained as the image of the relative semi-infinite cohomology functors of the simple $\W_k(\so_{2n+1},\fsub)$-modules
$$\mathbf{L}_\ell^-(\lambda)=\mathrm{H}_{\varepsilon^{-1}[\lambda]}^+\left(\mathbf{L}_k(\lambda)\right),$$
and their spectral flow twists. More precisely, we have:

\begin{MainThm}[Theorem \ref{classification of simple modules}/\ref{thm:classification_super_case}]
For the levels $k,\ell$ satisfying \eqref{eq:duality_condition_levels} and \eqref{eq:exceptional_levels_subreg}, 
the complete set of simple modules over $\W_{k}(\so_{2n+1},\fsub)$ and $\W_{\ell}(\osp_{2|2n})$ are respectively given by 
\begin{align*}
    \Irr(\W_{k}(\so_{2n+1},\fsub))&=\{S_{\theta} \mathbf{L}_k(\lambda)\mid \lambda\in \Pr_\Z^k,\ 0\leq \theta<\bar{q} \},\\
    \Irr(\W_{\ell}(\osp_{2|2n}))&=\{S_{\theta}\mathbf{L}_\ell^-(\lambda)\mid \lambda \in \Pr_\Z^k,\ 0\leq \theta < p \}.
\end{align*}
\end{MainThm}

\noindent 
Their character formulae are derived in Proposition \ref{chracter formula 2} and Remark \ref{last remark} through resolutions by Wakimoto-type modules which are related by the correspondence.
The periodicity of spectral flow twists gives an interesting description of $\W_k(\so_{2n+1},\fsub)$ in terms of lattice vertex algebras, with which we study the unitary cases (see Sect. \ref{sec:unitarity}).

\subsection{Outlook}
As we classify the simple $\W_k(\so_{2n+1},\fsub)$-modules in the exceptional case, the next natural and desirable step is to describe the fusion rules. In this view, and based on the construction of the simple modules, we formulate a precise conjecture when $k$ is principal admissible, i.e., denominator $q=2n-1$ (see Conjecture~\ref{conj:fusion_rules}). In that case, the description is similar to the type $A$ case \cite{AvE,CGNS}.

Moreover, we have identified all the simple $\W_k(\so_{2n+1},\fsub)$-modules with the image of the BRST reduction of admissible representations of $L_k(\so_{2n+1})$ and their spectral flow twists. 
From a categorical point of view, we expect the existence of a nice braided tensor category closed under the spectral flow twists extending the Kazhdan--Lusztig category for $L_k(\so_{2n+1})$, and the BRST reduction to be a quotient functor from this category to the category of $\W_k(\so_{2n+1},\fsub)$-modules, see for instance \cite{ACF,CR1,CR2}. 

Secondly, Main Theorem~\ref{mainthm:FFFS} suggests that the dualities between hook-type $\W$-superalgebras are closely related to the original Feigin--Frenkel duality. The key ingredient to emphasize these links seems to be the inverse Hamiltonian reduction for $\W$-algebras \cite{Fe23}.
It is also natural to expect a finite analogue of these dualities. As a nice presentation of finite $\W$-algebras, called the infinitesimal Cherednik (or Hecke) algebra, is obtained on the non-super side \cite{LT,T}, it would be useful and interesting to obtain a similar one on the super side.
We hope to come back to these topics in the near future.

\subsection{Plan of the paper}
The paper is organized as follows.
In Sect.~\ref{sec:definitions}, we present the main objects of our interest, that are the $\W$-superalgebras $\W^k(\so_{2n+1},\fsub)$ and $\W^\ell(\osp_{2|2n})$ ($n\geq2$).
After we recall the Feigin--Semikhatov duality, we relate it to the Feigin--Frenkel duality in Sect.~\ref{sec:FF-FS}.
To do so, we establish the inverse Hamiltonian reduction between the subregular and the principal $\W$-algebras associated to $\so_{2n+1}$ (see Theorem~\ref{inverse HR}).
Sect.~\ref{sec:spectral_flow} reviews some general facts on spectral flow twists and Sect.~\ref{Sec: equivalnce of categories} introduces the relative semi-infinite cohomology functors. Following \cite{CGNS}, we show that these functors define equivalences between categories of weight modules over $\W^k(\so_{2n+1},\fsub)$ and $\W^\ell(\osp_{2|2n})$.
In addition, we show that they are naturally related by spectral flow twists (Theorem~\ref{equivalence of categories}).
In Sect.~\ref{sec:free_fields}, the correspondence between Wakimoto-type free field representations is studied (Theorem \ref{thm:wakimoto_correspondence}).  
In Sect.~\ref{sec:subreg_symplectic_walgebra} and \ref{sec:rational_principal}, we classify the simple modules of the $\W$-superalgebras in the exceptional case (see Theorems~\ref{classification of simple modules} and \ref{thm:classification_super_case}).
In addition, the unitary cases (for the subregular $\W$-algebra) are studied in depth in Sect.~\ref{sec:unitarity}.
Finally, in Sect.~\ref{sec:characters}, we compute the characters of simple modules.

\subsection*{Acknowledgements} 
We thank Thomas Creutzig, Ching Hung Lam, and Andrew Linshaw for useful discussions.
We also thank a referee for their useful comment.
We are grateful to the organizers of the conferences ``Vertex Algebras and Representation Theory'' (CIRM, Luminy, 2022), ``Representation Theory XVII-XVIII'' (IUC, Dubrovnik, 2022-23), and ``Quantum symmetries: Tensor categories, Topological quantum field theories, Vertex algebras'' (CRM, Montréal, 2022). The project was initiated during the first conference, and significant progress was made during subsequent events. 
This paper has been completed during a Matrix Minor Program in Creswick (Australia).
S.N. thanks the institute and the University of Melbourne for their hospitality.

\section{Feigin--Semikhatov duality of orthosymplectic type}\label{sec:definitions}
In this section, we introduce the $\W$-superalgebras
$$\Wsub=\W^k(\so_{2n+1},f_{\mathrm{sub}}),\quad \Wsuper=\W^\ell(\osp_{2|2n},f_{\mathrm{prin}})$$
with $n \geq 2$.

\subsection{$\W$-superalgebras}\label{setting of our W-superalgebras}

The $\W$-algebra $\Wsub$ is the universal affine $\W$-algebra associated with a subregular nilpotent element in the Lie algebra $\so_{2n+1}$ of type $B_n$, whereas $\Wsuper$ is the principal $\W$-superalgebra associated to the Lie superalgebra $\osp_{2|2n}$.
Both are obtained from the quantized Drinfeld--Sokolov reduction of a universal affine vertex (super)algebra, $V^k(\so_{2n+1})$ and $V^\ell(\osp_{2|2n})$, respectively.
We refer to \cite{KRW03} for a detailed construction of general $\W$-superalgebras.
In the following, we sometimes abbreviate them as 
$$\subW=\Wsub,\quad \sprW=\Wsuper,$$
and denote their unique simple quotients by 
$$\ssubW=\sWsub,\quad \ssprW=\sWsuper.$$

For later use, we recall some data used in the construction of these $\W$-superalgebras.
We realize $\so_{2n+1}$ as 
$$\so_{2n+1} = \left\{ \begin{pmatrix} 0 & u & v \\
	- v^T &A & B \\
	- u^T & C & -A^T \\
\end{pmatrix} \left| 
\begin{array}{l}
A,B,C \in \gl_{n}\\
u,v\in\C^n\\
B = - B^T, \, C= - C^T 
\end{array}\right. \right\}\subset \mathrm{End}(\C^{2n+1})$$
where $x^T$ is the transpose matrix of $x$.
We fix a Cartan subalgebra $\h_+$ of $\so_{2n+1}$ to be the subset of diagonal matrices.
We identify $\h_+=\bigoplus_{i=1}^n \C \epsilon_i$ with its dual through the normalized invariant form $\half\mathrm{tr}$, so that $(\epsilon_i,\epsilon_j)=\delta_{i,j}$.
Let $\Delta$, $\Pi=\{\alpha_1,\ldots, \alpha_n\}$ and $\check{\Pi}=\{\check{\alpha}_1,\ldots, \check{\alpha}_n\}$ be the set of roots, simple roots, and simple coroots respectively. 
In particular, the simple roots and coroots are realized as 
\begin{align}\label{realization of Cartan of type B}
\alpha_i=
\begin{cases}
\epsilon_i-\epsilon_{i+1} & (i\neq n),\\
\epsilon_n & (i= n),
\end{cases}
\quad 
\check{\alpha}_i=
\begin{cases}
\epsilon_i-\epsilon_{i+1} & (i\neq n),\\
2\epsilon_n & (i= n).
\end{cases}
\end{align}
The highest root is $\theta=\epsilon_1+\epsilon_2$ and the highest short root is $\theta_s=\epsilon_1$.
The root lattice $\rQ$ and coroot lattice $\check{\rQ}$ are identified as
$$\rQ=\bigoplus_{i=1}^n\Z \epsilon_i,\quad  \check{\rQ}=\left\{\sum \lambda_i \epsilon_i\in \rQ\left| \sum \lambda_i= 0\ \mathrm{mod}\ 2 \right.\right\},$$
and the Weyl group is the semi-direct product $W=S_n \ltimes \Z_2^n$ where $S_n$ is the $n$-th symmetric group.
Finally, the fundamental weights $\varpi_i$ and coweights $\check{\varpi}_i$ are given by
$$\varpi_i=\begin{cases}\epsilon_1+\cdots +\epsilon_i & (i\neq n)\\ \tfrac{1}{2}(\epsilon_1+\cdots +\epsilon_n) & (i=n) \end{cases},\quad \check{\varpi}_i=\epsilon_1+\cdots +\epsilon_i.$$

Labeling the basis of the natural representation $\C^{2n+1}$ of $\so_{2n+1}$ as $\{0,\pm1,\ldots,\pm n\}$, a subregular nilpotent element is expressed as
$$\fsub=\sum_{i=2}^{n-1}(E_{i+1,i}-E_{-i,-i-1})+(E_{0,n}-E_{-n,0})$$
where the $E_{i,j}$'s denote the elementary matrices.
We embed $\fsub$ into an $\sll_2$-triple $\{e^+,h^+,f^+\}$ -- where $f^+=\fsub$ -- using the Jacobson--Morosov theorem.
The semisimple element $h^+=2\sum_{i=1}^n(n-i+1)(E_{i,i}-E_{-i,-i})-2(E_{1,1}-E_{-1,-1})$ defines a \emph{good grading}\footnote{See \cite{EK05} for a precise definition and classification of good gradings.}
on $\so_{2n+1}$.
The Fig.~\ref{good grading of so} below gives the good grading in terms of the weighted Dynkin diagram. 
Equivalently, it is the eigenvalue decomposition of the adjoint action of 
\begin{align}\label{grading element}
x_0^+:=\half h^+=\sum_{i=2}^n\check{\varpi}_i\in \h_+.
\end{align}

By \cite[Theorem~4.1]{KW04}, $\Wsub$ is strongly generated by a family of fields corresponding to a certain basis of the centralizer of $\fsub$ inside $\so_{2n+1}$ and so, is of type 
\begin{align}\label{sub strong generating type}
\Wsub=\W(1,2,4,\ldots,2n-2,(n)^2),
\end{align}
at noncritical levels $k\neq-h^\vee_+$, where $h^\vee_+$ is the dual Coxeter number of $\so_{2n+1}$.
Let us denote by $J^+(z)$, $L(z)$, $W_i(z)$ ($i=4,\ldots,2n-2$) and $G^\pm(z)$ the corresponding generating fields. 
In particular, $L(z)$ is the Virasoro field of central charge 
\begin{align}\label{cc}
c_k=1-\frac{h^\vee_{+}}{k+h^\vee_{+}}((n-1)(k+h^\vee_{+})-n)(2n(k+h^\vee_{+})-(2n+1))
\end{align}
and the strong generating type \eqref{sub strong generating type} is given by the $L_{0}$-eigenvalues of the generating fields.
Moreover, the field $J^+(z)$, which corresponds to the first fundamental coweight $\check{\varpi}_1=E_{1,1}-E_{-1,-1}$, generates a Heisenberg vertex subalgebra $\pi^{J^+}\subset \subW$. It satisfies the OPE 
\begin{equation}\label{level of J+}
J^+(z)J^+(w)\sim \frac{(k+h^\vee_+)-1}{(z-w)^2},\qquad h^\vee_+=2n-1.
\end{equation}
Since the $W_i(z)$'s correspond to vectors in $\so_{2n+1}$ which commute with $\check{\varpi}_1$, they commute with the field $J^+(z)$.
Meanwhile, $G^\pm(z)$ satisfy the OPEs
\begin{equation}\label{eq:OPEJG}
J^+(z)G^\pm(w)\sim \frac{\pm G^\pm(w)}{(z-w)}.
\end{equation}

On the other hand, we realize $\osp_{2|2n}$ as 
$$\osp_{2|2n} = \left\{ \begin{pmatrix} a & 0 & s^T & -v^T \\ 
0 & -a & r^T & -u^T \\
	u & v  & A & B \\
	r & s  & C & -A^T \\
\end{pmatrix} \left| \begin{array}{l}
      A,B,C \in \gl_{n}\\
      u,v,s,t \in \C^n, a\in\C  \\
    B^T=B, C^T=C  
\end{array}\right.\right\}\subset \mathrm{End}(\C^{2|2n})$$
and fix a Cartan subalgebra $\h_-$ of $\osp_{2|2n}$ to be the subset of diagonal matrices.
We identify $\h_-=\bigoplus_{i=0}^n \C t_i$ with the normalized invariant form $-\mathrm{str}$ so that $(t_i,t_j)=(-1)^{\delta_{i,0}}\delta_{i,j}/2$.
The simple roots and coroots are realized as
\begin{align}\label{realization of Cartan of type C}
\alpha_i=
\begin{cases}
-(t_0+t_1) & (i=0),\\
(t_i-t_{i+1}) & (1\leq i<n),\\
2t_n & (i= n),
\end{cases}
\quad 
\check{\alpha}_i=
\begin{cases}
-2(t_0+t_1) & (i=0),\\
2(t_i-t_{i+1}) & (1\leq i<n),\\
2t_n & (i= n).
\end{cases}
\end{align}

Labeling the basis of the natural representation $\C^{2|2n}$ of $\osp_{2|2n}$ as $\{+,-,\pm1,\ldots,\pm n\}$, a principal nilpotent element is expressed as 
$$\fpr=\sum_{i=1}^{n-1}(E_{i+1,i}-E_{-i,-i-1})+E_{-n,n}.$$
As previously, we embed $\fpr$ into an $\sll_2$-triple $\{e^-,h^-,f^-\}$ whose semisimple element $h^-=2\sum_{i=1}^n(n-i+\tfrac{1}{2})(E_{i,i}-E_{-i,-i})+2(n-\tfrac{1}{2})(E_{+,+}-E_{-,-})$ defines a good grading on $\osp_{2|2n}$. This grading corresponds to the weighted Dynkin diagram of Fig.~\ref{good grading of osp}.

\begin{figure}[h]
	\begin{minipage}[l]{0.5\linewidth}
		\begin{center} 
			\caption{$\g=\so_{2n+1}$}\label{good grading of so}
			\begin{tikzpicture}
				{\dynkin[root
					radius=.08cm,labels={\alpha_1,\alpha_2,\alpha_{n-1},\alpha_n},labels*={0,1,1,1},edge length=1cm]B{oo.oo}};
			\end{tikzpicture}
		\end{center}	
	\end{minipage}
	\begin{minipage}[c]{0.49\linewidth}
		\begin{center}
			\caption{$\g=\osp_{2|2n}$}\label{good grading of osp}
			\begin{tikzpicture}
				{\dynkin[root
					radius=.08cm,labels={\alpha_0,\alpha_1,\alpha_{n-1},\alpha_n},labels*={0,1,1,1},edge length=1cm]C{to.oo}};	
			\end{tikzpicture}
		\end{center}
	\end{minipage}
\end{figure}

When $\ell\neq-h^\vee_-$ is non-critical, $\Wsuper$ is conformal of strong generating type 
$$\Wsuper=\W(1,2,4,\ldots,2n,(n+\tfrac{1}{2})^2).$$
Its structure is very similar to that of $\Wsub$.
In particular, the field $J^-(z)$ of weight one corresponds to the $0$-th fundamental coweight $\check{\varpi}_0=E_{+,+}-E_{-,-}$ and generates a Heisenberg vertex subalgebra $\pi^{J^-}\subset\Wsuper$.
It satisfies the OPE 
\begin{align}\label{Heisenberg in sprin}
J^-(z)J^-(w)\sim \frac{-2(\ell+h^\vee_-)+1}{(z-w)^2},\qquad h^\vee_-=n.
\end{align}
The field of weight two is the Virasoro field, and the fields of weight $n+\tfrac{1}{2}$, which we denote by $G^\pm_o(z)$, are the only \emph{odd} fields and satisfy the OPEs 
$$J^-(z)G^\pm_o(w)\sim \frac{\pm G^\pm_o(w)}{(z-w)}.$$

\subsection{Duality}\label{sec. duality}
The Feigin--Semikhatov duality of orthosymplectic type states an isomorphism of vertex algebras between cosets of the previously defined Heisenberg vertex subalgebras inside their corresponding $\W$-superalgebras \cite[Theorem 1.2]{CGN}:
\begin{align}\label{FS duality}
\Com\left(\pi^{J^+}, \subW\right)\simeq \Com\left(\pi^{J^-}, \sprW \right)
\end{align}
where the levels $k,\ell$ satisfy
\begin{align}\label{level condition for FS}
2(k+h^\vee_+)(\ell+h^\vee_-)=1,\quad (k,\ell)\neq (-2n+2,-n+\tfrac{1}{2}).
\end{align}
Throughout this paper, we assume that \eqref{level condition for FS} holds. Note that the excluded pair of levels $(-2n+2,-n+\tfrac{1}{2})$ is exactly the levels where the fields $J^\pm(z)$ degenerate.

The proof in \cite{CGN} relies on the free field realizations of $\W$-superalgebras
\begin{align}\label{FFR}
\subW\hookrightarrow \pi^{k+h^\vee_{+}}_{\mathfrak{h}_+}\otimes \Pi(0),\quad \sprW\hookrightarrow \pi^{\ell+h^\vee_{-}}_{\mathfrak{h}_-}\otimes V_\Z,
\end{align}
The maps are obtained as composition of the \emph{Miura map} \cite{KW04} defined for general $\W$-superalgebras, which are injective by \cite{Ara17,Nak21}, and free field realizations of affine vertex (super)algebras for $\sll_2$ and $\gl_{1|1}$, respectively.
Here $\pi^{k+h^\vee_{+}}_{\mathfrak{h}_+}$ (resp. $\pi^{\ell+h^\vee_{-}}_{\mathfrak{h}_-}$) is the Heisenberg vertex algebra associated with $\mathfrak{h}_+\subset \so_{2n+1}$ at level $k+h^\vee_+$ (resp. $\mathfrak{h}_-\subset \osp_{2|2n}$ at level $\ell+h^\vee_-$), $V_\Z$ is the lattice vertex superalgebra associated with the lattice $\Z$, spanned by $x$ with norm $(x,x)=1$, and 
$\Pi(0)$ is the half-lattice vertex algebra defined as the subalgebra 
$$\Pi(0):=\bigoplus_{m\in \Z}\pi^{x,y}_{m(x+y)}\subset V_{\Z\oplus \ssqrt{-1}\Z}$$ 
of the lattice vertex superalgebra $V_{\Z\oplus \ssqrt{-1}\Z}$ associated with the lattice $\Z\oplus \ssqrt{-1}\Z$, spanned by $x,y$ with norms $(x,x)=1=-(y,y)$ and $(x,y)=0$.
It is an extension of the Heisenberg vertex algebra $\pi^{x,y}=\pi^{x,y}_{0}$ by the  Fock modules $\pi^{x,y}_{m(x+y)}$ generated by the highest weight vector $\hv{m(x+y)}$.

As in the proof of the Feigin--Frenkel duality \cite{FF1}, the proof of the Feigin--Semikhatov duality \eqref{FS duality} reduces to the identification of \emph{screening operators} \cite{Gen17} for \eqref{FFR}, which characterize the $\W$-superalgebras inside the free field algebras at generic levels.
Here, they are given by 
\begin{align}\label{screenings in Kazama Suzuki}
Q_i^+=\int Y(\hv{A^+_i},z)\dz, \quad Q_i^-=\int Y(\hv{A^-_i},z)\dz,
\end{align}
respectively, with 
\begin{align*}
    &A_0^+=x,\qquad A_1^+=-\alpha_1+(x+y),\qquad A_i^+=-\alpha_i\ (i=2,\ldots,n),\\
    &A_0^-=-\alpha_0+x,\qquad A_i^-=-\alpha_i\ (i=1,\ldots,n).
\end{align*}
The description by the screening operators implies the following refinement of \eqref{FS duality} known as the \emph{Kazama--Suzuki type coset construction}.

\begin{theorem}[{\cite[Theorem 1.3]{CGN}}]\label{Coset theorem}
For the levels $(k,\ell)$ satisfying \eqref{level condition for FS}, we have the following isomorphisms of vertex superalgebras:
\begin{align}
\label{FST}& \Wsub\simeq \Com \left(\pi^{J^-_\Delta}, \Wsuper\otimes V_{\ssqrt{-1}\Z}\right),\\
\label{KS}& \Wsuper\simeq \Com \left(\pi^{J^+_\Delta}, \Wsub\otimes V_{\Z}\right).
\end{align}
where $\pi^{J^\pm_\Delta}$ denote the Heisenberg vertex algebras generated by the fields
$$J^+_\Delta(z)=J^+(z)- x(z),\qquad J^-_\Delta(z)=J^-(z)+ y(z).$$
Moreover, the isomorphisms \eqref{FST} and \eqref{KS} still hold by replacing $\Wsub$ and $\Wsuper$ by their simple quotients $\sWsub$ and $\sWsuper$.
\end{theorem}

Using the coset construction, we can compare the multiplicities of the Fock modules over $\pi^{J^\pm}$. 
We decompose 
\begin{align}\label{eq: decomposition}
    \subW\simeq \bigoplus_{m\in \Z}\mathscr{C}^k_{D^+,m}\otimes \pi^{J^+}_m,\qquad 
    \sprW\simeq \bigoplus_{m\in \Z}\mathscr{C}^\ell_{D^-,m}\otimes \pi^{J^-}_m
\end{align}
as modules over $\pi^{J^+}$ (resp.\ $\pi^{J^-}$). 
Then $\mathscr{C}^k_{D^+,m}$ (resp.\ $\mathscr{C}^\ell_{D^-,m}$) is naturally a $\mathscr{C}^k_{D^+}$-module (resp.\ $\mathscr{C}^\ell_{D^-}$-module) 
where we set 
\begin{align*}
    \mathscr{C}^k_{D^+}:=\Com \left(\pi^{J^+}, \subW\right),\quad \mathscr{C}^\ell_{D^-}:=\Com \left(\pi^{J^-}, \sprW\right).
\end{align*}
By replacing $\subW$ (resp.\ $\sprW$) by the simple quotient $\ssubW$ (resp.\ $\ssprW$), we have a similar decomposition and write $\mathscr{C}_{k,m}^{D^+}$ (resp.\ $\mathscr{C}_{\ell,m}^{D^-}$) for the multiplicity spaces. 
They are simple quotients of $\mathscr{C}^k_{D^+,m}$ (resp.\ $\mathscr{C}^\ell_{D^-,m}$) as $\mathscr{C}^k_{D^+}$-modules (resp.\ $\mathscr{C}^\ell_{D^-}$-modules).

\begin{corollary}\label{isom for the multiplicities}
For the levels $k,\ell$ satisfying \eqref{level condition for FS}, we have isomorphisms
$$\mathscr{C}^k_{D^+,m}\simeq \mathscr{C}^\ell_{D^-,m},\qquad \mathscr{C}_{k,m}^{D^+}\simeq \mathscr{C}_{\ell,m}^{D^-}, \quad (m\in \Z)$$
as modules over the coset $\mathscr{C}^k_{D^+}\simeq \mathscr{C}^\ell_{D^-}$.  
\end{corollary}

\proof
By setting 
$$\widehat{J}^-(z):=\frac{1}{k+h^\vee_+}(J^+(z)+(k+h^\vee_+-1)x(z))$$
in \eqref{KS}, it is straightforward to show that we have an isomorphism of vertex algebras $\pi^{J^+}\otimes \pi^x\simeq \pi^{J^+_\Delta}\otimes \pi^{J^-}$
and then an isomorphism between their modules
\begin{align*}
\pi^{J^+}_a\otimes \pi^x_b\simeq \pi^{J^+_\Delta}_{a-b}\otimes \pi^{J^-}_{b+\frac{a-b}{k+h^\vee_+}},\quad (a,b\in \Z).
\end{align*}
It follows from \eqref{eq: decomposition} that we have an isomorphism 
 \begin{equation*}
\subW\otimes V_{\Z}
\simeq \bigoplus_{a,b\in \Z} \mathscr{C}^k_{D^+,a}\otimes \pi^{J^-}_{a+b(-2(\ell+h^\vee_-)+1)}\otimes \pi^{J^+_\Delta}_{-b}
\end{equation*}
as modules over $\mathscr{C}^k_{D^+}\otimes \pi^{J^-}\otimes \pi^{J^+_\Delta}$.
As a module over $\sprW\otimes \pi^{J^+_\Delta}$ via \eqref{KS}, we have the decomposition
\begin{align}\label{decompositoin of the KS}
\subW\otimes V_{\Z}
\simeq \bigoplus_{a\in \Z} \mathcal{E}(a)\otimes \pi^{J^+_\Delta}_{-a}
\end{align}
where we set $\mathcal{E}(a)$ to be the $\sprW$-submodules
\begin{align}\label{multiplicity in the KS}
\mathcal{E}(a)\simeq \bigoplus_{m\in \Z} \mathscr{C}^{k}_{D^+,m}\otimes \pi^{J^-}_{m+a(-2(\ell+h^\vee_-)+1)}.
\end{align}
Since $\mathcal{E}(0)\simeq \W_{D^-}^\ell$, we obtain $\mathscr{C}^k_{D^+,m}\simeq \mathscr{C}^\ell_{D^-,m}$ by comparing \eqref{eq: decomposition} and \eqref{multiplicity in the KS}.
The second isomorphism is obtained similarly.
\endproof

\subsection{Some special levels}\label{sec:affine_conformal_embeddings}
Here we consider the noncritical levels when the simple $\W$-algebra $\ssubW$ has some special property: one of $\mathscr{C}_{D^+}^k$ or $\pi^{J+}$ appearing in the decomposition \eqref{eq: decomposition} degenerates to $\C$.
\\

\paragraph{Case (i): $\mathscr{C}_{D^+}^k$ degenerates.}
In this case, the embedding $\pi^{J^+}\hookrightarrow \subW$ is \emph{conformal}, i.e. preserves the conformal vector \cite{AMP}.
To detect such levels, we decompose the conformal vector $L=L_{\mathscr{C}}+L_{J^+}$
where $L_{J^+}\in \pi^{J^+}$ is given by the Segal--Sugawara construction and $L_{\mathscr{C}}=L-L_{J^+}$ is the conformal vector in $\mathscr{C}^k_{D^+}$. 
Since the image of $L_{\mathscr{C}}$ is zero in the simple quotient $\ssubW$ by assumption, its central charge must be zero.
This happens if and only if $k$ is one of 
$$k_1=-h^\vee_++\frac{n}{n-1},\qquad k_2=-h^\vee_++\frac{2n+1}{2n},$$ 
and thus the dual level $\ell$ corresponds to
$$\ell_1=-h^\vee_-+\frac{n-1}{2n},\qquad \ell_2=-h^\vee_-+\frac{n}{2n+1}.$$

\begin{proposition}\label{free field case}
The embedding $\pi^{J^+}\hookrightarrow \W^{D^+}_k(n,1)$ preserves the conformal vectors iff $k=k_1$, $k_2$. In these cases, we have isomorphisms of vertex (super)algebras
\begin{align*}
\begin{array}{ll}
\W_{k_1}^{D^+}(n,1)\simeq \pi^{J^+},\quad& \W_{k_2}^{D^+}(n,1)\simeq V_{\sqrt{2n}\Z},\\
\W_{\ell_1}^{D^-}(n,1)\simeq \pi^{J^-},\quad& \W_{\ell_2}^{D^-}(n,1)\simeq V_{\sqrt{2n+1}\Z}.
\end{array}
\end{align*}
\end{proposition}

\proof
For the first assertion, we have already shown that $k=k_1,k_2$ are necessary. The sufficiency follows from the second assertion, i.e., the isomorphisms for $\W_{k_i}^{D^+}$ $(i=1,2)$.
For $i=1$, the strong generators $W_j$ ($j\geq 4$) and $G^\pm$ have non-zero conformal weights with respect to $L_\mathscr{C}$: indeed $L_{\mathscr{C},0}W_j=L_{0}W_j=2j$ as $W_i\in \mathscr{C}_{D^+}^k$ and $L_{\mathscr{C},0}G^\pm=(L_{0}-L_{J^+,0})G^\pm=(n-\frac{n-1}{2})G^\pm$ as $G^\pm$ are highest weight vectors for $\pi^{J^+}$. Hence, by \cite[Theorem 1.2]{AMP}, $\W_{k_1}^{D^+}$ is a simple quotient of $\pi^{J^+}$, that is $\pi^{J^+}$ itself.
We postpone the proof in the case $i=2$ to Theorem~\ref{unitary case1}. Indeed, the level $k_2$ is known to be admissible and exceptional. We study these levels in more detail in Sect.~\ref{sec:unitarity}.
The isomorphisms for $\W_{\ell_i}^{D^-}$ are straightforward to obtain by Theorem~\ref{Coset theorem}.
\endproof

\paragraph{Case (ii): $\pi^{J^+}$ degenerates.} 
By \eqref{level of J+}, the corresponding level is $k=-h^\vee_++1$. 
In this case, it is obvious that $J^+$ maps to zero by considering the quotient to $\ssubW$, and so do $G^\pm$ as they have nonzero $J^+_0$-eigenvalues. 
Hence, $\ssubW$ is a quotient of a $\W$-algebra of type $\W(2,4,\ldots,2n-2)$. 
The latter is obtained by a quotient of an universal object, called the even spin $\W_\infty$-algebra $\W^{\mathrm{ev}}(c,\lambda)$ \cite{KL19}, which depends on two parameters $(c,\lambda)\in \C^2$ and is of strong generating type $\W(2,4,6,\ldots)$.
The well-known quotients of strong generating type $\W(2,4,\ldots,2n-2)$ are the principal $\W$-algebras $\W^\ell(\so_{2n-1})$ -- and $\W^{\ell'}(\symp_{2n-2})$ by the Feigin--Frenkel duality. 

Making the ansatz $\sWsub\simeq \W_\ell(\so_{2n-1})$, the coincidence of central charges implies that $\ell$ must be one of
$$\ell_1+h^\vee_{n-1}=1\qquad\text{and}\qquad\ell_2+h^\vee_{n-1}=\frac{2n-3}{2n}$$
where we set $h^\vee_m=2m-1$.
Based on lower rank examples, the first case seems plausible in general \cite{FLN}. In this case,
$\W^{\ell_1}(\so_{2n-1})$ is isomorphic to the $\mathrm{Sp}_{2n-2}$-orbifold $\mathcal{A}(n-1)^{\mathrm{Sp}_{2n-2}}$ of the rank $n-1$ symplectic fermion algebra $\mathcal{A}(n-1)$ \cite[Theorem~5.1]{KL19}, which is simple by \cite{DLM}.

Similar considerations apply to the super side at the dual level $\ell=-h^\vee_-+\frac{1}{2}$: ${J^-}$ belongs to the maximal ideal as well as the fields $G^\pm_o$.
Hence, at first sight, $\ssprW$ is of strong generating type $\W(2,4,\ldots,2n)$. 
Moreover, the first order pole in the OPE of $G^+_o(z)G^-_o(w)$ has conformal weight $2n$ and commutes with ${J^-}$. This suggests that the strong generator of conformal weight $2n$ appears in this pole and thus belongs to the maximal ideal and, therefore, $\ssprW$ should be of type $\W(2,4,\ldots,2n-2)$ too with the same central charge as $\ssubW$. 
We conjecture the following:
\begin{conjecture}\label{conjecture on collapsing}
    For $n\geq2$, we have an isomorphism of vertex algebras
    $$\W_{-h^\vee_{n}+1}^{D^+}(n,1)\simeq \W^{-h^\vee_{n-1}+1}(\so_{2n-1})\simeq \W_{-\frac{1}{2}h^\vee_{n}}^{D^-}(n,1).$$
\end{conjecture}
For $n=2$, $\W^{-h^\vee_{n-1}+1}(\so_{2n-1})$ is the Virasoro vertex algebra of central charge $-2$. In this case, the first isomorphism is established by the first-named author \cite{Fas21} and the second one can be checked using the explicit description of the $\W$-superalgebras $\sprW$.

\section{Feigin--Frenkel duality}\label{sec:FF-FS}
We relate the Feigin--Semikhatov duality \eqref{FS duality} to the Feigin--Frenkel duality \cite{FF1} between the principal $\W$-algebras of type $B$ and $C$ by adapting the technique called \emph{inverse Hamiltonian reduction} introduced previously to study subregular $\W$-algebras of type $A$ \cite{Ada19,AKR21,Fe}.

\subsection{Inverse Hamiltonian reduction}
Recall that we have an embedding of vertex algebras
\begin{align}\label{free field realization of subreg}
    \subW\overset{\mu_{\mathrm{sub}}}{\hookrightarrow} \bigcap_{i=0}^n \Ker Q_i^+ \subset \Pi(0)\otimes\pi^{k+h^\vee_{+}}_{\h_+},
\end{align}
such that $\mu_{\mathrm{sub}}$ is surjective for generic levels $k$, see \eqref{screenings in Kazama Suzuki}. 
By \cite{FF90b}, the principal $\W$-algebra $\Wpr$ also has an embedding, through the Miura map $\mu_{\mathrm{pr}}$
\begin{align}\label{Principal W-algebra of type B}
    \Wpr\overset{\mu_{\mathrm{pr}}}{\hookrightarrow}\bigcap_{i=1}^n\Ker Q_i^{\mathrm{pr}}\subset \pi^{k+h^\vee_+}_{\h_+},\quad Q_i^{\mathrm{pr}}=\int Y\left(\hv{-\alpha_i},z\right)\dz,
\end{align}
which is again surjective for generic $k$.
Hence, both $\W$-algebras are realized as subalgebras of $\Pi(0)\otimes\pi^{k+h^\vee_{+}}_{\h_+}$, which are characterized as the joint kernel of screening operators at generic levels. Moreover, most of these operators coincide: $Q_i^+=Q_i^{\mathrm{pr}}$ for $i=2,\ldots,n$. 
In order to relate $Q_1^+$ and $Q_1^{\mathrm{pr}}$, we use the automorphism of $\Pi(0)\otimes\pi^{k+h^\vee_{+}}_{\h_+}$ given in the lemma below.

\begin{lemma}\label{prop:automVA}
There is an automorphism of $\Pi(0)\otimes\pi^{k+h^\vee_+}_{\h_+}$ satisfying 
	\begin{align*}
		g\colon &\Pi(0)\otimes\pi^{k+h^\vee_+}_{\h_+}\overset{\sim} 
  {\longrightarrow}\Pi(0)\otimes\pi^{k+h^\vee_+}_{\h_+} \\
		&\hv{n(x+y)}\mapsto \hv{n(x+y)},\quad  \check{\alpha}_i\mapsto \check{\alpha}_i-\delta_{i,1}(k+h^\vee_+)(x+y) \quad (i=1,\ldots,n)\\
		&x\mapsto x-\tfrac{k+h^\vee_+}{2}(x+y)+\varpi_1,\quad y\mapsto y+\tfrac{k+h^\vee_+}{2}(x+y)-\varpi_1.
	\end{align*}
\end{lemma}

\proof
The proof is straightforward by checking the OPEs and constructing an inverse map.
\endproof

Under the automorphism on $\Pi(0)\otimes\pi^{k+h^\vee_+}_{\h_+}$, the joint kernel of the screening operators \eqref{free field realization of subreg} maps to the joint kernel of another set of screening operators.
It is straightforward to show that $Q_i^+=Q_i^{\mathrm{pr}}$ remains the same for $1\leq i\leq n$ whereas $Q_0^+$ is replaced by
\begin{align*}
S=\int  Y\left(\hv{x-\tfrac{k+h^\vee_+}{2}(x+y)-\varpi_1},z\right)\,\dz.
\end{align*}
Hence, by combining the embedding \eqref{free field realization of subreg} with the automorphism $g^{-1}$, we obtain the embedding 
\begin{align}\label{new free field realization of subreg}
    \subW\overset{\mu_{\mathrm{sub}}}{\hookrightarrow} \bigcap_{i=0}^n \Ker Q_i^+\xrightarrow{\substack{g^{-1}\\\sim}} \Ker S \cap \bigcap_{i=1}^n \Ker Q_i^{\mathrm{pr}} \subset \Pi(0)\otimes\pi^{k+h^\vee_{+}}_{\h_+}.
\end{align}

\begin{theorem}\label{inverse HR}
    For $k\neq-h^\vee_+$, there exists an embedding of vertex algebras 
\begin{align*}
   \mu_{\mathrm{iHR}}\colon \Wsub\hookrightarrow \Ker S \cap \left(\Pi(0)\otimes\Wpr\right),
\end{align*}
which is an isomorphism at generic levels.
\end{theorem}

For generic levels, the assertion follows immediately by comparing \eqref{Principal W-algebra of type B} and \eqref{new free field realization of subreg}. To deduce the assertion for all the non-critical levels, we use an argument of continuity.
Consider a vector space $V=\bigoplus_{a\in A}V_a$ graded by a set $A$ such that each $V_a$ is finite-dimensional. A family of graded vector subspaces $\{W_z=\bigoplus_{a\in A}W_z^a\}_{z\in U}$ over a connected subset $U\subset \C$ is said continuous if $\dim W_z^a$ is constant for each $a$, say $d_a$, and the induced map $U\rightarrow \operatorname{Gr}(d_a,V_a)$ to the Grassmannian manifold is continuous \cite{Tau11}. If we have two continuous families of vector subspaces $\{W^1_\alpha\}_{\alpha\in U}$, $\{W^2_\alpha\}_{\alpha\in U}$ such that $W^1_\alpha\subset W^2_\alpha$ holds for a dense subset $U^\circ\subset U$, then $W^1_\alpha\subset W^2_\alpha$ holds for $U$. This is a weaker version of \cite[Lemma~5.14]{CGN}.

\proof[Proof of Theorem \ref{inverse HR}]
All the Heisenberg vertex algebras $\pi^{k+h^\vee_+}_{\h_+}$ with $k$ in $U=\C\backslash \{-h^\vee_+\}$ are isomorphic. Thus, we may identify them and set $V=\Pi(0)\otimes \pi^{k+h^\vee_+}_{\h_+}$, which is graded by the highest weights and the conformal weights.
Since $\{\subW\}_{k\in U}$ and $\{\Pi(0)\otimes\Wpr\}_{k\in U}$ are continuous families such that $\subW\subset \Pi(0)\otimes\Wpr$ holds for generic $k$, the argument of continuity implies the inclusion for all $k\neq -h^\vee_+$. 
Next, the condition for an element of $V$ with homogeneous weight to lie in $\Ker S$ is a closed condition depending on $k$ algebraically. 
Since $\subW\subset \Ker S$ holds for generic $k$, it follows that $\subW\subset \Ker S$ for all $k\neq -h^\vee_+$ by the upper-semicontinuity of subspaces characterized by closed conditions.
\endproof

\begin{example}[cf.\ \cite{BM21}] The map $\mu_{\mathrm{iHR}}$ for $n=2$ is explicitly given by 
	\begin{align*}
	&J^+\mapsto \mathbf{y},\quad L\mapsto \tfrac{1}{2}(u^2-v^2)+2\partial\mathbf{x}-\tfrac{1}{k+3}S_2,\quad G^+\mapsto \hv{x+y},\\
	&G^-\mapsto (36S_4-\tfrac{k}{8}\partial^2S_2+\tfrac{1}{4}S_2^2-\tfrac{k+2}{2}(\partial S_2) \mathbf{x}-\tfrac{2k+5}{2}S_2 \partial \mathbf{x}+\tfrac{1}{2}S_2\mathbf{x}^2+\tfrac{1}{4}\mathbf{x}^4\\
        &\hspace{1.2cm}-\tfrac{4k+9}{2}\partial\mathbf{x}\cdot\mathbf{x}^2+\tfrac{7k^2+33k+39}{4}(\partial\mathbf{x})^2-\tfrac{(4k+9)(3k+7)(k+2)}{2}\partial^2\mathbf{x}\cdot\mathbf{x})\hv{-(x+y)}\
	\end{align*}
 where 
 \begin{align*}
    \mathbf{x}=\tfrac{k+1}{2}x+\tfrac{k+3}{2}y,\quad \mathbf{y}=\tfrac{k+3}{2}x+\tfrac{k+1}{2}y.
 \end{align*}
Here $S_2,S_4$ are the strong generators of $\W^k(\so_5)$ given in \cite[Theorem 4.1]{KW04} with conformal weights 2 and 4, respectively.
See Sect.~\ref{setting of our W-superalgebras} for the strong generators $J^+, L,G^\pm$ and \cite{Fas21} for their explicit OPEs. This implies that $\mu_{\mathrm{iHR}}$ is well-defined for all levels $k\in \C$.
\end{example}

\subsection{Feigin--Frenkel duality}
The embedding $\mu_{\mathrm{iHR}}$ induces an embedding
\begin{align}\label{iHR for coset}
\begin{split}
    \Com(\pi^{J^+_\Delta},\subW\otimes V_\Z)
&\hookrightarrow \Com(\pi^{J^+_\Delta},\Wpr \otimes \Pi(0)\otimes V_\Z)\\
&\subset \Com(\pi^{J^+_\Delta},\pi^{k+h^\vee_+}_{\h_+}\otimes \Pi(0)\otimes V_\Z)
\end{split}
\end{align}
whose image agrees with the kernel of the screening operator $S$ at generic $k$.
As the embedding \eqref{free field realization of subreg} maps the Heisenberg field $J^+$ to $\mathbf{1}\otimes\varpi_1-y\otimes \mathbf{1}$, we find that
\begin{align*}
   \mu_{\mathrm{iHR}}\colon \pi^{J^+_\Delta}\hookrightarrow \Pi(0)\otimes V_\Z,\quad J^+_\Delta\mapsto (-y+\tfrac{k+h^\vee_+}{2}(x+y))\otimes \mathbf{1}-\mathbf{1}\otimes y
\end{align*}
and that \eqref{iHR for coset} factors as 
$$\Com(\pi^{J^+_\Delta},\subW\otimes V_\Z)
\hookrightarrow \Wpr \otimes\Com(\pi^{J^+_\Delta},\Pi(0)\otimes V_\Z).$$
The following can be checked straightforwardly.

\begin{lemma}
There is an isomorphism of vertex superalgebras
\begin{align*}
    \begin{array}{llll}
        \phi\colon &V_\Z\otimes \pi^y &\xrightarrow{\simeq}  &\Com(\pi^{J^+_\Delta},\Pi(0)\otimes V_\Z)\\ 
        & \hv{nx}\otimes \mathbf{1} &\mapsto &\hv{n(x+y)}\otimes \hv{nx}\\
        &x\otimes \mathbf{1}&\mapsto & (x+y)\otimes \mathbf{1}+\mathbf{1}\otimes x\\
        &\mathbf{1}\otimes y &\mapsto &\tfrac{\sqrt{k+h^\vee_+}}{2}(x+y)\otimes \mathbf{1}+\tfrac{1}{\sqrt{k+h^\vee_+}}(y\otimes \mathbf{1}+\mathbf{1}\otimes x).
    \end{array}
\end{align*}
\end{lemma}

Let $\pi^{k+h^\vee_+}_{\epsilon_0|\epsilon_1,\ldots,\epsilon_n}$ be a Heisenberg vertex algebra generated by the fields $\epsilon_i(z)$ satisfying the OPEs
\begin{align}\label{intermediate heisenberg}
    \epsilon_i(z)\epsilon_j(w)\sim \frac{(-1)^{\delta_{i,0}}\delta_{i,j}(k+h^\vee_+)}{(z-w)^2}.
\end{align}
Note that the subalgebra generated by $\epsilon_i(z)$ ($i=1,\ldots, n$) is isomorphic to $\pi^{k+h^\vee_+}_{\h_+}$ since $\h_+=\bigoplus_{i=1}^n \C \epsilon_i$ (see Sect.~\ref{setting of our W-superalgebras}) and the one by $\epsilon_0(z)$ to $\pi^y$.
Hence \eqref{iHR for coset} implies the embedding of vertex superalgebras 
    \begin{align}\label{FFR for the Kazama-Suzuki coset}
        \Com(\pi^{J^+_\Delta},\subW\otimes V_\Z)
        &\hookrightarrow \Wpr\otimes V_\Z\otimes \pi^y \subset \pi^{k+h^\vee_+}_{\epsilon_0|\epsilon_1,\ldots,\epsilon_n}\otimes V_\Z.
    \end{align}
By applying the Feigin--Frenkel duality of the Virasoro vertex algebras (c.f. \cite[Chapter 15]{FBZ}), we have
$$\Ker\ Q_i^{\mathrm{pr}}=\Ker\ Q_i^-,\quad (i=1,\ldots,n)$$
through the following isomorphism of vertex superalgebras 
\begin{align}\label{isom for the large FFR}
\begin{array}{cll}
\pi^{k+h^\vee_+}_{\epsilon_0|\epsilon_1,\ldots,\epsilon_n}\otimes V_\Z &\xrightarrow{\simeq}& \pi^{\ell+h^\vee_-}_{\h_-}\otimes V_\Z\\
\epsilon_i\otimes \mathbf{1} &\mapsto & -\tfrac{1}{\ell+h^\vee_-}t_i \otimes \mathbf{1}\\
\mathbf{1}\otimes x &\mapsto & \mathbf{1}\otimes x\\
\mathbf{1}\otimes \hv{mx} &\mapsto & \mathbf{1}\otimes \hv{mx}.
\end{array}
\end{align}
This identification induces the Feigin-Frenkel duality of the principal $\W$-algebras
$$\Wpr\otimes V_\Z \otimes \pi^{k+h^\vee_+}_{\epsilon_0} \simeq \W^\ell(\symp_{2n})\otimes V_\Z \otimes \pi^{\ell+h^\vee_-}_{t_0},$$ 
as $Q_i^-$ ($i=1,\ldots,n$) corresponds to the screenings of  $\W^\ell(\symp_{2n})$.
It is straightforward to show that the remaining screening $S$ is identified with 
$$\int Y(\hv{-(\epsilon_0+\epsilon_1)}\otimes \hv{x},z)\ \dz$$
under \eqref{FFR for the Kazama-Suzuki coset} and then $Q_0^-$ through \eqref{isom for the large FFR}. 
Therefore, we obtain a relation between the Feigin--Semikhatov duality and the Feigin--Frenkel duality:

\begin{corollary}\label{FS vs FF}
For the levels $k,\ell$ satisfying \eqref{level condition for FS}, the diagram below commutes:
    \begin{center}
		\begin{tikzcd}[row sep=huge, column sep = huge]
			\Com(\pi^{J^+_\Delta},\Wsub\otimes V_\Z)
			\arrow[d, leftrightarrow, "\text{F.-S.}"']
			\arrow[r, hook, "\mu_{\mathrm{iHR}}"]&
			\Wpr\arrow[d, leftrightarrow, "\text{F.-F.}"]&
            \hspace{-1.9cm}\otimes  V_\Z\otimes\pi^{k+h^\vee_+}_{\epsilon_0}
			\\
			\Wsuper
			\arrow[r, hook, "\text{Wakimoto real.}"]&
			\W^\ell(\symp_{2n})&
            \hspace{-2.4cm}\otimes V_\Z\otimes\pi^{\ell+h^\vee_-}_{t_0}.
		\end{tikzcd}
	\end{center}
\end{corollary}

\section{Spectral flow twists}\label{sec:spectral_flow}
We collect some useful properties of spectral flow twists for later use.
\subsection{Spectral flow twist}
Let $V$ be a conformal vertex superalgebra with conformal vector $L$ and $\pi^H\subset V$ be a simple Heisenberg vertex subalgebra generated by the field $H(z)=\sum_{m\in \Z}H_mz^{-m-1}$ satisfying 
\begin{itemize}
\item
$H_0$ acts semisimply on $V$ with $\Z$-eigenvalues,
\item
$L(z)H(w)\sim \frac{H(w)}{(z-w)^2}+\frac{\partial_w H(w)}{(z-w)}$.
\end{itemize}
Given a (weak) $V$-module $(M,Y_M(\cdot,z))$, we can define another $V$-module structure on $M$ by modifying the structure map
$$V\times M\rightarrow M(\!(z)\!),\quad (a,m)\mapsto Y_M^{H}(a,z)m:=Y_M(\Delta_H(z)a,z)m$$
where $\Delta_H(z)$ is the Li's $\Delta$-operator \cite{Li2}
$$\Delta_{H}(z)=z^{H_0}\mathrm{exp}\left(-\sum_{m>0}\frac{H_m}{m}(-z)^{-m} \right).$$
The module endowed with the new action is called a \emph{spectral flow twist} of $M$ and denoted by $S_HM$.
Similarly, we define $S_{\theta H}M$ for $\theta\in\Z$. Note that for $\theta,\theta'\in\Z$, $S_{(\theta+\theta')H}M\simeq S_{\theta H}\circ S_{\theta' H}M$.
In the following, we simply write $S_{\theta}$ when the choice of $H$ is obvious.

For $V$-modules $M_i$ $(1\leq i\leq 3)$, let $I_V\binom{M_3}{M_1\ M_2}$ denote the (super-)space of (logarithmic) intertwining operators of type $\binom{M_3}{M_1\ M_2}$ \cite{HLZ}. By definition, it is a vector superspace spanned by parity-homogeneous linear maps of the form
\begin{align*}
\begin{array}{cccl}
\mathcal{Y}(\cdot,z)\colon & M_1 \otimes M_2 &\rightarrow & M_3\{z\}[\log z]\\
& m_1\otimes m_2 &\mapsto & \mathcal{Y}(m_1,z)m_2=\displaystyle{\sum_{\alpha=0}^{K} \sum_{n\in \C}} (m_1)^{\mathcal{Y}}_{(n;\alpha)}m_2 z^{-n-1}(\log z)^\alpha, 
\end{array}
\end{align*}
with $K\in \Z_{\geq0}$, subject to certain axioms, see \cite{HLZ}.
In particular, for the Heisenberg vertex algebra $\pi=\pi^x$
$I_{\pi}\binom{\pi_{\nu}}{\pi_{\lambda }\ \pi_{\mu}}$ is nonzero iff $\nu=\lambda+\mu$ and, in this case, it is spanned by the intertwining operator $I_\lambda(\cdot,z)$ uniquely determined by 
\begin{align}\label{Fock intertwiner}
I_\lambda(\hv{\lambda},z)= S_\lambda z^{(\lambda,\mu)}E^-(\lambda,z)E^+(\lambda,z),
\end{align} 
where $S_\lambda$ is the shift operator $\hv{\mu}\mapsto \hv{\lambda+\mu}$ and 
\begin{align*}
E^-(\lambda,z)=\mathrm{exp}\left(-\sum_{n<0}\frac{\lambda_{n}}{n}z^{-n}\right),\quad E^+(\lambda,z)=\mathrm{exp}\left(-\sum_{n>0}\frac{\lambda_{n}}{n}z^{-n}\right).
\end{align*}
Note that $\Delta_H(z)=z^{H_0} E^+(H,-z)$ in this notation.
A well-known result of Li \cite[Proposition 2.11]{Li2} asserts the following in this generality.

\begin{proposition}[\cite{Li2}]\label{spectral flow twist and IO}
We have functorial isomorphisms between spaces of intertwining operators
\begin{align*}
I_V\binom{S_H(M_3)}{S_H(M_1)\ M_2}\simeq I_V\binom{M_3}{ M_1\ M_2}\simeq I_V\binom{S_H(M_3)}{M_1\ S_H(M_2)}.
\end{align*}
In particular, if a category of $V$-modules is a braided tensor (super)category closed under the spectral flow twists $S_{\theta H}$ $(\theta\in\Z)$,
then we have isomorphisms of $V$-modules
$$S_H(M_1)\boxtimes M_2\simeq S_H(M_1\boxtimes M_2)\simeq M_1 \boxtimes S_H(M_2).$$
\end{proposition}

\proof
For the completeness of the paper, we give a proof of the proposition.
It is straightforward to show that we have linear maps 
\begin{align*}
&I_V\binom{M_3}{ M_1\ M_2}\simeq I_V\binom{S_H(M_3)}{M_1\ S_H(M_2)},\quad \mathcal{Y}(\cdot,z)\mapsto \mathcal{Y}(\Delta_{H}(z)\cdot,z),\\
&I_V\binom{M_3}{ M_1\ M_2}\simeq I_V\binom{S_H(M_3)}{S_H(M_1)\ M_2},\quad \mathcal{Y}(\cdot,z)\mapsto E^-(H ,z)\mathcal{Y}(\cdot,z)E^+(H,z)(-z)^{H_0},
\end{align*}
which are functorial with respect to the $M_i$'s. Since they have obvious inverses, they are indeed isomorphisms. Next, suppose $M_1$ and $M_2$ admit a fusion product $M_1\boxtimes M_2$, which represents the functor 
\begin{align}\label{reprentativity of IO}
    I_V\binom{\bullet}{ M_1\ M_2} \simeq \mathrm{Hom}_{V}(M_1\boxtimes M_2,\bullet).
\end{align}
Since we have functorial isomorphisms 
\begin{align*}
&I_V\binom{\bullet}{S_H (M_1)\ M_2}\simeq \mathrm{Hom}_{V}(M_1\boxtimes M_2,S_{-H}(\bullet)) \simeq \mathrm{Hom}_{V}(S_H(M_1\boxtimes M_2),\bullet),\\
&I_V\binom{\bullet}{ M_1\ S_H(M_2)}\simeq \mathrm{Hom}_{V}(M_1\boxtimes M_2,S_{-H}(\bullet)) \simeq \mathrm{Hom}_{V}(S_H(M_1\boxtimes M_2),\bullet),
\end{align*}
$S_H(M_1\boxtimes M_2)$ is the fusion product of $S_H(M_1)\boxtimes M_2$ and $M_1\boxtimes S_H(M_2)$.
\endproof

\subsection{Periods of spectral flow twists}
Let $V$ be a simple conformal vertex superalgebra which contains a Heisenberg vertex algebra $\pi\subset V$ generated by a non-degenerate Cartan subspace $\mathfrak{h}\subset V$ and set $\mathscr{C}=\Com(\pi,V)$.
Suppose that $\mathscr{C}\otimes \pi\hookrightarrow V$ is a conformal embedding and that $V$, as a $\pi$-module, decomposes into a direct sum of Fock modules
\begin{align}\label{decomposition into Focks}
V\simeq \bigoplus_{\lambda\in M}\mathscr{C}_\lambda \otimes \pi_\lambda.
\end{align}
Here $M\subset \mathfrak{h}^*$ is the subset consisting of the $\lambda$ satisfying $\mathscr{C}_\lambda\neq 0$, which forms a $\Z$-lattice as $V$ is simple. 
The dual lattice $M^*:=\{\lambda\in \h\mid (\lambda,M)\subset \Z\}$ is identified with the set of all the spectral flow twists $S_{h}$ ($h\in \h$) on $V$.
We identify $\h\simeq \h^*$ by the bilinear form $(h,h')=h_{(1)}h'$ for $h,h'\in \h^*$.

\begin{proposition}\label{period of spectral flow}
Let $\mathcal{P}\subset M^*$ denote the stabilizer, i.e. $S_h(V)\simeq V$ iff $h\in \mathcal{P}$.
Then for $\lambda-\mu\in \mathcal{P}$, $\mathscr{C}_\lambda\simeq \mathscr{C}_\mu$ holds as $\mathscr{C}$-modules. 
Moreover, \eqref{decomposition into Focks} induces an isomorphism of $\mathscr{C}\otimes V_{\mathcal{P}}$-modules
\begin{align}\label{decomposition into lattice}
V\simeq \bigoplus_{\lambda\in M/\mathcal{P}} \mathscr{C}_\lambda \otimes V_{\lambda+\mathcal{P}}
\end{align} 
and $V_{\mathcal{P}}$ is the maximum lattice vertex superalgebra inside $V$ obtained as a conformal extension of $\pi$.
\end{proposition}

\proof
The first assertion is immediate from the definition of ${\mathcal{P}}$ and the isomorphism of $\mathscr{C}\otimes \pi$-modules
$$S_h(V)\simeq \bigoplus_{\lambda\in M}\mathscr{C}_\lambda \otimes \pi_{\lambda+h}.$$
Therefore, \eqref{decomposition into lattice} holds as $\mathscr{C}\otimes \pi$-modules. 
By \cite[Proposition 5.6]{CGN}, $\mathscr{C}_\lambda$ is simple as a $\mathscr{C}$-module and the vertex superalgebra
$\Com(\mathscr{C},V)\simeq \bigoplus_{\lambda}\pi_\lambda$, where $\lambda$ satifies $\mathscr{C}_\lambda\simeq\mathscr{C}$, is simple.
Then, since $\mathscr{C}\otimes \pi\hookrightarrow V$ is a conformal embedding, $\Com(\mathscr{C},V)$ is a conformal extension of $\pi$. Therefore, it is a lattice vertex superalgebra, and so is the subalgebra $V_{\mathcal{P}}\subset \Com(\mathscr{C},V)$. This completes the proof of \eqref{decomposition into lattice}. 
Finally, suppose that we have a lattice $\mathcal{P}\subset Q\subset M$ such that $V_\mathcal{P}\subset V_Q\subset V$ and thus $V\simeq \bigoplus_{\lambda\in M/Q}\mathscr{C}_\lambda \otimes V_{\lambda+Q}$. 
By using isomorphisms $\mathscr{C}_{\lambda+\mu}\simeq \mathscr{C}_{\mu}$ as $\mathscr{C}$-modules for $\lambda\in M$, $\mu\in Q$, it is straightforward to show that the shift of the highest weight vectors $\hv{\lambda}\mapsto \hv{\lambda-\mu}$ ($\lambda \in M$) induces $V\simeq S_{\mu}(V)$ as $V$-modules. Thus $Q\subset \mathcal{P}$, i.e. $Q= \mathcal{P}$. This completes the proof.
\endproof

This proposition immediately implies the following result established by Mason \cite{Ma} under strong rationality in the purely even case: 

\begin{corollary}\label{extension to lattice}
If $V$ is a $\Z_{\geq0}$-graded simple lisse conformal vertex superalgebra of CFT type, then the Cartan subspace $\pi\subset V$ is extended to a lattice vertex superalgebra inside $V$.
\end{corollary}

\proof
It follows from the assumption that $V$ has only finitely many inequivalent simple modules. Since $V$ is simple, the spectral flow twists $S_h(V)$ ($h\in M^*$) are all simple $V$-modules. 
Therefore, the stabilizer $\mathcal{P}\subset M^*$ satisfies $|M^*/\mathcal{P}|<\infty$ and thus $\mathcal{P}$ forms a $\Z$-sublattice of $M^*$ of the same rank. By Proposition~\ref{period of spectral flow}, $V_{\mathcal{P}}\subset V$ is a lattice vertex superalgebra.
\endproof

\subsection{Affine vertex algebras}\label{sec:affine_VA}
Let $\g$ be a simple Lie algebra with Dynkin diagram $X_n$.   
Fix a triangular decomposition of $\g=\mathfrak{n}_+\oplus \mathfrak{h}\oplus \mathfrak{n}_-$, $\Delta$ the set of roots, $\Pi$ (resp.\ $\check{\Pi}$) the set of simple (co)roots.
Denote by $\varpi_i$ (resp.\ $\check{\varpi}_i$) the $i$-th fundamental (co)weight, $\rP$ and $\rQ$ (resp.\ $\check{\rP}$ and $\check{\rQ}$) the (co)weight and (co)root lattices, and $(\cdot,\cdot)$ the normalized invariant bilinear form. 

The affine Lie algebra $\widehat{\g}=\g[t^{\pm1}]\oplus \C K$ also has a triangular decomposition $\widehat{\g}=\widehat{\mathfrak{n}}_+\oplus \widehat{\mathfrak{h}}\oplus \widehat{\mathfrak{n}}_-$ with the Cartan subalgebra $\widehat{\h}=\h\oplus \C K$.
By extending $\widehat{\h}$ to $\widetilde{\h}=\h\oplus \C K \oplus \C D$ with $D=t\partial_t$, we obtain the affine Kac--Moody algebra $\widetilde{\g}$ whose Dynkin diagram is denoted by $X_n^{(1)}$. 
We set $\widetilde{\h}^*=\h^*\oplus \C \Lambda_0\oplus \C \delta$ to be the dual of $\widetilde{\h}$, with nontrivial pairing $\Lambda_0(K)=\delta(D)=1$ extending the natural pairing $\h^*\times \h\rightarrow \C$. In particular, the dual $\widehat{\h}^*$ of $\widehat{\h}$ is identified with $\h^*\oplus \C \Lambda_0\subset \widetilde{\h}^*$. 
Let $\Lambda_i$ denote the $i$-th affine fundamental weight and $a_i$ (resp.\ $\check{a}_i$) the $i$-th Kac (co)label.

Let $V^k(\g)$ be the universal affine vertex algebra at level $k$ and $L_k(\g)$ its unique simple quotient. 
The set of simple $L_{k}(\g)$-modules for $k\in \Z_{\geq 1}$ is given by the integrable representations $L_k(\Lambda)$ of highest weight $\Lambda\in \rP^k_+$ where
$$\rP^k_+=\left\{\Lambda=\sum \lambda_i \Lambda_i \left| \sum \check{a}_i \lambda_i=k\right.\right\}\subset \bigoplus_{i=0}^n\Z_{\geq0}\Lambda_i,$$
see \cite[Theorem 3.1.3]{FZ}.
By fixing the level $k$, we can identify the highest weight $\Lambda$ with its restriction $\lambda$ to $\h\subset\widehat{\h}$, which gives
\begin{align}\label{parametrization of highest weights}
\rP^k_+=\left\{\lambda=\sum \lambda_i \varpi_i \left| (\lambda,\theta)\leq k\right.\right\}\subset \rP_+:=\bigoplus_{i=1}^n\Z_{\geq0}\varpi_i\subset \rP,
\end{align}
where $\theta\in \Delta$ is the highest root.
By \cite{Fu}, the set of simple currents modules $\mathrm{Pic}(L_{k}(\g))$ for the classical type (i.e.\ $X=ABCD$) are given by 
$$\mathrm{Pic}(L_k(\g))=\{L_k(k \Lambda_i)\mid 0\leq i\leq n,\ a_i=1\}.$$
These modules are indeed realized by the spectral flow twists $L_k(k \Lambda_i)\simeq S_{\check{\varpi}_i}L_k(\g)$ \cite[Proposition 3.5]{Li2}.
By Proposition \ref{spectral flow twist and IO},
\begin{align}\label{spectral flow twist of affine weight}
L_k(k \Lambda_i)\boxtimes L_k(\Lambda)\simeq S_{\check{\varpi}_i}(L_k(\Lambda)),\quad (\Lambda \in \rP^k_+).
\end{align}
The permutations, denoted $s_i$ ($1\leq i\leq n$), of $\rP^k_+$ given by \eqref{spectral flow twist of affine weight} form a group which we also denote by $\mathrm{Pic}(L_k(\g))$.
The following result is well-known (see, for instance \cite{dFMS}).

\begin{proposition}\label{simple current action}
For $X=ABCD$, $\mathrm{Pic}(L_k(X_n))$ is isomorphic to $\check{\rP}/\check{\rQ}$ and the action on $\rP^k_+$ is given in Table \ref{tab: affine simple currents}. 
\end{proposition}

\renewcommand{\arraystretch}{1.2}
\begin{table}[h]
    \caption{Simple currents of $L_k(\g)$}
    \label{tab: affine simple currents}
    \centering
    \renewcommand{\arraystretch}{1.5}
\begin{tabular}{cccc}
\hline
$X_n$ & $\mathrm{Pic}(L_k(\g))$ & Group structure & $s_i$-action \\ \hline 
$A_n$ & $0\leq i\leq n$ & $\Z_{n+1}$ &$\begin{array}{c}\Lambda_j\overset{s_i}{\mapsto} \Lambda_{j+i}\\ (\text{mod}\  n+1)\end{array}$ \\ \hline 
$B_n$ & $i=0,1$ &  $\Z_2$ &  $ \Lambda_0\overset{s_1}{\leftrightarrow} \Lambda_1$ \\ \hline 
$C_n$ & $i=0,n $ & $\Z_2$ & $\Lambda_i \overset{s_n}{\leftrightarrow} \Lambda_{n-i}$  \\  \hline 
$\begin{array}{c}D_n \\ (n\colon \text{even}) \end{array}$& $i=0,1,n-1,n$ & $\Z_2\times \Z_2$  & $\begin{array}{c} (\Lambda_0,\Lambda_n) \overset{s_1}{\leftrightarrow} (\Lambda_{1}, \Lambda_{n-1}) \\ \Lambda_i \overset{s_n}{\leftrightarrow} \Lambda_{n-i}\\ s_{n-1}=s_1s_{n} \end{array} $ \\  \hline 
$\begin{array}{c}D_n \\ (n\colon \text{odd}) \end{array}$ & $i=0,1,n-1,n$ & $\Z_4$&  $\begin{array}{c}(\Lambda_0,\Lambda_1, \Lambda_{n-1}, \Lambda_n)\\ \overset{s_{n-1}}{\leftrightarrow} (\Lambda_n,\Lambda_{n-1}, \Lambda_{0}, \Lambda_1)   \\ \Lambda_i \overset{s_{n-1}}{\leftrightarrow} \Lambda_{n-i}\  (\text{o.w.}) \\ s_{1}=s_{n-1}^2,\ s_n=s_{n-1}^3 \end{array} $ \\  \hline
\end{tabular}
\renewcommand{\arraystretch}{1}
\end{table}

\subsection{Twisted BRST reduction}\label{sec:BRST}
We consider the compatibility of spectral flow twists and the BRST reduction for $\W$-algebras $\WA$ with $\Z$-grading.
Let us recall shortly the definition of $\WA$, that is the BRST reduction of $V^k(\g)$ indexed by nilpotent elements $f$ in $\g$ and an (even) good grading $\Gamma\colon \g=\bigoplus_{j\in \Z}\g_j$ for $f$. The grading $\Gamma$ gives a grading for the set of roots $\Delta=\sqcup_{j\in \Z}\Delta_j$ through the grading of their corresponding root vectors $e_{\alpha}\in \g$ ($\alpha\in \Delta$).
Then, given a $\VA$-module $M$, we associate the BRST complex \cite{KRW03}
$$C_f^\bullet(M)=M\otimes \bigwedge{}^{\frac{\infty}{2}+\bullet}(\g_{>0}).$$
Here $\bigwedge{}^{\frac{\infty}{2}+\bullet}(\g_{>0})$ denotes the charged fermion vertex superalgebra generated by the odd fields $\varphi_{\alpha}(z)$ and $\varphi_{\alpha}^*(z)$ ($\alpha\in \Delta_{>0}$) satisfying the OPEs
$$\varphi_\alpha(z)\varphi_\beta^*(w)\sim \delta_{\alpha,\beta}/(z-w),\quad \varphi_\alpha^*(z)\varphi_\beta^*(w)\sim \varphi_\alpha(z)\varphi_\beta(w)\sim 0.$$
The cohomological degree on $C_f^\bullet(M)$ is given by the $\Z$-grading on $\bigwedge{}^{\frac{\infty}{2}+\bullet}(\g_{>0})$ defined by 
$\mathrm{deg}(\varphi_\alpha^*(z))=1$ and $\mathrm{deg}(\varphi_\alpha(z))=-1$.
The differential $d$ is given by
\begin{align*}
&d=d_{\mathrm{st}}+d_\chi,\\
& d_{\mathrm{st}}=\int \dz \sum_{\alpha\in \Delta_{>0}}e_\alpha(z) \varphi_\alpha^*(z)-\half\sum_{\alpha,\beta,\gamma\in \Delta_{>0}}c_{\alpha,\beta}^\gamma :\varphi_\alpha(z)^*\varphi_\beta(z)^*\varphi_\gamma(z) :,\\
&d_{\chi}=\int \dz \sum_{\alpha\in \Delta_{>0}}\chi(e_\alpha)\varphi_\alpha^*(z),
\end{align*}
where $\chi(\cdot)=(f,\cdot)$.
Let $H_f^\bullet(M)$ denote the cohomology $H(C_f^\bullet(M))$.
In particular, the $\W$-algebra $\WA$ is defined as $\WA=H^0_f(\VA)$ and $H_f^\bullet(M)$ is naturally a $\WA$-module. 
By \cite[Theorem 2.4]{KRW03}, $\WA$ has a Heisenberg vertex subalgebra generated by the cohomology class of $\h^f$: 
$$\check{\mu}(z)+\sum_{\alpha\in \Delta_{>0}}(\check{\mu}, \alpha):\varphi_\alpha(z)\varphi_\alpha^*(z):,\quad (\check{\mu}\in \h^f),$$
which we also denote $\check{\mu}(z)$ by abuse of notation. Then we can consider spectral flow twists $S_{\check{\mu}}$ ($\check{\mu}\in \mathfrak{h}^f\cap \check{\rP}$) of $\WA$-modules.

Since $[d,\check{\mu}(z)]=0$, the spectral flow twist $S_{\check{\mu}}$ for the complex itself is well-defined
$$Y^{\check{\mu}}(\cdot,z)\colon C_f^\bullet(\VA)\otimes C_f^\bullet(M)\mapsto C_f^\bullet(M)(\!(z)\!),$$
and $d$ is twisted to a new differential $d_{\check{\mu}}$ given by
$$d_{\check{\mu}}=d_{st}+d_{\chi,\check{\mu}},\quad d_{\chi,\check{\mu}}=\int \dz \sum_{\alpha\in \Delta_{>0}}\chi(e_\alpha)\varphi_\alpha^*(z)z^{(\check{\mu},\alpha)}.$$
The corresponding cohomology, denoted by $H_{f,\check{\mu}}^\bullet(M)$, is called the twisted BRST reduction of $M$ \cite{ACF,AF} and is a module over $\W^k(\g,f)$. 
Therefore, we have obtained the first part of the following:

\begin{theorem}\label{compatibility of spectral flow twists}
For $\check{\mu}\in \h^f\cap \check{\rP}$, we have isomorphisms of $\W^k(\g,f)$-modules 
\begin{align*}
S_{\check{\mu}} H_f^\bullet(M)\simeq  H_{f,\check{\mu}}^\bullet(M)\simeq H_f^{\bullet -2 (\check{\mu},\rho_{>0})}(S_{\check{\mu}}M)
\end{align*}
where we set $\rho_{>0}:=\half\sum_{\alpha\in \Delta_{>0}}\alpha$.
\end{theorem}

\proof  
It remains to show the second isomorphism. Note that $ \bigwedge{}^{\frac{\infty}{2}+\bullet}(\g_{>0})$ is a tensor product of the bc-system $\bigwedge^{\bullet}$ generated by $\varphi_\alpha^*(z)$ and $\varphi_\alpha(z)$ for each $\alpha\in\Delta_{>0}$. 
By the boson--fermion correspondence, $\bigwedge^{\bullet}\simeq V_\Z$ as vertex superalgebra and $\bigwedge^{\bullet}$ has only one simple module  $\bigwedge^{\bullet}$, up to parity conjugation.
The Heisenberg field $x(z)=\NO{\varphi(z)\varphi^*(z)}$ satisfies $x(z)x(w)\sim 1/(z-w)^2$ and 
\begin{align*}
    x(z)\varphi(w)\sim \varphi(w)/(z-w),\quad  x(z)\varphi^*(w)\sim -\varphi^*(w)/(z-w).
\end{align*}
Hence, the spectral flow twists $S_{\theta x} \bigwedge^\bullet$ are all isomorphic to $\bigwedge^\bullet$ as $\bigwedge^\bullet$-modules but the parity and the $\Z$-grading are modified
\begin{align*}
\bigwedge{}^{\bullet+\theta} \simeq  S_{\theta x} \bigwedge{}^\bullet, \quad \mathbf{1}\mapsto 
\begin{cases}
\varphi_{0}^*\cdots \varphi_{-(\theta-1)}^* & (\theta>0),\\
\varphi_{-1}\cdots \varphi_{\theta} & (\theta<0).
\end{cases}
\end{align*}
This implies an isomorphism of graded $V^k(\g)\otimes  \bigwedge^\bullet(\g_{>0})$-modules
\begin{align}\label{graded isom}
S_{\check{\mu}}C_f^\bullet(M)\simeq C_f^{\bullet -2(\check{\mu},\rho_{>0})}(S_{\check{\mu}}M).
\end{align}
For $\alpha\in \Delta_{>0}$ such that $\chi(e_\alpha)\neq 0$, we have 
$$(\check{\mu},\alpha)= \frac{(f|[\check{\mu},e_\alpha])}{(f|e_\alpha)}= \frac{([f,\check{\mu}]|e_\alpha)}{(f|e_\alpha)}=0$$
since $\check{\mu}\in \h^f$.
Then it follows that $d_{\chi,\check{\mu}}=d_{\chi}$ and thus $d_{\check{\mu}}=d$. Therefore, \eqref{graded isom} induces the isomorphism $H_{f,\check{\mu}}^\bullet(M)\simeq H_f^{\bullet -2 (\check{\mu},\rho_{>0})}(S_{\check{\mu}}M)$ as $\W^k(\g,f)$-modules.
\endproof

\begin{remark}
The cohomological shift $-2(\check{\mu}|\rho_{>0})$ is equal to $\ell^{\frac{\infty}{2}}_\g(t_{\check{\mu}})-\ell^{\frac{\infty}{2}}_{\g_0}(t_{\check{\mu}})$ in terms of the semi-infinite lengths of extended affine Weyl group \cite{FF90}.
\end{remark}

\section{Block-wise equivalence of Kazhdan--Lusztig categories}\label{Sec: equivalnce of categories}

\subsection{Convolution}
Following \cite{CLNS}, we rephrase Theorem~\ref{Coset theorem} into the language of \emph{relative semi-infinite cohomology} \cite{Feigin,FGZ}. 
This approach is motivated by the module decompositions \eqref{eq: decomposition} of $\subW$ and $\sprW$. By Corollary~\ref{isom for the multiplicities}, these decompositions are the same over the common coset subalgebra $\mathscr{C}^k_{D^+}\simeq \mathscr{C}^\ell_{D^-}$, but not over the rank one Heisenberg vertex algebras as the generating Heisenberg fields $J^+(z)$ and $J^-(z)$ have different norms.
To exchange the Fock modules $\pi^{J^+}_m$ with $\pi^{J^-}_m$ in \eqref{eq: decomposition}, we use the following property of the relative semi-infinite cohomology functor \cite{Feigin,FGZ}
\begin{align}\label{rel coh functor}
\mathrm{H}_{\mathrm{rel}}^{\frac{\infty}{2}+m}(\widehat{\mathfrak{gl}}_1,\mathfrak{gl}_1; \pi^{A_*}_p\otimes \pi^{A}_q)\simeq \delta_{m,0}\delta_{p+q,0}\C[\hv{p} \otimes \hv{q}],
\end{align}
(see \cite[Proposition 4]{CGNS} for a short proof). Here $A(z)$ is a Heisenberg field whose norm is non-zero and $A_*(z)=\ssqrt{-1}A(z)$. 
Introduce the following modules over Heisenberg vertex algebras: 
\begin{align*}
\bigoplus_{m\in \Z}\pi^{J^+_*}_{-m}\otimes \pi^{J^-}_m,\quad \bigoplus_{m\in \Z}\pi^{J^-_*}_{-m}\otimes \pi^{J^+}_m.
\end{align*} 
We find that they actually have a structure of vertex superalgebras given by
\begin{align}\label{decomposition of kernels}
\begin{array}{cccccc}
V_{\Z}\otimes \pi^y &\xrightarrow{\simeq}& \displaystyle{\bigoplus_{m\in \Z}}\pi^{J^+_*}_{-m}\otimes \pi^{J^-}_m,& V_{\ssqrt{-1}\Z}\otimes \pi^x &\xrightarrow{\simeq}& \displaystyle{\bigoplus_{m\in \Z}}\pi^{J^-_*}_{-m}\otimes \pi^{J^+}_m\\
x\otimes 1&\mapsto &\tfrac{1}{\varepsilon^2-1}(J^+_*+\varepsilon^2 J^-)&y\otimes \mathbf{1}&\mapsto &\tfrac{1}{\varepsilon^2-1}(\varepsilon^2J^-_*+J^+) \\
1\otimes y&\mapsto &\frac{\varepsilon}{\varepsilon^2-1}(J^+_*+J^-)& \mathbf{1}\otimes x&\mapsto&\frac{\varepsilon}{\varepsilon^2-1}(J^-_*+J^+) \\
\hv{mx}\otimes \mathbf{1}&\mapsto &\hv{-m} \otimes \hv{m}&\hv{my}\otimes \mathbf{1}&\mapsto &\hv{-m} \otimes \hv{m},\\
\end{array}
\end{align} 
where
$$\varepsilon:=(k+h^\vee_+)^{1/2}=(2(\ell+h^\vee_-))^{-1/2}.$$
We set 
\begin{align}\label{kernels}
K_{D^+}:=V_\Z\otimes \pi^y,\quad K_{D^-}:=V_{\ssqrt{-1}\Z}\otimes \pi^x.
\end{align}
We briefly write $\mathrm{H}_{\mathrm{rel}}^{m}(\mathfrak{gl}_1,\bullet)$ for $\mathrm{H}_{\mathrm{rel}}^{\frac{\infty}{2}+m}(\widehat{\mathfrak{gl}}_1,\mathfrak{gl}_1;\bullet)$.
The following theorem is a refinement of \cite[Theorem 4.4]{CLNS} stated for irrational levels.

\begin{proposition}\label{reconstruction}
For $k,\ell\in \C$ satisfying \eqref{level condition for FS}, there exist isomorphisms of vertex superalgebras
\begin{align*}
&\Wsub\simeq \mathrm{H}_{\mathrm{rel}}^{0}(\mathfrak{gl}_1, \Wsuper\otimes K_{D^-} ),\\
&\Wsuper\simeq \mathrm{H}_{\mathrm{rel}}^{0}(\mathfrak{gl}_1, \Wsub\otimes K_{D^+} ).
\end{align*}
\end{proposition}

\proof
The assertion is obtained from Theorem \ref{Coset theorem} by word-by-word translation of the proof of \cite[Proposition 5]{CGNS} for the similar statement in type $A$.
\endproof

\subsection{Equivalence of categories} \label{subsection_Equivalence of categories}
We use Proposition \ref{reconstruction} to relate the modules of $\W$-superalgebras, following the argument \cite{CGNS} in type $A$. 
We introduce the full subcategories
$$\KL_{D^+}^k,\quad \KL_{D^-}^\ell$$
of the categories of weight (super)modules over $\subW$ and $\sprW$, which by definition decompose into direct sums of Fock modules over $\pi^{J^+}$ and $\pi^{J^-}$. 
Since the set of eigenvalues of $J^+_0$ on $\subW$ (resp.\ of $J^-_0$ on $\sprW$) is $\Z$, they decompose into 
$$\KL_{D^+}^k=\bigoplus_{[\lambda\varepsilon]\in \C/\Z}\KL_{D^+}^{k,[\lambda\varepsilon]},\quad \KL_{D^-}^\ell=\bigoplus_{[\lambda\varepsilon]\in \C/\Z}\KL_{D^-}^{\ell,[\lambda\varepsilon]}$$
where $\KL_{D^+}^{k,[\lambda\varepsilon]}$ (resp.\ $\KL_{D^-}^{\ell,[\lambda\varepsilon]}$) is the full subcategory consisting of modules whose $J^+_0$-eigenvalues (resp.\ $J^-_0$-eigenvalues) modulo $\Z$ lie in $[\lambda\varepsilon]$.
By using the modules 
$$K_{D^+}^\lambda:=V_\Z\otimes \pi^y_{\lambda y},\quad K_{D^-}^\lambda:=V_{\ssqrt{-1}\Z}\otimes \pi^x_{\lambda x},\quad (\lambda\in \C),$$
over \eqref{kernels}, we assign
\begin{align}\label{assignment}
& M\mapsto \mathrm{H}_{\mathrm{rel}}^{0}(\mathfrak{gl}_1, M \otimes K_{D^-}^\lambda ),\quad N\mapsto \mathrm{H}_{\mathrm{rel}}^{0}(\mathfrak{gl}_1, N \otimes K_{D^+}^\lambda ),
\end{align}
to a $\sprW$-module $M$ (resp.\ $\subW$-module $N$). 
Proposition \ref{reconstruction} implies that the cohomology is naturally a module over $\subW$ (resp.\ $\sprW$).
Let $\mathrm{H}_\lambda^\pm$ denote the functors obtained by the assignments \eqref{assignment}
$$\mathrm{H}_\lambda^+\colon \KL_{D^+}^k\ \rightleftarrows \ \KL_{D^-}^\ell \colon \mathrm{H}_\lambda^-.$$
They are exact thanks to \eqref{rel coh functor}. 
It follows from \eqref{decomposition of kernels} that they can be restricted block-wise to 
\begin{align}
\mathrm{H}_\lambda^+\colon \KL_{D^+}^{k,[\lambda\varepsilon]}\ \rightleftarrows \ \KL_{D^-}^{\ell,[\lambda/\varepsilon]} \colon \mathrm{H}_\lambda^-.
\end{align}
Observe that the domain of $\mathrm{H}_\lambda^+$ (resp.\ $\mathrm{H}_\lambda^-$) does not change under the shift $\lambda \mapsto \lambda_\theta^+:=\lambda-\theta \varepsilon^{-1}$ (resp.\ $\lambda \mapsto \lambda_\theta^-:=\lambda-\theta \varepsilon$), but does the target. 

\begin{theorem}\label{categorical equivalence}
For $\theta\in \Z$, the functors 
		\begin{align*}
			\mathrm{H}_{\lambda_\theta^+}^+\colon \KL_{D^+}^{k,[\lambda\varepsilon]}\ \rightleftarrows \ \KL_{D^-}^{\ell,[\lambda_\theta^+/\varepsilon]} \colon \mathrm{H}_{\lambda_\theta^+}^-
		\end{align*}
		are quasi-inverse to each other and give equivalences of categories.
\end{theorem}

\proof
The assertion is obtained by word-by-word translation of the proof of \cite[Theorems 7 and 8]{CGNS} in type $A$. 
\endproof

The functors $\mathrm{H}_{\lambda_\theta^+}^+$ are actually related to each other through spectral flow twists.
\begin{theorem}\label{equivalence of categories}
 For $\theta_1,\theta_2\in\Z$, there are natural isomorphisms between functors 
\begin{align}\label{spectral flow for KL}
\mathrm{H}_{\lambda_{\theta_1}^+}^+\simeq\mathrm{H}_{{\lambda'}_{\theta_1}^+}^+\circ S_{\theta_2 J^+},\quad 
S_{(\theta_2-\theta_1)J^-}\circ \mathrm{H}_{\lambda_{\theta_1}^+}^+\simeq \mathrm{H}_{\lambda_{\theta_2}^+}^+,
\end{align}
where $\lambda'=\lambda+\varepsilon\theta_2$.
In other words, the diagrams
	\begin{center}
		\begin{tikzcd}[row sep=huge, column sep = huge]
			\KL_{D^+}^{k,[\lambda\varepsilon]}
			\arrow[d,"S_{\theta_2J^+}"']
			\arrow[r,"\mathrm{H}_{\lambda_{\theta_1}^+}^+"]&
			\KL_{D^-}^{\ell,[\lambda_{\theta_1}^+/\varepsilon]}\\
			\KL_{D^+}^{k,[\lambda'\varepsilon]}
			\arrow[ru, "\mathrm{H}_{{\lambda'}_{\theta_1}^+}^+"']&
		\end{tikzcd}
		\begin{tikzcd}[row sep=huge, column sep = huge]
		& \KL_{D^-}^{\ell,[\lambda_{\theta_1}^+/\varepsilon]}
		\arrow[d,"S_{(\theta_2-\theta_1)J^-}"]\\
		\KL_{D^+}^{k,[\lambda\varepsilon]}
		\arrow[ru, "{\mathrm{H}_{\lambda_{\theta_1}^+}^+}"]
		\arrow[r, "\mathrm{H}_{\lambda_{\theta_2}^+}^+"']& \KL_{D^-}^{\ell,[\lambda_{\theta_2}^+/\varepsilon]}
	\end{tikzcd}
\end{center}
		commute up to natural isomorphisms.
\end{theorem}

\proof
We first show $S_{(\theta_2-\theta_1)J^-}\circ \mathrm{H}_{\lambda_{\theta_1}^+}^+\simeq \mathrm{H}_{\lambda_{\theta_2}^+}^+$. 
It suffices to show the case $(\theta_1,\theta_2)=(0,\theta)$. 
As $J^-=x-\frac{1}{\varepsilon}y$ by \eqref{decomposition of kernels}, we have 
$$S_{J^-}K_{D^+}^\lambda= (S_x\otimes S_{-\frac{1}{\varepsilon}y})K_{D^+}^\lambda\simeq K_{D^+}^{\lambda-\frac{1}{\varepsilon}}$$
as $V_\Z\otimes \pi^y$-modules.
Then, given a module $M$ in $\KL_{D^+}^{k,[\lambda\varepsilon]}$, we have
\begin{align*}
S_{\theta J^-}\circ\mathrm{H}_{\lambda^+}^+(M)
&=S_{\theta J^-}\circ\mathrm{H}_{\mathrm{rel}}^{0}(\mathfrak{gl}_1, M \otimes K_{D^+}^\lambda)
\simeq\mathrm{H}_{\mathrm{rel}}^{0}(\mathfrak{gl}_1, M \otimes  K_{D^+}^{\lambda-\frac{\theta}{\varepsilon}})\simeq\mathrm{H}_{\lambda_{\theta}^+}^+(M).
\end{align*}
Thus we obtain an isomorphism $S_{\theta J^-}\circ\mathrm{H}_{\lambda^+}^+(M)\simeq \mathrm{H}_{\lambda_{\theta}^+}^+(M)$, which is functorial for $M$, i.e. $S_{\theta J^-}\circ\mathrm{H}_{\lambda^+}^+\simeq \mathrm{H}_{\lambda_{\theta}^+}^+$.
Next, we show $\mathrm{H}_{\lambda_{\theta_1}^+}^+\simeq\mathrm{H}_{{\lambda'}_{(\theta_1-\theta_2)}^+}^+\circ S_{\theta_2 J^+}$.
For $M$ in $\KL_{D^+}^{k,[\lambda\varepsilon]}$, one can show similarly a functorial isomorphism of $\sprW$-modules
\begin{align*}
    \mathrm{H}_{{\lambda'}_{\theta_1}^+}^+\circ S_{\theta_2 J^+}(M)
    &=\mathrm{H}_{\mathrm{rel}}^{0}(\mathfrak{gl}_1, S_{\theta_2 J^+}(M) \otimes  K_{D^+}^{\lambda'-\frac{\theta_1}{\varepsilon}})\\
    &\simeq \mathrm{H}_{\mathrm{rel}}^{0}(\mathfrak{gl}_1, S_{\theta_2(J^+-x+\varepsilon y)}(M) \otimes  K_{D^+}^{\lambda-\frac{\theta_1}{\varepsilon}})\simeq \mathrm{H}_{\lambda_{\theta_1}^+}^+\circ S_{\theta_2(J^++J^+_*)}(M)
\end{align*}
as $J^+-x+\varepsilon y=J^++J^+_*$. 
Recall that $\sprW$ acts on the right-hand side through the isomorphism in Proposition \ref{reconstruction}. Then cohomology classes $[x]\in\mathrm{H}_{\mathrm{rel}}^{0}(\mathfrak{gl}_1, \subW\otimes K_{D^+})$ act as 
$Y([\Delta\left(\theta_2(J^++J^+_*),z\right)x],z)$, which coincides with $Y([x],z)$. 
It follows that $\mathrm{H}_{{\lambda'}_{\theta_1}^+}^+\circ S_{\theta_2(J^++J^+_*)}(M)\simeq \mathrm{H}_{\lambda_{\theta_1}^+}^+(M)$
canonically as $\sprW$-modules. This completes the proof. 
\endproof

\begin{corollary}\label{cor of spectral flowwise equivalence}
For $\theta_i\in\Z$ ($1\leq i\leq 4$) and $\lambda,\lambda'\in \C$ with relations 
$$\theta_4=\theta_2+\theta_3,\quad \lambda'=\lambda+\varepsilon \theta_1,$$ 
the diagram 
	\begin{center}
		\begin{tikzcd}[row sep=large, column sep = huge]
			\KL_{D^+}^{k,[\lambda\varepsilon]}
			\arrow[d,"S_{\theta_1 J^+}"']
			\arrow[r,"\mathrm{H}_{\lambda_{\theta_3}^+}^+"]&
			\KL_{D^-}^{\ell,[\lambda_{\theta_3}^+/\varepsilon]}
			\arrow[d,"S_{\theta_2 J^-}"]\\
			\KL_{D^+}^{k,[\lambda'\varepsilon]}
			\arrow[r, "\mathrm{H}_{{\lambda'}_{\theta_4}^+}^+"]&
			\KL_{D^-}^{\ell,[{\lambda'}_{\theta_4}^+/\varepsilon]}
		\end{tikzcd}
	\end{center}
commutes up to a natural isomorphism and all the arrows give equivalences of categories. 
\end{corollary}

\begin{remark}
One can show the analogues of Theorem \ref{equivalence of categories} and Corollary \ref{cor of spectral flowwise equivalence} for type A, i.e. the original Feigin--Semikhatov duality between the $\W$-(super)algebras
\begin{align*}
    \W^k(\sll_{n},f_{\mathrm{\mathrm{sub}}}),\quad \W^\ell(\sll_{1|n},f_{\mathrm{prin}})
\end{align*}
under the relation $(k+n)(\ell+n-1)=1$. 
Then Corollary \ref{cor of spectral flowwise equivalence} generalizes the case $n=2$ between $V^k(\sll_2)$ and the $\mathcal{N}=2$ superconformal algebra \cite{FST}.
\end{remark}

\subsection{Twists of braided tensor categories}
The functors $\mathrm{H}_\lambda^{\pm}$ ($\lambda\in \C$) are enhanced by functorial homomorphisms for the (super)spaces of intertwining operators.
Given the $\subW$-modules $M_i$ in $\KL_{D^+}^k$ ($1\leq i\leq 3$) and an intertwining operator $\mathcal{Y}(\cdot,z)$ of type $\binom{M_3}{M_1\ M_2}$, we have an intertwining operator
\begin{align}\label{extending IO}
&\mathcal{Y}(\cdot,z)\otimes Y_{V_\Z}(\cdot,z)\otimes I_{\lambda y}(\cdot,z)\colon  (M_1\otimes K_{D^+}^\lambda) \otimes (M_2\otimes K_{D^+}^\mu)
\longrightarrow  (M_3\otimes K_{D^+}^{\lambda+\mu})\{z\}[\log z],
\end{align}
see \eqref{Fock intertwiner} for $I_{\lambda y}(\cdot,z)$. Here $Y_{V_\Z}(\cdot,z)\colon V_\Z\otimes V_\Z\rightarrow V_\Z(\!(z)\!)$ is the structure map.
Then, it is straightforward to show that \eqref{extending IO} gives rise to an intertwining operator for the relative semi-infinite cohomology
\begin{align*}
\mathcal{H}_\lambda^+(\mathcal{Y})(\cdot,z)\colon \mathrm{H}_\lambda^{+}(M_1)\otimes \mathrm{H}_\mu^{+}(M_2)\rightarrow \mathrm{H}_{\lambda+\mu}^{+}(M_3)\{z\}[\log z]
\end{align*}
for the modules over $\mathrm{H}_0^+(\subW)\simeq\sprW$ by Proposition \ref{reconstruction}, see \cite [Sect.~6.3]{CGNS} for a detailed proof in type $A$. 
Therefore, we obtain a functorial homomorphism of spaces of intertwining operators
\begin{align*}
\mathcal{H}_\lambda^+(\cdot)\colon I_{\subW}\binom{M_3}{M_1\ M_2}\longrightarrow I_{\sprW}\binom{\mathrm{H}_{\lambda+\mu}^+(M_3)}{\mathrm{H}_\lambda^+(M_1)\ \mathrm{H}_\mu^+(M_2)}.
\end{align*}
Similarly, replacing $I_{\lambda y}(\cdot,z)$ with $I_{\lambda x}(\cdot,z)$ in \eqref{extending IO}, we obtain a functorial homomorphism of spaces of intertwining operators in the opposite direction for the $\sprW$-modules $N_i$ in $\KL_{D^-}^\ell$ ($1\leq i\leq3$)
\begin{align*}
\mathcal{H}_\lambda^-(\cdot)\colon I_{\sprW}\binom{N_3}{N_1\ N_2}\longrightarrow I_{\subW}\binom{\mathrm{H}_{\lambda+\mu}^-(N_3)}{\mathrm{H}_\lambda^-(N_1)\ \mathrm{H}_\mu^-(N_2)}.
\end{align*}

\begin{proposition}\label{correspondence of IO}
The functorial homomorphisms $\mathcal{H}_\lambda^\pm(\cdot)$ are linear isomorphisms.
\end{proposition}

\proof
The assertion is obtained by word-by-word translation of the proof of \cite[Theorem 9]{CGNS} for type $A$. 
\endproof

Suppose that a full subcategory 
$$\mathscr{C}_{D^+}^k=\bigoplus\mathscr{C}_{D^+}^{k,[\lambda \varepsilon]}\subset \bigoplus\KL_{D^+}^{k,[\lambda \varepsilon]}=\KL_{D^+}^k$$
is closed under spectral flow twists and has a structure of braided tensor category (BTC for short) in the sense of Huang--Lepowsky--Zhang \cite{HLZ}. 
In particular, we have a fusion product $M_1\boxtimes M_2$ for objects in $\mathscr{C}_{D^+}^k$, which represents the functor assigning the space of intertwining operators, see \eqref{reprentativity of IO}.
By the compatibility with the $J^+_0$-eigenvalues, the fusion product restricts to 
\begin{align*}
    \boxtimes \colon \mathscr{C}_{D^+}^{k,[\lambda\varepsilon]} \times \mathscr{C}_{D^+}^{k,[\mu \varepsilon]}\rightarrow \mathscr{C}_{D^+}^{k,[\nu\varepsilon]},\quad (\nu=\lambda+\mu).
\end{align*}
Let
$$\mathscr{C}_{D^-}^\ell=\bigoplus\mathscr{C}_{D^-}^{\ell,[\lambda /\varepsilon]}\subset \bigoplus\KL_{D^-}^{\ell,[\lambda /\varepsilon]}=\KL_{D^-}^\ell$$
denote the full subcategory of $\KL_{D^-}^\ell$ obtained from $\mathscr{C}_{D^+}^k$ by the equivalence in Theorem \ref{categorical equivalence}.
By Proposition \ref{correspondence of IO}, Theorem \ref{categorical equivalence} implies that the $\sprW$-modules $\mathrm{H}^+_{\lambda_{\theta_1}^+}(M_1)$ and $\mathrm{H}^+_{\mu_{\theta_2}^+}(M_2)$ admit a fusion product which is
\begin{align}\label{fusion product}
    \mathrm{H}^+_{\lambda_{\theta_1}^+}(M_1)\boxtimes \mathrm{H}^+_{\mu_{\theta_2}^+}(M_2):=\mathrm{H}^+_{\nu_{\theta_3}^+}(M_1\boxtimes M_2),\quad (\theta_3=\theta_1+\theta_2)
\end{align}
equipped with the canonical intertwining operator 
\begin{align*}
\mathcal{H}_{\lambda_{\theta_{1}}}^+(\eta)(\cdot,z)\colon  \mathrm{H}^+_{\lambda_{\theta_1}^+}(M_1)\otimes \mathrm{H}^+_{\mu_{\theta_2}^+}(M_2)\rightarrow \mathrm{H}^+_{\nu_{\theta_3}^+}(M_1\boxtimes M_2)\{z\}[\log z]
\end{align*}
induced by $\eta(\cdot,z)\colon M_1\otimes M_2\rightarrow (M_1\boxtimes M_2)\{z\}[\log z]$ on the $\subW$-side. Therefore, we have a bifunctor 
\begin{align}\label{fusion bifunctor for super W}
    \boxtimes \colon \mathscr{C}_{D^-}^{\ell,[\lambda_{\theta_1}/\varepsilon]}\times \mathscr{C}_{D^-}^{\ell,[\mu_{\theta_2}/\varepsilon]} \rightarrow 
			\mathscr{C}_{D^-}^{\ell,[\nu_{\theta_3}/\varepsilon]}.
\end{align}
We consider the braiding and associator for \eqref{fusion bifunctor for super W}. 
Recall that the braiding and associator
\begin{align*}
\mathcal{M}\colon M_1\boxtimes M_2\xrightarrow{\simeq}M_2\boxtimes M_1,\quad \mathcal{A}\colon M_1\boxtimes (M_2\boxtimes M_3)\xrightarrow{\simeq} (M_1\boxtimes M_2)\boxtimes M_3
\end{align*}
on the $\mathscr{C}_{D^+}^{k}$-side are constructed through the so-called $P(z)$-tensor product $\boxtimes_{z}:=\boxtimes_{P(z)}$ with $z\in \C^\times$ \cite{HL,HLZ}. Then the values of the canonical intertwining operator $\eta(a,z)b$ are regarded as a multi-valued function in $z$ with coefficient $M_1\boxtimes M_2$ (more precisely, its algebraic completion). The isomorphisms are derived from the $P(z)$-tensor product and parallel transport isomorphisms $T_{z,w}$ depending on paths connecting $z,w\in \C^\times$ up to homotopy as depicted in Fig.~\ref{fig:Braiding an associators} below.
\begin{figure}[h]
	\caption{}\label{fig:Braiding an associators}
 \begin{center}
		\begin{tikzcd}[row sep=large, column sep = small]
			M_1\boxtimes M_2 
			\arrow[d,"T_{1,z}"']
			\arrow[r,"\mathcal{M}"]&
			M_2 \boxtimes M_1 
                \arrow[d,"T_{1,\text{-}z}"] \\
			M_1\underset{z}{\boxtimes} M_2
			\arrow[r,"\simeq"]&
                M_2\underset{-z}{\boxtimes} M_1,
		\end{tikzcd}
		\begin{tikzcd}[row sep=huge]
              M_1\boxtimes (M_2\boxtimes M_3)
              \arrow[r,"\mathcal{A}"]
               \arrow[d,"T_{1,z}\boxtimes T_{1,w}"']
		& (M_1\boxtimes M_2)\boxtimes M_3
               \arrow[d,"T_{1,z\text{-}w}\boxtimes T_{w}"]\\
                M_1\underset{z}{\boxtimes} (M_2 \underset{w}{\boxtimes} M_3)
		\arrow[r, "\simeq"]& 
           (M_1\underset{z-w}{\boxtimes} M_2)\underset{w}{\boxtimes} M_3.
	\end{tikzcd}
\end{center}
\end{figure}
Here $M_1\boxtimes M_2$ is regarded as $M_1\boxtimes_{P(1)} M_2$, see also \cite[Sect.~3.3]{CKM1}. 
We note that the horizontal isomorphisms are given by the skew-symmetry of intertwining operators and the associativity isomorphism of $P(z)$-intertwining map \cite[Theorem 10.3]{HLZ}.

The category $\mathscr{H}_\R^x$ (resp.\ $\mathscr{H}_\R^y$) of $\pi^x$-modules (resp.\ $\pi^y$-modules) which are direct sums of $\pi^x_{\lambda x}$ (resp.\ $\pi^y_{\lambda y}$) for $\lambda\in \R$, forms a BTC, see e.g. \cite[Sect.~2.2]{CKLR}. They are braided tensor equivalent to the categories of graded vector spaces 
\begin{align*}
   &\mathscr{H}_\R^x \xrightarrow{\simeq}\mathrm{Vect}_{\mathbb{R}}\ (\pi^x_{\lambda x}\mapsto \C_\lambda),\quad \mathscr{H}_\R^y \xrightarrow{\simeq}\mathrm{Vect}_{\ssqrt{-1}\mathbb{R}}\ (\pi^y_{\lambda y}\mapsto \C_{\lambda\ssqrt{-1}}),
\end{align*}
where $\C_\lambda$'s are copies of $\C$ of grading $\lambda$'s satisfying the braiding isomorphisms 
\begin{align}\label{braiding for the Heisenberg}
    &e^{\pi\ssqrt{-1}\lambda\mu}\colon \C_{\lambda+\mu}=\C_\lambda\otimes \C_\mu\xrightarrow{\simeq}\C_\mu\otimes \C_{\lambda}=\C_{\mu+\lambda}.
\end{align}
The correspondence $\mathscr{H}_\R^x\rightarrow \mathscr{H}_\R^y$ ($\pi_{\lambda x}^x\mapsto \pi_{\lambda y}^y$) induces a braided reverse equivalence.
Let us replace the $M_i$'s in the above diagram with the complexes associated to $M_i\otimes K_{D^+}^{\bullet_{\theta_i}}$ ($\bullet=\lambda,\mu,\nu$). 

Taking the cohomology and using \eqref{fusion product}, we obtain the braiding and associator for \eqref{fusion bifunctor for super W}. 
For example, the braiding is depicted in Fig.~\ref{fig:Braiding an associators for super side} below. 
There, $[(-1)^\bullet e^{\pi\ssqrt{-1}\lambda_{\theta_1}^+\mu_{\theta_2}^+}\mathcal{M}]$ is the linear operator on the cohomology induced by the original associator $\mathcal{M}$ dressed by the associator $(-1)^\bullet$, which is the Koszul sign rule of $V_\Z$, and $e^{\pi\ssqrt{-1}\lambda_{\theta_1}^+\mu_{\theta_2}^+}$, which corresponds to \eqref{braiding for the Heisenberg}; $T_{1,z}\otimes T_{1,\pm z}^K$ are the parallel transport isomorphisms for the cohomology complexes $M_i\otimes K_{D^+}^{\bullet_{\theta_i}}$. 
As we work on the category of supermodules, we omit the Koszul sign rule $(-1)^\bullet$ as usual.
\begin{figure}[h]
	\caption{}\label{fig:Braiding an associators for super side}
 \begin{center}
		\begin{tikzcd}[column sep = large]
		\mathrm{H}_{\lambda_{\theta_1}^+}(M_1)\boxtimes \mathrm{H}_{\mu_{\theta_2}^+}(M_2) \arrow[d,equal] \arrow[rr,"\mathcal{M}"] \arrow[ddd,bend right=70,"T_{1,z}"'] & &\mathrm{H}_{\mu_{\theta_2}^+}(M_2)\boxtimes \mathrm{H}_{\lambda_{\theta_1}^+}(M_1)  \arrow[d,equal] \arrow[ddd,bend left=70,"T_{1,-z}"]\\
                \mathrm{H}_{\nu_{\theta_3}^+}(M_1\boxtimes M_2) 
			\arrow[d,"{[}T_{1,z}\otimes T_{1,z}^K{]}"'] \arrow[rr,"{[}(-1)^\bullet e^{-\pi\ssqrt{-1}\lambda_{\theta_1}^+\mu_{\theta_2}^+}\mathcal{M}{]}"]&&
			\mathrm{H}_{\nu_{\theta_3}^+}(M_2 \boxtimes M_1) 
                \arrow[d,"{[}T_{1,\text{-}z}\otimes T_{1,-z}^K{]}"]    \\
		    \mathrm{H}_{\nu_{\theta_3}^+}(M_1\underset{z}{\boxtimes} M_2)
			\arrow[rr,"\simeq"]&&
                \mathrm{H}_{\nu_{\theta_3}^+}(M_2\underset{-z}{\boxtimes} M_1)\\
                \mathrm{H}_{\lambda_{\theta_1}^+}(M_1)\underset{z}{\boxtimes} \mathrm{H}_{\mu_{\theta_2}^+}(M_2) \arrow[u,equal] \arrow[rr,"\simeq"] && \mathrm{H}_{\mu_{\theta_2}^+}(M_2)\underset{-z}{\boxtimes} \mathrm{H}_{\lambda_{\theta_1}^+}(M_1)  \arrow[u,equal]
		\end{tikzcd}
\end{center}
\end{figure}
To summarize, we prove the following.
\begin{corollary}\label{BTC bijection}
Let $\mathscr{C}_{D^+}^k\subset \KL_{D^+}^k$, $\mathscr{C}_{D^-}^\ell\subset \KL_{D^-}^\ell$ be full subcategories block-wisely equivalent as abelian categories
$$\mathrm{H}_{\lambda_{\theta}^+}\colon \mathscr{C}_{D^+}^{k,[\lambda \varepsilon]}\xrightarrow{\simeq}\mathscr{C}_{D^-}^{\ell,[\lambda^+_\theta/ \varepsilon]}.$$ 
Then $\mathscr{C}_{D^+}^k$ is a BTC iff $\mathscr{C}_{D^-}^\ell$ is a BTC. 
In this case, the braiding and associator on $\mathscr{C}_{D^-}^\ell$ are expressed by those of $\mathscr{C}_{D^+}^k$ and $\mathrm{Vect}_{\ssqrt{-1}\mathbb{R}}$.
\end{corollary}

\section{Free field representations}\label{sec:free_fields}
\subsection{Free field realization revisited}
Recall the Wakimoto realization of $V^k(\g)$ \cite{FF88,Wak}
\begin{equation}\label{wakimoto realization}
 V^k(\g)\hookrightarrow \beta\gamma^{\Delta_+}\otimes \pi^{k+h^\vee_+}_{\h_+},
\end{equation}
see Sect.~\ref{setting of our W-superalgebras} for the realization of $\h_+$. 
Here, $\beta\gamma^{\Delta_+}$ denotes the tensor product of the $\beta\gamma$-system indexed by the positive roots in $\Delta_+$. Explicitly, it is generated by the fields $\beta_\alpha(z), \gamma_\alpha(z)$ ($\alpha \in \Delta_+$) which satisfy the following OPEs:
\begin{align*}
    \beta_\alpha(z)\gamma_{\alpha'}(w)\sim \frac{\delta_{\alpha,\alpha'}}{(z-w)},\quad \beta_\alpha(z)\beta_{\alpha'}(w)\sim 0 \sim \gamma_\alpha(z)\gamma_{\alpha'}(w).
\end{align*}
Applying $H_f^0$ to \eqref{wakimoto realization}, we obtain the embedding 
\begin{align}\label{Wakimoto realization of subregular}
\subW\hookrightarrow \beta\gamma\otimes\pi^{k+h^\vee_{+}}_{\mathfrak{h}_+},
\end{align}
called the Wakimoto realization of $\subW$, see \cite{Gen}.
Here, $\beta\gamma$ is naturally associated to the positive root $\alpha_1$ which is of zero weight for the good grading in \S \ref{setting of our W-superalgebras}.

The embedding \eqref{Wakimoto realization of subregular} is indeed a restriction of \eqref{FFR} 
\begin{align}\label{successive embeddings}
    \subW\hookrightarrow 
    \beta\gamma\otimes\pi^{k+h^\vee_{+}}_{\mathfrak{h}_+}\hookrightarrow 
    \Pi(0)\otimes \pi^{k+h^\vee_{+}}_{\mathfrak{h}_+}
\end{align}
where we have used the Friedan--Martinec--Shenker bosonization \cite{FMS}
\begin{align*}
    \beta\gamma\xrightarrow{\simeq }\Ker_{\Pi(0)}\int Y(\hv{x},z)\ \dz,\quad \beta\mapsto \hv{x+y},\ \gamma\mapsto -x\hv{-(x+y)}.
\end{align*}
Note that at generic $k$, the image of the first embedding in \eqref{successive embeddings} is characterized by the screening operators $Q_i^+$ ($1\leq i\leq n$) defined in \eqref{screenings in Kazama Suzuki} and the screening operator $Q_0^+$ agrees with the above one.

By using the Heisenberg vertex algebra  $\pi^{k+h^\vee_+}_{\epsilon_0|\epsilon_1,\ldots,\epsilon_n}$ introduced in \eqref{intermediate heisenberg}, we may express the relative semi-infinite cohomology $\mathrm{H}_{\mathrm{rel}}^{0}(\mathfrak{gl}_1, \Pi(0)\otimes \pi^{k+h^\vee_{+}}_{\h_+}\otimes K_{D^+} ) $, which is a reformulation of the coset appearing in \eqref{iHR for coset}. 

\begin{lemma}\label{relcoh for the free field algebra}
    There is an isomorphism of vertex superalgebras 
    \begin{align}
     V_\Z\otimes \pi^{k+h^\vee_+}_{\epsilon_0|\epsilon_1,\ldots,\epsilon_n}  \xrightarrow{\simeq} \mathrm{H}_{\mathrm{rel}}^{0}(\mathfrak{gl}_1, \Pi(0)\otimes \pi^{k+h^\vee_{+}}_{\h_+}\otimes K_{D^+} )
    \end{align}
    which sends 
    \begin{align*}
        x&\mapsto \left(x-\frac{\varepsilon^2}{1-\varepsilon^2}y\right)\otimes \mathbf{1}\otimes \mathbf{1} +\mathbf{1}\otimes\frac{1}{1-\varepsilon^2}\epsilon_1\otimes \mathbf{1}-\mathbf{1}\otimes \mathbf{1}\otimes\frac{\varepsilon^2}{1-\varepsilon^2}J^-\\
        \epsilon_0&\mapsto \varepsilon^2\left(x-\frac{\varepsilon^2}{1-\varepsilon^2}y\right)\otimes \mathbf{1}\otimes \mathbf{1} +\mathbf{1}\otimes\frac{\varepsilon^2}{1-\varepsilon^2}\epsilon_1\otimes \mathbf{1}-\mathbf{1}\otimes \mathbf{1}\otimes\frac{\varepsilon^2}{1-\varepsilon^2}J^-\\
        \epsilon_1&\mapsto -\varepsilon^2(x+y)\otimes \mathbf{1}\otimes \mathbf{1}+ \mathbf{1}\otimes \epsilon_1 \otimes \mathbf{1}\\
        \epsilon_i  & \mapsto \epsilon_i \quad (2\leq i\leq n)\\
        \hv{mx} & \mapsto \hv{m(x+y)}\otimes \mathbf{1}\otimes \hv{mx}.
    \end{align*}
    Under the isomorphism, $Q_i^+$ $(0\leq i \leq n)$ are identified as 
    \begin{align}\label{modified screening}
    \widehat{Q}_i^+=\int Y(\hv{\widehat{A}_i^+},z)\ \dz,\quad (0\leq i \leq n)   
    \end{align}
    where 
    \begin{align*}
    \widehat{A}_0^+=x-(\epsilon_0+\epsilon_1),\ \widehat{A}_i^+=-\tfrac{1}{\varepsilon^2}(\epsilon_i-\epsilon_{i+1})\ (1\leq i <n),\ \widehat{A}_n^+=-\tfrac{1}{\varepsilon^2}\epsilon_n.
    \end{align*}
\end{lemma}

Let us apply the Feigin--Frenkel duality of the Virasoro vertex algebras (c.f. \cite[Chapter 15]{FBZ}) to the screening operators \eqref{modified screening} with $i=1,\ldots,n$. Then we obtain the following equivalent set of screening operators on $V_\Z\otimes \pi^{k+h^\vee_+}_{\epsilon_0|\epsilon_1,\ldots,\epsilon_n}$:
\begin{align}\label{final screening}
    \widetilde{Q}_0^+=\widehat{Q}_0^+,\ \widetilde{Q}_i^+=\int Y(\hv{\epsilon_i-\epsilon_{i+1}},z)\ \dz,\ \widetilde{Q}_n^+=\int Y(\hv{2\epsilon_n},z)\ \dz.  
    \end{align}
with $i=1,\ldots, n-1$.
Under the duality relation of the levels \eqref{level condition for FS}, recall that we have an isomorphism 
\begin{align}\label{switching to the osp side}
        \begin{array}{ccl}
       V_\Z\otimes \pi^{k+h^\vee_+}_{\epsilon_0|\epsilon_1,\ldots,\epsilon_n}  &  \xrightarrow{\simeq} & V_\Z\otimes \pi^{\ell+h^\vee_-}_{\h_-},
    \end{array}
\end{align}
see \eqref{isom for the large FFR}.
It identifies the screening operators \eqref{final screening} with \eqref{screenings in Kazama Suzuki}.

\begin{proposition}\label{finding spr via relcoh}
    There is an isomorphism of vertex superalgebras which makes the following diagram commutative:
	\begin{center}
		\begin{tikzcd}[row sep=large, column sep = huge]
			\Wsuper
			\arrow[d,hook]
			\arrow[r,"\simeq"]&
			\mathrm{H}_{\mathrm{rel}}^{0}(\mathfrak{gl}_1, \Wsub\otimes K_{D^+} )
			\arrow[d,hook]\\
			V_\Z\otimes \pi^{\ell+h^\vee_-}_{\h_-}
			\arrow[r, "\simeq"]&
			\mathrm{H}_{\mathrm{rel}}^{0}(\mathfrak{gl}_1, \Pi(0)\otimes \pi^{k+h^\vee_+}_{\h_+}\otimes K_{D^+} ).
		\end{tikzcd}
	\end{center}
\end{proposition}

\subsection{Wakimoto-type representations}\label{sec:wakimoto}
The realization \eqref{wakimoto realization} of $V^k(\g)$ induces a family of free field representations 
$\mathbb{W}^k_\mu:=\beta\gamma^{\Delta_+}\otimes \pi^{k+h^\vee_+}_{\h_+,\mu}$ $(\mu\in \h_+^*)$
called Wakimoto modules. 
They are $H_f^\bullet$-acyclic in general (see e.g. \cite[Proposition 4.5]{Gen}), and more precisely we have
\begin{align*}
    H_{f_{\mathrm{sub}}}^n(\mathbb{W}^k_\mu)\simeq \delta_{n,0} \mathbb{W}^{k,+}_\mu,\quad \mathbb{W}^{k,+}_\mu:=\beta\gamma\otimes \pi^{k+h^\vee_+}_{\h_+,\mu},
\end{align*}
which are naturally $\subW$-modules, also called Wakimoto modules. We note that the module structure on $\mathbb{W}^{k,+}_\mu$ can also be obtained through the embedding \eqref{Wakimoto realization of subregular}.
Consider the embedding \eqref{successive embeddings} instead of \eqref{Wakimoto realization of subregular}. Then by using the $\Pi(0)$-modules 
$$\Pi_\theta[a]=\bigoplus_{m\in \Z}\pi^{x,y}_{-\theta y+(m+a)(x+y)},\quad (\theta\in \Z, [a]\in \C/\Z)$$
we introduce the following induced $\subW$-modules
\begin{align*}
\widehat{\mathbb{W}}^{k,+}_{\mu_0;\mu}:=\Pi[\mu_0]\otimes \pi^{k+h^\vee_+}_{\h_+,\mu},\qquad \widehat{\mathbb{W}}^{k,+}_{\mu_0;\mu;\theta}:=S_{\theta J^+}\widehat{\mathbb{W}}^{k,+}_{\mu_0;\mu}\simeq \Pi_\theta[\mu_0]\otimes \pi^{k+h^\vee_+}_{\h_+,\mu+\theta\varepsilon^2\epsilon_1}\quad(\theta\in\Z).
\end{align*}
We call these modules \emph{thick Wakimoto modules} in the following.
Note that when $[\mu_0]=[0]$ we have $\mathbb{W}^{k,+}_{\mu}\subset \widehat{\mathbb{W}}^{k,+}_{0;\mu}$.
Similarly, we introduce the Wakimoto modules for the $\W$-superalgebra $\sprW$ by using \eqref{FFR}
$$\mathbb{W}^{\ell,-}_{\mu}:=V_{\Z}\otimes\pi^{\ell+h^\vee_-}_{\h_-,\mu}\quad (\mu\in\h_-^*).$$
If $(\alpha_0,\mu)=0$, then the $0$-th screening operator $Q_0^-$ in \eqref{screenings in Kazama Suzuki} can be used to introduce a $\sprW$-submodule
\begin{align*}
{W}^{\ell,-}_\mu:=\Ker\left(Q_0^-\colon\mathbb{W}^{\ell,-}_{\mu}\rightarrow \mathbb{W}^{\ell,-}_{\mu-\alpha_0}\right).
\end{align*}
We call ${W}^{\ell,-}_\mu$ the \emph{thin Wakimoto module} of highest weight $\mu$.

We have the following correspondence among these Wakimoto type free field representations.

\begin{theorem}\label{thm:wakimoto_correspondence}
For $(\mu_0,\mu)\in \C\times \h_+^*$ and $\theta\in\Z$, set
\begin{align*}
    \mu=\sum_{i=1}^n \mu_i \varpi_i\in \h_+^*,\quad \check{\mu}=\sum_{i=0}^n\mu_i\check{\varpi}_i\in\h_-^*.
\end{align*}

\noindent
\begin{enumerate}[wide, labelindent=0pt, font=\normalfont]
\item We have isomorphisms of $\Wsuper$-modules
    \begin{align*}
    \mathrm{H}_{\varepsilon^{-1}(\mu_0+(\mu,\varpi_1)+\varepsilon^2\theta)}^+\left(\widehat{\mathbb{W}}^{k,+}_{\mu_0;\mu;\theta}\right)\simeq
    \mathbb{W}^{\ell,-}_{-\frac{1}{2\varepsilon^2}\check{\mu}},\qquad
    \mathrm{H}_{\varepsilon^{-1}(\mu,\varpi_1)}^+\left(\mathbb{W}^{k,+}_\mu\right)\simeq
    W^{\ell,-}_{-\frac{1}{2\varepsilon^2}\check{\mu}}.
    \end{align*}
\item We have isomorphisms of $\Wsub$-modules    \begin{align*}
    \mathrm{H}_{\varepsilon((\check{\mu},\check{\varpi}_0)+\theta)}^-\left(\mathbb{W}^{\ell,-}_{\check{\mu}}\right)\simeq
    \widehat{\mathbb{W}}^{k,+}_{-2\varepsilon^2\mu_0;-2\varepsilon^2\mu;\theta},\qquad
    \mathrm{H}_{\varepsilon(\check{\mu},\check{\varpi}_0)}^-(W^{\ell,-}_{\check{\mu}})\simeq
    \mathbb{W}^{k,+}_{-2\varepsilon^2\mu}.
    \end{align*} 
\end{enumerate}
\end{theorem}

\proof
(1) Let us write $\mu=\sum \nu_i \epsilon_i$. We show the first isomorphism for $\theta=0$. Since $\widehat{\mathbb{W}}^{k,+}_{\mu_0;\mu}$ lies in $\KL_{D^+}^{k,[\mu_0+\nu_1]}$, we may apply the functor $\mathrm{H}_{\varepsilon^{-1}(\mu_0+(\mu,\varpi_1))}^+=\mathrm{H}_{\varepsilon^{-1}(\mu_0+\nu_1)}^+$. 
It follows from Lemma \ref{relcoh for the free field algebra} that we have an isomorphism 
\begin{align*}
    \mathrm{H}_{\varepsilon^{-1}(\lambda,\mu_0+\nu_1)}^+\left(\widehat{\mathbb{W}}^{k,+}_{\mu_0;\mu}\right)\simeq V_\Z\otimes \pi^{k+h^\vee_+}_{\epsilon_0|\epsilon_1,\ldots,\epsilon_n,[-(\mu_0+\nu_1);\nu_1,\ldots, \nu_n]}
\end{align*}
as $V_\Z\otimes \pi^{k+h^\vee_+}_{\epsilon_0|\epsilon_1,\ldots,\epsilon_n}$-modules where $\pi^{k+h^\vee_+}_{\epsilon_0|\epsilon_1,\ldots,\epsilon_n,[-(\mu_0+\nu_1);\nu_1,\ldots, \nu_n]}$ is the Fock module over $\pi^{k+h^\vee_+}_{\epsilon_0|\epsilon_1,\ldots,\epsilon_n}$ whose highest weight is determined by 
$$\epsilon_0=-(\mu_0+\nu_1), \epsilon_1=\nu_1,\ldots, \epsilon_n=\nu_n.$$
By the isomorphism \eqref{switching to the osp side}, we identify 
\begin{align*}
    V_\Z\otimes\pi^{k+h^\vee_+}_{\epsilon_0|\epsilon_1,\ldots,\epsilon_n,[-(\mu_0+\nu_1);\nu_1,\ldots, \nu_n]}\simeq  V_\Z\otimes\pi^{\ell+h^\vee_-}_{\h_-,[(\mu_0+\nu_1)/2\varepsilon^2, -\nu_1/2\varepsilon^2,\ldots, -\nu_n/2\varepsilon^2]}
\end{align*}
as $V_\Z\otimes\pi^{\ell+h^\vee_-}_{\h_-}$-modules where $\pi^{\ell+h^\vee_-}_{\h_-,[(\mu_0+\nu_1)/2\varepsilon^2, -\nu_1/2\varepsilon^2,\ldots, -\nu_n/2\varepsilon^2]}$ is the Fock module over $\pi^{\ell+h^\vee_-}_{\h_-}$ with the highest weight
$$t_0=\frac{(\mu_0+\nu_1)}{2\varepsilon^2}, t_1=-\frac{\nu_1}{2\varepsilon^2},\ldots, t_n=-\frac{\nu_n}{2\varepsilon^2},$$
see \eqref{realization of Cartan of type C} for the basis $t_i$'s of $\h_-$. Since the fundamental coweights $\check{\varpi}_i$ of $\osp_{2|2n}$ are expressed as
\begin{align*}
\check{\varpi}_0=2t_0,\ldots, \check{\varpi}_{n-1}=2(t_0+\cdots+t_{n-1}), \check{\varpi}_{n}=(t_0+\cdots+t_{n}),
\end{align*}
we obtain an isomorphism 
\begin{align*}
    \mathrm{H}_{\varepsilon^{-1}(\lambda,\mu_0+\nu_1)}^+\left(\widehat{\mathbb{W}}^{k,+}_{\mu_0;\mu}\right)\simeq V_\Z\otimes \pi^{\ell+h^\vee_-}_{\h_-,-\frac{1}{2\varepsilon^2}\check{\mu}}=\mathbb{W}^{\ell,-}_{-\frac{1}{2\varepsilon^2}\check{\mu}}
\end{align*}
as $V_\Z\otimes \pi^{\ell+h^\vee_-}_{\h_-}$-modules and thus as $\W^\ell_{D^-}$-modules by Proposition~\ref{finding spr via relcoh}.
This proves the case $\theta=0$.
The general case $\theta\in \Z$ follows from the case $\theta=0$ by a direct application of Theorem~\ref{equivalence of categories}. 
Next, setting $(\mu_0,\theta)=(0,0)$, the first isomorphism 
gives an embedding
\begin{align*}
     \mathrm{H}_{\varepsilon^{-1}(\mu,\varpi_1)}^+\left(\mathbb{W}^{k,+}_\mu\right)\subset \mathrm{H}_{\varepsilon^{-1}(\mu,\varpi_1)}^+\left(\widehat{\mathbb{W}}^{k,+}_{0;\mu}\right)\simeq
    \mathbb{W}^{\ell,-}_{-\frac{1}{2\varepsilon^2}\check{\mu}}
    \end{align*}
as $\sprW$-modules. 
The image coincides with the kernel of the screening operator induced by $Q_0^+$, that is $Q_0^-$ by Proposition~\ref{finding spr via relcoh}. Therefore, $\mathrm{H}_{\varepsilon^{-1}(\mu,\varpi_1)}^+\left(\mathbb{W}^{k,+}_\mu\right)\simeq W^{\ell,-}_{-\frac{1}{2\varepsilon^2}\check{\mu}}$.
(2) The assertion follows from (1) and the quasi-inverse property of functors $\mathrm{H}^\pm_\bullet$ (see Theorem~\ref{categorical equivalence}).
\endproof

\begin{remark}
Note that in Theorem~\ref{thm:wakimoto_correspondence} (1) the effect of the spectral flow twists $S_{\theta J^+}$ on $\widehat{\mathbb{W}}^{k,+}_{\mu_0;\mu}$ is invisible on the $\sprW$-module side: this difference is encoded in the choice of the functor $\mathrm{H}^+_\bullet$ depending on $\theta$. Conversely, the effect of the spectral flow twists $S_{\theta J^-}$ on $\mathbb{W}^{\ell,-}_{-\frac{1}{2\varepsilon^2}\check{\mu}}$, which shift $\mu_0$ by integers, is invisible on the $\subW$-modules side since $\widehat{\mathbb{W}}^{k,+}_{\mu_0;\mu}$ depends on $\mu_0$ \emph{modulo $\Z$} by definition.
\end{remark}

\subsection{Characters}
Here we compute the characters of Wakimoto-type modules over $\subW$ and $\sprW$, respectively.
Set
$$x_{\Gamma}^+=x_{\Gamma,1}^++x_{\Gamma,2}^+,\quad x_{\Gamma}^-=x_{\Gamma,1}^-+x_{\Gamma,2}^-$$
with
\begin{align*}
    &x_{\Gamma,1}^+:=-\tfrac{1}{2}(x+y)+ny,  &&x_{\Gamma,2}^+:=x_0^+-(n-1)\check{\varpi}_1=\rho+\varpi_n-n\varpi_1\in \h_+,\\
    &x_{\Gamma,1}^-:=-n x,  &&x_{\Gamma,2}^-:=x_0^--\tfrac{2n-1}{2}\check{\varpi}_0=2\rho-\varpi_n-\tfrac{2n-1}{2}\check{\varpi}_0\in \h_-,
\end{align*}
where $\rho=\half\sum_{\alpha\in\Delta}(-1)^{p(\alpha)}\alpha$ is the Weyl vector.
Then the embeddings \eqref{FFR} send
\begin{align}\label{character info}
\begin{array}{lllll}
\subW&\hookrightarrow\Pi(0)\otimes\pi^{k+h^\vee_+}_{\h_+},& \sprW&\hookrightarrow V_{\Z}\otimes\pi^{\ell+h^\vee_-}_{\h_-}\\
L &\mapsto L_{\mathrm{sug}}+\partial\left(x_{\Gamma}^+-\tfrac{1}{k+h^\vee_+}\rho\right)\quad 
& L& \mapsto L_{\mathrm{sug}}+\partial\left(x_{\Gamma}^--\tfrac{1}{\ell+h^\vee_-}\rho\right)\\
J^+ &\mapsto -y+\check{{\varpi}}_{1} & J^-& \mapsto x+\check{\varpi}_0\\
\end{array}
\end{align}
by \cite{KW22} where $L_{\mathrm{sug}}$ is the conformal vector given by the Segal--Sugarawa construction.
By using the $q$-Pochhammer symbol and the delta function
\begin{align*}
(a;q)_\infty=\prod_{n=0}^{\infty}(1-aq^n),\quad \delta(z)=\sum_{n\in \Z} z^n
\end{align*}
we obtain the following formulae for the characters $\ch M (z,q)=\mathrm{tr}_{M} z^{J^{\pm}_{0}}q^{L_0}$ of the Wakimoto-type modules.

\begin{proposition}\label{prop:Wakimoto_characters}\hspace{0mm}
\begin{enumerate}[wide, labelindent=0pt, font=\normalfont]
\item Set $\mu_{[\theta]}=\mu+\varepsilon^2\theta\varpi_1$ and  $$\Delta_{\mu_0;\mu;\theta}^k=\tfrac{(\mu_{[\theta]},\mu_{[\theta]}+2\rho)}{2\varepsilon^2}-(x_{\Gamma}^+,\mu_{[\theta]}+(\theta-\mu_0)J^+),\quad s_{\mu_0,\theta}=\tfrac{\mu_0(\mu_0-1)-(\mu_0-\theta)^2}{2}.$$ 
Then we have
\begin{align*}
\mathrm{ch}\ \mathbb{W}^{k,+}_{\mu}=\frac{q^{\Delta_{0;\mu;0}^k}z^{(\mu,\varpi_{1})}}{\poch{q}^{n}\poch{zq^{n},z^{-1}q^{1-n}}},\qquad
\mathrm{ch}\ \widehat{\mathbb{W}}^{k,+}_{\mu_0;\mu;\theta}=\frac{q^{\Delta_{\mu_0;\mu;\theta}^k+s_{\mu_0,\theta}}z^{(\mu_{[\theta]},\varpi_{1})+(\mu_0-\theta)}\delta(zq^{n+\theta})}{\poch{q}^{n+2}}.
\end{align*}
\item
    Set $\Delta_\mu^\ell=\varepsilon^2(\mu,\mu+2\rho)-(x_{\Gamma}^-,\mu)$. Then we have
    \begin{align*}
        &\mathrm{ch}\ \mathbb{W}^{\ell,-}_{\mu}=q^{\Delta_{\mu}^\ell}z^{(\mu,\check{\varpi}_0)}\frac{\poch{-zq^{n+\half},-z^{-1}q^{-n+\half} }}{\poch{q}^{n+1}},\quad \mathrm{ch}\ W^{\ell}_\mu= \frac{\mathrm{ch}\ \mathbb{W}^{\ell,-}_{\mu}}{1+z^{-1}q^{-n+\half} }.
    \end{align*}
\end{enumerate}
\end{proposition}

\proof
Except for $W^{\ell}_\mu$, the characters are obtained straightforwardly by using \eqref{character info}. 
For the remaining case, it suffices to show instead the equality 
\begin{align}
\mathrm{ch}\ \mathrm{H}^+_{\frac{(\mu,\varpi_1)}{\varepsilon}}(\mathbb{W}^{k,+}_\mu) (w,q)=\frac{\mathrm{ch}\ \mathbb{W}^{\ell,-}_{-\frac{1}{2\varepsilon^2}\check{\mu}}(w,q)}{1+z^{-1}q^{-n+\half}},\quad (\mu\in \h_+^*)
\end{align}
by Theorem \ref{thm:wakimoto_correspondence}. 
The left-hand side can be computed via the Euler--Poincar\'{e} principle thanks to the cohomology vanishing \eqref{rel coh functor}.
Moreover 
\begin{align*}
    \mathrm{ch}\ K_{D^+}^\lambda (z,w,q):&=\mathrm{tr}_{K_{D^+}^\lambda}z^{J^+_{*0}}w^{J^-}q^{L_0}
    =q^{-\half\lambda^2}z^{-\varepsilon \lambda}w^{\frac{1}{\varepsilon}\lambda}\frac{\poch{-\frac{w}{z}q^{\half},-\frac{z}{w}q^{\half}}}{\poch{q}}.
\end{align*}
It follows from the definition of the cohomology $\mathrm{H}_{\mathrm{rel}}^{\frac{\infty}{2}+\bullet}(\widehat{\mathfrak{gl}}_1,\mathfrak{gl}_1; \bullet)$ \cite{Feigin,FGZ} that
\begin{align*}
    \mathrm{ch}\  &\mathrm{H}^+_{\frac{(\mu,\varpi_1)}{\varepsilon}}(\mathbb{W}^{k,+}_\mu) (w,q)
    =\int \mathrm{ch}\ \mathbb{W}^{k,+}_\mu (z,q) \cdot \mathrm{ch}\ K_{D^+}^{\frac{(\mu,\varpi_1)}{\varepsilon}} (z,w,q) \cdot \poch{q}^2 \frac{\dz}{z}\\
    &=\frac{q^{\Delta_\mu^k-\frac{(\mu,\varpi_1)^2}{2\varepsilon^2}}w^{\frac{1}{\varepsilon^2}(\mu,\varpi_{1})}}{\poch{q}^{n-1}} \int\frac{(-\frac{w}{z}q^\half,-\frac{z}{w}q^\half;q)_\infty}{(zq^n,z^{-1}q^{1-n};q)_\infty}\frac{\dz}{z}\\
    &=q^{\Delta_\mu^k-\frac{(\mu,\varpi_1)^2}{2\varepsilon^2}}w^{\frac{1}{\varepsilon^2}(\mu,\varpi_{1})} \frac{\poch{-wq^{n+\half},-w^{-1}q^{-n+\frac{3}{2}}}}{\poch{q}^{n+1}}=\frac{\mathrm{ch}\ \mathbb{W}^{\ell,-}_{-\frac{1}{2\varepsilon^2}\check{\mu}}(w,q)}{1+z^{-1}q^{-n+\half}}.
\end{align*}
In this computation, we have used the identity \cite{Sto}
\begin{align*}
    \int\frac{\poch{-zq^\half,-z^{-1}q^\half}}{\poch{-azq^\half,-bz^{-1}q^\half}}\frac{\dz}{z}=\frac{\poch{aq,bq}}{\poch{q,abq}}
\end{align*}
with adaptation $(z,a,b)\mapsto (z^{-1}w,-wq^{n-\half},-w^{-1}q^{n+\half})$
for the third equality and 
\begin{align*}
    \Delta_{-\frac{1}{2\varepsilon^2}\check{\mu}}^\ell=\Delta_{0;\mu;0}^k-\frac{(\mu,\varpi_1)^2}{2\varepsilon^2},\quad 
    (\check{\mu},\check{\varpi}_0)=-2(\mu,\varpi_1)
\end{align*}
for the last equality. This completes the proof.
\endproof

\section{Rational case: exceptional subregular $\W$-algebra $\W_k(\so_{2n+1},\fsub)$}\label{sec:subreg_symplectic_walgebra}

In this section, we classify the irreducible modules over the simple subregular $\W$-algebras $\ssubW$ at exceptional levels. 
We adapt the strategy for exceptional $\W$-algebras developed in \cite{AvE} in our setting. Therefore, we recollect some general results from \cite{AvE} at the beginnings of \S \ref{admissible representations} and \ref{Zhu's algera} for convenience of readers.

\subsection{Admissible representations}\label{admissible representations}
We start with collecting some useful data to describe the representation theory of the affine Kac--Moody algebra associated with ${\so}_{2n+1}$.
Following the notations of Sect.~\ref{sec:affine_VA}, let $\widetilde{\h}$ be the extended Cartan subalgebra of the affine Kac--Moody algebra $\widetilde{\g}$ and $\widehat{\Delta}\subset \widetilde{\h}^*$ be the set of roots for $\widehat{\g}$.
The subsets of real roots and positive real roots in $\widehat{\Delta}$ are given by
\begin{align*}
    \widehat{\Delta}^{\mathrm{re}}=\{\alpha+n\delta\mid \alpha\in \Delta, n\in \Z\},\quad 
    \widehat{\Delta}^{\mathrm{re}}_+=\Delta_+\sqcup \{\alpha+n\delta\mid \alpha\in \Delta, n>0\},
\end{align*}
respectively.
Let $\widehat{W}$ denote the affine Weyl group of $\widehat{\g}$, which is isomorphic to $W\ltimes \check{\rQ}$, and $\widetilde{W}=W\ltimes \check{\rP}$ the extended affine Weyl group. The elements $\alpha \in \check{\rP}$ act on $\widetilde{\h}^*$ by 
$$t_\alpha \lambda=\lambda+(\lambda,\delta)\alpha-\left(\tfrac{1}{2}|\alpha|^2(\lambda,\delta) + (\lambda,\alpha)\right)\delta.$$
For $\lambda\in \widehat{\h}^*$, let $\widehat{\Delta}(\lambda)$ be the integral root system
\begin{align*}
\widehat{\Delta}(\lambda)=\{\alpha\in \widehat{\Delta}^{\mathrm{re}}\mid (\lambda+\widehat{\rho},\check{\alpha})\in \Z\} 
\end{align*}
where $\widehat{\rho}=\rho+h^\vee\Lambda_0$ is the affine Weyl vector for $\widehat{\g}$ and $\check{\alpha}=\frac{2}{(\alpha,\alpha)}\alpha$.
Let $\widehat{\Delta}(\lambda)_+=\widehat{\Delta}(\lambda)\cap \widehat{\Delta}^{\mathrm{re}}_+$ be the set of positive roots and $\widehat{\Pi}(\lambda)\subset \widehat{\Delta}(\lambda)_+$ the set of simple roots.
The weight $\lambda\in \widehat{\h}^*$ is \emph{admissible} if it is regular dominant, i.e. $(\lambda+\widehat{\rho},\check{\alpha})>0$ for all $\alpha\in\widehat{\Delta}(\lambda)_+$, and $\Q\widehat{\Delta}(\lambda)=\Q\widehat{\Delta}^{\mathrm{re}}$.
The level $k$ is called admissible if $k\Lambda$ is an admissible weight.

\begin{proposition}[\cite{KW89}]
    The level $k\in\C$ is admissible iff it is of the form
    \begin{align}\label{exceptional admissible weights}
        k+h^\vee=\frac{p}{q},\quad p,q>0,\,(p,q)=1,\,  
    \begin{cases}
    p\geq h^\vee & (r^\vee,q)=1,\\
    p\geq h & (r^\vee,q)=r^\vee,
    \end{cases}
    \end{align}
    where $r^\vee$ is the lacing number of $\g$, that is, $r^\vee=1,2,3$ if $\g$ has type $ADE$, $BCF$, $G$ respectively.
\end{proposition}

From now on, we set 
$$\g=\so_{2n+1},\quad f=f_{\mathrm{sub}},\quad h^\vee=h^\vee_+=2n-1,\quad h=h_+=2n.$$
We consider admissible levels $k$ with denominator $q=2n-1$, called principal, and $q=2n$, called coprincipal.
For such levels, the associated variety of the simple affine vertex algebra $L_k(\g)$ is
\begin{align*}
X_{L_k(\g)}=\begin{cases}\{x\in \g\mid x^{2\bar{q}}|_{L(\theta)}=0 \} & (q=2n-1)\\ \{x\in \g\mid x^{2\bar{q}}|_{L(\theta_s)}=0 \} &(q=2n) \end{cases},\quad 
\bar{q}=\begin{cases} 2n-1 & (q=2n-1)\\ n &(q=2n), \end{cases}
\end{align*}
which coincides with the closure of the subregular nilpotent orbit $\overline{G.f}$ \cite{A3}.
In this case, the simple $\W$-algebra $\ssubW$ is called \emph{exceptional} \cite{AvE,KW08}. It is known to be lisse \cite{A3} and rational \cite{Mc}.

We have $\widehat{\Pi}(k\Lambda_0)=\{\dot{\alpha}_0,\ldots, \dot{\alpha}_n\}$ where\begin{equation*}
\dot{\alpha}_0=\begin{cases}-\theta+\bar{q}\delta& (q=2n-1)\\ -\theta_s+\bar{q}\delta& (q=2n)\end{cases},\quad \dot{\alpha}_1=\alpha_1,\ldots, \dot{\alpha}_n=\alpha_n,
\end{equation*}
and the associated Dynkin diagram is the extended Dynkin diagram $\tilde{B}_n$ for $q=2n-1$ and the twisted one $\tilde{C}_n^t$ for $q=2n$ as described in Fig.~\ref{fig:Bn}, \ref{fig:Cnt} below.
\begin{figure}[h]
\centering
\begin{minipage}[t]{.5\textwidth}
	\centering
	\caption{$\tilde{B}_n$}\label{fig:Bn}
\end{minipage}
\begin{minipage}[t]{0.49\textwidth}
	\centering
	\caption{$\tilde{C}_n^t$}\label{fig:Cnt}
\end{minipage}
\begin{tikzpicture}
    \node at (-2.2, 0)   (a) {\dynkin[root
					radius=.08cm,labels={\dot{\alpha}_1,\dot{\alpha}_0,\dot{\alpha}_2,\dot{\alpha}_{n-1},\dot{\alpha}_n},edge length=1cm]B[1]{oo.oo}};
    \node at (6.2, -0.17)   (b) {\dynkin[extended, reverse arrows, root
					radius=.08cm,labels={\dot{\alpha}_0,\dot{\alpha}_1,\dot{\alpha}_{n-1},\dot{\alpha}_n},edge length=1cm]C{o.oo}};
\end{tikzpicture} 
\end{figure}
Note that in both cases the automorphism group $\mathrm{Aut}(\widehat{\Pi}(k\Lambda_0))$ is isomorphic to $\Z_2$ and is generated by $\widetilde{\sigma}={\sigma}t_{-\bar{q}\eta_\circ}$ with 
\begin{align}\label{Dynkin automorphisms}
{\sigma}=\begin{cases}\epsilon_i\mapsto (-1)^{\delta_{i,1}}\epsilon_i& (q= 2n-1)\\ \epsilon_i\mapsto -\epsilon_{n+1-i} & (q=2n) \end{cases},\quad 
\eta_\circ=\begin{cases}\check{\varpi}_1& (q= 2n-1)\\ \check{\varpi}_n & (q=2n). \end{cases}
\end{align}

The category $\mathcal{O}_k$ of $L_k(\g)$-modules,
whose objects are $\widehat{\g}$-modules in the category $\mathcal{O}$,
is semisimple and the simple modules are given by the \emph{(co)principal admissible representations} \cite{A4}. 
More concretely, they are simple highest weight representations $L_k(\lambda)$ such that the associated Dynkin diagram of $\widehat{\Pi}(\lambda)$ is the same as the one of $\widehat{\Pi}(k \Lambda_0)$, or equivalently such that $\lambda$ lies in
\begin{align}\label{principal admissble weights}
\Pr^k=\bigcup_{\widetilde{w}\in \widetilde{S}}
    \Pr_{\widetilde{w}}^k,\quad \Pr_{\widetilde{w}}^k=\tilde{w}\circ \Pr_\Z^k,\quad \widetilde{S}=\left\{\widetilde{w}\in \widetilde{W}\left| \widetilde{w}\ \widehat{\Pi}(k\Lambda_0)\subset \widehat{\Delta}_+^{\mathrm{re}}\right.\right\}
\end{align}
where $\circ$ is the dot action $\widetilde{w}\circ\lambda=\widetilde{w}(\lambda+\widehat{\rho})-\widehat{\rho}$, for $\widetilde{w}\in \widetilde{W}$ and $\lambda\in\widehat{\h}^*$, and $\Pr_\Z^k$ is the set of admissible weights at level $k$ such that $\widehat{\Pi}(\lambda)=\widehat{\Pi}(k\Lambda_0)$.
More precisely, we have
\begin{align*}
\Pr_\Z^k
=\begin{cases}
\{\lambda\in \rP_+ \mid (\lambda,\theta)\leq p-h^\vee\} & (q=2n-1),\\  
\{\lambda\in \rP_+ \mid (\lambda,\check{\theta}_s^)\leq p-h\} & (q=2n).
\end{cases}
\end{align*} 
The set $\Pr_\Z^k$ can be identified with
\begin{align*}
\Pr_\Z^k
\simeq \begin{cases} \rP_+^{p-h^\vee} & (q=2n-1),
\\ {}^L\check{\rP}_+^{p-h} & (q=2n), \end{cases}
\end{align*} 
where ${}^L\check{\rP}_+^{p-h}$ is the set of dominant integral coweights at level $p-h$ for type $C_n$ -- the Langlands dual of $B_n$ -- in which we identify $\lambda=\sum \lambda_i \varpi_i\in \rP_+$ for $B_n$ with $\lambda=\sum \lambda_i \check{\varpi}_i\in \check{\rP}_+$ for $C_n$. 

\subsection{Zhu's algebra}\label{Zhu's algera}
To any conformal vertex algebra $V$ is attached a unital associative $\C$-algebra $A(V)$, called Zhu's algebra \cite{Z}. 
In particular, for a general Lie algebra $\g$, Zhu's algebra $A(V^k(\g))$ of the universal affine vertex algebra $V^k(\g)$ is isomorphic to the enveloping algebra $U(\g)$. 
Therefore, the Zhu's algebra $A(L_k(\g))$ of the simple affine vertex algebra $L_k(\g)$ is isomorphic to a quotient of $U(\g)$ by a certain two-sided ideal $I_k$ depending on the level $k$. 
When $k$ is admissible, $A(L_k(\g))$ is known to be semisimple, and we have the decomposition \cite[Theorem 3.4]{AvE}
\begin{align}
A(L_k(\g))\simeq \bigoplus_{\lambda\in [\Pr^k]}U(\g)/J_\lambda,
\end{align}
where $J_\lambda$ is the annihilating ideal in $U(\g)$ of the highest weight $\g$-module $L(\lambda)$ ($\lambda\in \h^*$) and $[\Pr^k]$ the set of equivalent classes of $\Pr^k$ under the dot action of $W$ as $J_\lambda=J_\mu$ iff $\mu\in W\circ \lambda$.

Still for a general $\g$ and given a nilpotent element $f\in\g$, the Zhu's algebra of the universal $\W$-algebra $\W^k(\g,f)$ is isomorphic to the corresponding finite $\W$-algebra $U(\g,f)$ \cite{Pr}, which is isomorphic to the finite analogue of the BRST reduction $H_f^0(U(\g))$ (see \cite{A1,DK}). Moreover, we have an isomorphism of $\C$-algebras
\begin{align*}
A(H_f^0(L_k(\g)))\simeq \bigoplus_{\lambda\in [\mathrm{Pr}^k]} H_f^0(U(\g)/J_\lambda).
\end{align*}
By \cite{Lo}, $H_{f}^0(U(\g)/J_\lambda)\neq 0$ does not vanish only if $\lambda$ lies in  
\begin{align}\label{non-degenerate condition}
\Pr_\circ^k=\{\lambda\in \Pr^k\mid \sharp\Delta(\lambda)=2\},
\end{align}
and in this case, it is a semisimple algebra which decomposes as
$$H_f^0(U(\g)/J_\lambda)\simeq \bigoplus_{E}\mathrm{End}_\C(E)$$
where $E$ runs over the complete set of representatives $I^\W_{[\lambda]}$ of simple finite-dimensional $U(\g,f)$-modules such that the corresponding simple $U(\g)$-module under the Skryabin's equivalence has the annihilating ideal $J_\lambda$.

Going back to our setting $\g=\so_{2n+1}$ and $f=f_{\mathrm{sub}}$, by \cite[Theorem 7.2]{AvE} we have an isomorphism of vertex algebras
\begin{align}\label{simple exceptional W-algebra}
    \ssubW\simeq H_f^0(L_k(\g)),
\end{align}
which can be used to describe Zhu's algebra of the simple $\W$-algebras $\ssubW$ \cite[Theorem 4.5]{AvE}
\begin{align}\label{ finite W-algebra}
A(\ssubW)\simeq \bigoplus_{[\lambda]\in [\Pr^k_\circ]}\bigoplus_{E\in I^\W_{[\lambda]}}\mathrm{End}_\C(E).
\end{align}
Since $\ssubW$ is lisse \cite{A3}, all the simple $\ssubW$-modules are positively graded by conformal weights and thus are in one-to-one correspondence with simple $A(\ssubW)$-modules by \cite{Z}. Therefore, we have
\begin{align}\label{simple modules of finite}
    \Irr(\ssubW)\simeq  \bigsqcup_{\lambda\in [\Pr^k_\circ]} I^\W_{[\lambda]}.
\end{align}
By constructing enough simple $\ssubW$-modules, we describe the sets $I^\W_{[\lambda]}$. 

To start with, we give a more precise description of $[\Pr^k_\circ]$.
Let us describe the subset $\widetilde{S}\subset \widetilde{W}$ in \eqref{principal admissble weights}.
We introduce the sets
\begin{align*}
    \check{\rP}_+^{\bar{q}}:=\{\alpha\in \check{\rP}_+\mid (\alpha,\theta)\leq {\bar{q}} \},\quad 
    {}^L\rP_+^{\bar{q}}:=\{\alpha\in {}^L\rP_+\mid (\alpha,{}^L\theta)\leq \bar{q} \}
\end{align*}
for $q=2n-1$ and $q=2n$ respectively. Note that ${}^L\rP_+^{\bar{q}}$ is naturally identified with 
$\{\alpha\in \check{\rP}_+\mid (\alpha,\check{\theta}_s)\leq q \}$.
We have
\begin{align*}
    \widetilde{S}=\left\{\widetilde{w}=wt_{-\eta}\in \widetilde{W}\mid (w,\eta) \text{ satisfies } (S1), (S2), (S3)\right\}
\end{align*}
where for $q=2n-1$ (resp. $q=2n$),
\begin{itemize}
    \item[(S1)] $\eta\in \check{\mathrm{P}}_+^{\bar{q}}$ (resp.\ $\eta\in  {}^L\rP_+^{\bar{q}}$),
    \item[(S2)] $(\eta,\theta)< {\bar{q}}$ if $w\theta \in \Delta_+$ (resp. $(\alpha,\theta_s)< \overline{q}$ if $w\theta_s \in \Delta_+$),
    \item[(S3)] $(\eta, \alpha_i)>0$ if $w \alpha_i\in \Delta_-$ for each $i=1,\ldots, n$.
\end{itemize}
When representing the admissible weights as $\widetilde{w}\circ \lambda =w\circ(\lambda-\frac{p}{q}\eta)$ with $\widetilde{w}=wt_{-\eta}\in\widetilde{S}$ and $\lambda\in \Pr_{\Z}^k$, the condition $\sharp\Delta(\lambda)=2$ in \eqref{non-degenerate condition} reads
$\sharp\{\alpha \in \Delta_+\mid (\eta,\check{\alpha})\in q\Z\}=1$,
which is equivalent to 
\begin{align}\label{condition for eta}
\sharp\left\{\dot{\alpha}_i\in \widehat{\Pi}(k\Lambda_0) \left| (\check{\dot{\alpha}}_i|\eta+D)=0 \right.\right\}=1.
\end{align}
Note that the set appearing in the condition \eqref{condition for eta} is independent of the choice of $\lambda$ and $w$. 
Let $\check{\rP}_{+,\mathrm{sub}}^{\bar{q}}$ (resp.\ ${}^L\mathrm{P}_{+,\mathrm{sub}}^{\bar{q}}$) denote the subset of $\check{\mathrm{P}}_+^{\bar{q}}$ (resp.\ ${}^L\mathrm{P}_+^{\bar{q}}$) consisting of elements satisfying \eqref{condition for eta}.
It is straightforward to show
\begin{align}\label{shift parameter}
\check{\rP}_{+,\mathrm{sub}}^{\bar{q}}=\{\eta_0,\eta_1, \eta_2, \eta_2',\ldots,  \eta_n, \eta_n'\},\quad 
{}^L\mathrm{P}_{+,\mathrm{sub}}^{\bar{q}}=\{\eta_0,\eta_1, \ldots,  \eta_n\},
\end{align}
where
\begin{align*}
    \eta_i=\sum_{\begin{subarray}{c}0\leq j \leq n \\ j\neq 0,i\end{subarray}}\check{\varpi}_j,\quad
    \eta_i'=\check{\varpi}_1+\sum_{\begin{subarray}{c}1\leq j \leq n \\ j\neq i\end{subarray}}\check{\varpi}_j.
\end{align*}
Conversely, for each $\eta$ in \eqref{shift parameter}, we find an element $\widetilde{w}=wt_{-\eta}\in \widetilde{W}$ which actually lies in $\widetilde{S}$. Explicitly, we find the elements
\begin{align}\label{lifts to affine Weyl}
    &\sigma t_{-\eta_0},\quad t_{-\eta_i}\ (i=1,\ldots, n),\quad t_{-\eta_i'}\ (i=2,\ldots, n),
\end{align}
which indeed satisfy the conditions (S1), (S2), (S3). 
Thereby, we obtain a subset of $\Pr^k_\circ$ parameterized by $\rP_+^{p-h^\vee} \times \check{\rP}_{+,\mathrm{sub}}^{\bar{q}}$ for $q=2n-1$ and ${}^L\check{\rP}_+^{p-h} \times {}^L\rP_{+,\mathrm{sub}}^{\bar{q}}$ for $q=2n$, respectively. 
Note that the assignment $\eta\mapsto wt_{-\eta}$ only differs by the choice of elements in $W$, which does not change the image under the quotient $\Pr^k_\circ\twoheadrightarrow [\Pr^k_\circ]$.
This implies that these subsets surject onto $[\Pr^k_\circ]$:
\begin{align}\label{surjection1}
\begin{array}{cll}
\rP_+^{p-h^\vee}\times \check{\rP}_{+,\mathrm{sub}}^{\bar{q}} & \twoheadrightarrow &[\Pr_\circ^k]\\
(\lambda,\eta_0) &\mapsto & [\sigma\circ(\lambda-\tfrac{p}{q}\eta_0)]\\
(\lambda,\eta_i) &\mapsto & [\lambda-\tfrac{p}{q}\eta_i]\quad (1\leq i\ \leq n)\\
(\lambda,\eta_i') &\mapsto & [\lambda-\tfrac{p}{q}\eta_i']\quad (2\leq i\ \leq n),
\end{array}
\end{align}
for $q=2n-1$ and 
\begin{align}\label{surjection2}
\begin{array}{cll}
{}^L\check{\rP}_+^{p-h}\times {}^L\rP_{+,\mathrm{sub}}^{\bar{q}} & \twoheadrightarrow &[\Pr_\circ^k]\\
(\lambda,\eta_0) &\mapsto & [\sigma\circ(\lambda-\tfrac{p}{q}\eta_0)]\\
(\lambda,\eta_i) &\mapsto & [\lambda-\tfrac{p}{q}\eta_i]\quad (1\leq i \leq n),
\end{array}
\end{align}
for $q=2n$. 
To describe the fibers, recall that the union \eqref{principal admissble weights} is almost disjoint by \cite{KW89}
\begin{align*}
\Pr_{\widetilde{w}_1}^k\cap \Pr_{\widetilde{w}_2}^k\neq \emptyset \Leftrightarrow \Pr_{\widetilde{w}_1}^k= \Pr_{\widetilde{w}_2}^k\ \Leftrightarrow \widetilde{w}_2\in \widetilde{w}_1 \mathrm{Aut}(\widehat{\Pi}(k\Lambda_0)).
\end{align*} 
It follows that $\Aut(\widehat{\Pi}(k\Lambda_0))(\simeq\Z_2)$ acts on each fiber transitively. The action of the generator $\widetilde{\sigma}$ on the set $\rP_+^{p-h^\vee}\times \check{\rP}_{+,\mathrm{sub}}^q$ (resp.\ ${}^L\check{\rP}_+^{p-h}\times {}^L\rP_{+,\mathrm{sub}}^{\bar{q}}$) for $q=2n-1$ (resp.\ $q=2n$) is described as
\begin{align}\label{fiber action} 
\begin{cases}
(\lambda,\eta_0)\leftrightarrow (\sigma\lambda,\eta_1),\quad (\lambda,\eta_i) \leftrightarrow (\sigma\lambda,\eta_i'),\quad (i=2,\ldots,n) & (q=2n-1),\\
(\lambda,\eta_i )\leftrightarrow (\sigma\lambda, \eta_{n-i}) & (q=2n). 
\end{cases}
\end{align}
where the ${\sigma}$-action on the first factor is defined in \eqref{Dynkin automorphisms}.
\begin{proposition}\label{parametrization}
We have a bijective correspondence
\begin{align*}
[\Pr_\circ^k]\simeq \begin{cases}\displaystyle{(\rP_+^{p-h^\vee}\times \check{\rP}_{+,\mathrm{sub}}^{\bar{q}})/\Z_2} & (q=2n-1),\\ \displaystyle{({}^L\check{\rP}_+^{p-h}\times {}^L\rP_{+,\mathrm{sub}}^{\bar{q}})/\Z_2} & (q=2n). \end{cases}
\end{align*}
Moreover, the $\Z_2$-action on the right-hand side is fixed-point free. 
\end{proposition}

\proof
It remains to show that the last statement.
For $q=2n-1$, it follows from \eqref{fiber action} immediately. 
For $q=2n$, we express the first factor $\lambda=\sum_i \lambda_i \varpi_i$. Then $\sigma$ acts as 
\begin{align*}
    (\lambda_1,\ldots,\lambda_{n-1},\lambda_n )\mapsto \left(\lambda_{n-1},\ldots,\lambda_1,p-h-2\sum_{i=1}^{n-1}\lambda_{i}-\lambda_n\right). 
\end{align*}
Since $p$ is odd and $h=2n$ is even, the $\sigma$-action on the first factor ${}^L\check{\mathrm{P}}_+^{p-h}$ is fixed-point free. 
\endproof

\subsection{Construction of simple modules}\label{sec:simple_modules_subreg}
For $\lambda \in \Pr_\Z^k$, let us introduce the $\ssubW$-module
\begin{equation*}
\mathbf{L}_k(\lambda):=H^0_f(L_k(\lambda))
\end{equation*}
obtained from the BRST reduction of the simple highest-weight $\widehat{\g}$-module $L_k(\lambda)$ at level $k$ (see Sect.~\ref{sec:BRST}).
By \cite[Theorem 7.8]{AvE}, it is unique simple $\ssubW$-module whose lowest conformal subspace $\mathbf{L}_k(\lambda)^{\Top}$ admits the central character $\chi_{\lambda-\frac{p}{q}x_0^+}\colon \mathcal{Z}(\g)\rightarrow \C$, see \eqref{grading element} for the definition of $x_0^+$.
Hence, we have
\begin{align*}
I^{\W}_{[\lambda-\frac{p}{q}x_0^+]}=\{\mathbf{L}_k(\lambda)^{\Top}\}.
\end{align*}
Note that $x_0^+=\eta_1$ by \eqref{shift parameter} and thus $[\lambda-\frac{p}{q}x_0^+]$ $(\lambda \in \Pr_\Z^k)$ are all different by Proposition \ref{parametrization}.

\begin{lemma}\label{upper bound of simples}
The number of simple $\sWsub$-modules is bounded by
$$\sharp\Irr(\ssubW)\leq \bar{q}\cdot \sharp{\Pr}_\Z^k.$$
\end{lemma}

\proof
We realize the desired upper bound by using \eqref{simple modules of finite}. 
By \cite{Lo}, each set $I^\W_{[\lambda]}$ ($\lambda\in [\Pr_\circ]^k$) has a transitive action of the component group $C(f)$, which is $\Z_2$ in our case \cite[Corollary 6.1.6]{CM}.
Therefore, $\sharp I^\W_{[\lambda]}=1,2$.
We have $\sharp I^{\W}_{[\lambda-\frac{p}{q}x_0^+]}=1$ for $\lambda\in \Pr_\Z^k$ and bound $\sharp I^{\W}_{[\lambda]}\leq 2$ otherwise. It follows from Proposition~\ref{parametrization} that 
\begin{align*}
\sharp [\Pr_\circ^k]
&=\half\begin{cases}
2n \cdot\sharp\Pr_\Z^k &(q=2n-1)\\
(n+1)\cdot \sharp\Pr_\Z^k &(q=2n)\\
\end{cases}\\
&=\half(\bar{q}+1)\sharp \Pr_\Z^k.
\end{align*}
Then we obtain the required upper bound
$$\sharp \Irr(\ssubW) \leq \sharp\Pr_\Z^k+ 2 (\sharp [\Pr_\circ^k ] -\sharp \Pr_\Z^k )=\bar{q}\cdot \sharp \Pr_\Z^k.$$
\endproof 

In order to construct more simple $\ssubW$-modules, we consider spectral flow twists of $\mathbf{L}_k(\lambda)$. 
It is useful to introduce the set $\mathrm{wt}(M)$ of $J_0^+$-eigenvalues for a given $\ssubW$-module $M$. Since $\mathrm{wt}(\ssubW)=\Z$, we have the spectral flow twists $S_\theta=S_{\theta J^+}$ $(\theta\in \Z)$.
Recall that $S_\theta$ transforms the actions of $J^+(z)=\sum_{m\in \Z}J^+_mz^{-m-1}$ and $L(z)=\sum_{m\in \Z}L_mz^{-m-2}$ into
\begin{align*}
J_{m}^+\mapsto J_m^++\delta_{m,0}\tfrac{p-q}{q}\theta,\quad L_m\mapsto L_m+\theta J_m^++\tfrac{1}{2}\tfrac{p-q}{q}\theta^2 \delta_{m,0}.
\end{align*}
Accordingly, the character $\mathrm{ch}_{M}(z,q):=\mathrm{tr}_{M}z^{J^+_0}q^{L_0}$ and the set $\mathrm{wt}(M)$ transform into 
\begin{align}\label{effect of spectral flow twists}
\mathrm{ch}_{S_\theta M}(z,q)=q^{\half\theta^2 \frac{p-q}{q}}z^{\frac{p-q}{q}\theta}\mathrm{ch}_{M}(zq^\theta,q),\quad \mathrm{wt}(S_\theta M)=\tfrac{p-q}{q}\theta+ \mathrm{wt}(M).
\end{align}

The following classification was obtained for $n=2$ by the first-named author \cite{Fa}.

\begin{theorem}\label{classification of simple modules}
For admissible levels $k=-h^\vee+\frac{p}{q}$ with $q=2n-1, 2n$,
the complete set of simple $\sWsub$-modules is given by 
$$\Irr(\sWsub)=\{S_\theta \mathbf{L}_k(\lambda)\mid \lambda\in \Pr_\Z^k,\ 0\leq \theta<\bar{q} \}.$$
\end{theorem}

\proof
Since $J^+$ corresponds to $\check{\varpi}_1$, $\mathrm{wt}(\mathbf{L}_k(\lambda))=\Z$ or $\half+\Z$ for $\lambda \in \Pr_\Z^k$.
It follows from \eqref{effect of spectral flow twists} that $S_\theta \mathbf{L}_k(\lambda)\simeq  \mathbf{L}_k(\mu)$ ($\lambda,\mu\in \Pr_\Z^k$) happens only if $\theta\in \bar{q}\Z$. Since $\mathbf{L}_k(\lambda)$ ($\lambda\in \Pr_\Z^k$) are all inequivalent simple modules, $S_\theta \mathbf{L}_k(\lambda)$ ($\lambda\in \Pr_\Z^k,\ 0\leq \theta<\bar{q})$ are also all inequivalent simple modules. The number of these modules already gives the upper bound by Lemma \ref{upper bound of simples}. Hence, they complete the set of all the simple $\ssubW$-modules.
\endproof

Note that we have not used the rationality of $\ssubW$ so far. In fact, the previous classification of simple modules gives a quick alternative proof.

\begin{corollary}[\cite{Mc}]
The simple exceptional $\W$-algebra $\sWsub$ is rational.
\end{corollary}

\proof
By Theorem \ref{classification of simple modules}, it suffices to show
$$\mathrm{Ext}^1_{\ssubW}(S_{\theta_1} \mathbf{L}_k(\lambda_1),S_{\theta_2} \mathbf{L}_k(\lambda_2))=0,\quad (\lambda_i\in \Pr_\Z^k,\ 0\leq \theta_i<\bar{q}).$$ 
Suppose that there exists a short exact sequence of $\ssubW$-modules
\begin{align}\label{ext vanishing}
0\rightarrow S_{\theta_2} \mathbf{L}_k(\lambda_2)\rightarrow M \rightarrow S_{\theta_1} \mathbf{L}_k(\lambda_1)\rightarrow 0.
\end{align}
If $\theta_1\neq \theta_2$, then 
$\mathrm{wt}(S_{\theta_1}\mathbf{L}_k(\lambda_1))\cap \mathrm{wt}(S_{\theta_2}\mathbf{L}_k(\lambda_2))=\emptyset$. 
In this case, since $J_0^+$ acts on $M$ semisimply, its eigenvalue decomposition induces the splitting of \eqref{ext vanishing} as $\ssubW$-modules. 
Next, we consider the case $\theta_1=\theta_2$. 
We apply the exact functor $S_{-\theta_1}$ to \eqref{ext vanishing} and obtain
\begin{align*}
0\rightarrow \mathbf{L}_k(\lambda_2)\rightarrow S_{-\theta_1}M \rightarrow \mathbf{L}_k(\lambda_1)\rightarrow 0.
\end{align*}
It splits as the proof of \cite[Theorem 7.9]{AvE} implies $\mathrm{Ext}^1_{\ssubW}(\mathbf{L}_k(\lambda_2),\mathbf{L}_k(\lambda_1))=0$.
\endproof

It follows from Theorem~\ref{compatibility of spectral flow twists} that all the simple $\ssubW$-modules are obtained as the BRST reduction of spectral flow twists of admissible representations of $L_k(\g)$.

\begin{proposition}\label{BRST image of spectral flow twists}
For $\lambda\in \Pr_\Z^k$ and $\theta \in \Z$, 
we have isomorphisms of $\sWsub$-modules
\begin{align*}
H_{f}^{n}(S_{\theta \check{\varpi}_1} L_k(\lambda))\simeq \delta_{n,\theta(h^\vee-1)}S_{\theta} \mathbf{L}_k(\lambda),\qquad
H_{f,{\theta\check{\varpi}_1}}^n(L_k(\lambda)) \simeq \delta_{n,0} S_{\theta} \mathbf{L}_k(\lambda).
\end{align*}
\end{proposition}

\proof
Recall that $J^+$ corresponds to $\check{\varpi}_1$.
The statement follows from Theorem \ref{compatibility of spectral flow twists} since 
$2(\check{\varpi}_1,\rho_{>0})=2(\check{\varpi}_1,\rho-\tfrac{1}{2}\alpha_1)=h^\vee-1$.
\endproof

\begin{proposition}
For $\lambda\in \Pr_\Z^k$ and $\theta \in \Z$,  the simple $L_k(\g)$-modules 
$S_{\theta \check{\varpi}_1}L_k(\lambda)$
are all non-isomorphic to each other.
\end{proposition}

\proof
Suppose we have an isomorphism 
$S_{a \check{\varpi}_1}L_k(\lambda)\simeq S_{b\check{\varpi}_1}L_k(\mu)$.
Applying $S_{-b\check{\varpi}_1}$ to this isomorphism, we may assume $b=0$. Then by applying $H_f$, Corollary \ref{BRST image of spectral flow twists} gives an isomorphism of $\ssubW$-modules between 
$H_f^n(S_{a\check{\varpi}_1}L_k(\lambda))\simeq \delta_{n,a(h^\vee-1)}S_{a}\mathbf{L}_k(\lambda)$ 
and 
$H_f^n(L_k(\mu))\simeq \delta_{n,0}\mathbf{L}_k(\mu)$.
Therefore, $n=0$ and $\lambda=\mu$ hold by Theorem \ref{classification of simple modules}. 
\endproof

\begin{lemma}\label{classification of simple currents}
The $\sWsub$-module $\mathbf{L}_k(\lambda)$ $(\lambda\in \Pr_\Z^k)$ is a simple current iff
\begin{align*}
    \lambda=\begin{cases}
        0, (p-h^\vee)\varpi_1 & (q=2n-1),\\
        0, (p-h)\varpi_n & (q=2n).
    \end{cases}
\end{align*}
\end{lemma}

\proof
Since the category of $\W_k^+$-modules is a modular tensor category, the simple module $\mathbf{L}_k(\lambda)$ is a simple current iff its quantum dimension is one. By \cite[Proposition 4.10]{AvEM}, it is given by
\begin{align}\label{qdim for coprincial case}
    D(\lambda)=\prod_{\alpha\in \Delta_+}\frac{\sin(\tfrac{(\lambda+\rho,\alpha)}p\pi)}{\sin(\tfrac{(\rho,\alpha)}p\pi)}\quad (q=2n-1),\quad\prod_{\alpha\in \Delta_+}\frac{\sin(\tfrac{(\lambda+\rho,\check{\alpha})}p\pi)}{\sin(\tfrac{(\rho,\check{\alpha})}p\pi)}\quad (q=2n).
\end{align}
The proof then reduces to the study of the function $D(\lambda)$.
In the case $q=2n-1$, $D(\lambda)$ agrees with the quantum dimension of the module $L_{p-h^\vee}(\lambda)$ of the simple affine vertex algebra $L_{p-h\vee}(\g)$. Hence, the assertion follows from Proposition~\ref{simple current action}. For $q=2n$, we analyze $D(\lambda)$ with classic tools of analysis, adapting an argument of \cite{Fu}. Details are omitted. 
\endproof

\begin{theorem}\label{branching rule}
Let $k=-h^\vee+\frac{p}{q}$ be an admissible level with $q=2n-1, 2n$.
\begin{enumerate}[wide, labelindent=0pt, font=\normalfont]
\item There is an isomorphism of $\sWsub$-modules
$$S_{\bar{q}} \mathbf{L}_k(0) \simeq \mathbf{L}_k(\sigma\circ0)$$
where $\sigma$ is defined by \eqref{Dynkin automorphisms}.
Moreover, the spectral flow twist $S_{J^+}$ has periodicity $2\bar{q}$.\\
\item Set
\begin{align*}
\mathcal{P}_{+}=\begin{cases}
    2\sqrt{(p-q)q}\Z & (q=2n-1),\\
    \sqrt{(p-q)q}\Z & (q=2n),
\end{cases}
\quad 
M=\begin{cases}
    \Z_{2(p-q)} & (q=2n-1),\\
    \Z_{(p-q)} & (q=2n).
\end{cases}
\end{align*}
Then we have an isomorphism of $\mathrm{Com}(\pi^{J^+},\ssubW)\otimes  \pi^{J^+}$-modules
\begin{align*}
\sWsub\simeq 
\bigoplus_{a\in M}\mathscr{C}_{k,a}^{D^+}\otimes V_{\sqrt{\frac{q}{p-q}}a+\mathcal{P}_+}.
\end{align*}
\item The fusion ring of $\ssubW\Mod$ decomposes into
\begin{align*}
\mathcal{K}(\sWsub)\simeq \mathcal{K}^0(\sWsub)\underset{\Z[\Z_2]}{\otimes}\Z[\Z_{2\bar{q}}]
\end{align*}
where $\mathcal{K}^0(\sWsub)$ is spanned by $\{\mathbf{L}_k(\lambda),\,\lambda\in\Pr_\Z^k\}$.
\end{enumerate}
\end{theorem}

\noindent
For $\lambda,\mu\in\Pr^k_\Z$, we write the fusion rules for  $\mathbf{L}_k(\lambda),\mathbf{L}_k(\mu)$ in $\mathcal{K}^0(\ssubW)$ as 
\begin{equation}\label{eq:coeff_tensor_product_subreg}
    \mathbf{L}_k(\lambda)\boxtimes\mathbf{L}_k(\mu)\simeq\bigoplus_{\nu\in\Pr^k_{\Z}}N^\nu_{\lambda,\mu}\mathbf{L}_k(\nu).
\end{equation}

\proof
(1) By Theorem \ref{classification of simple modules}, $S_{\bar{q}}\mathbf{L}_k(0)\simeq\mathbf{L}_k(\lambda)$ for a certain $\lambda\in\Pr_\Z^k$. Moreover, since spectral flow twists preserve simple currents, it follows from Lemma~\ref{classification of simple currents} that 
\begin{equation*}
    \lambda\in\begin{cases}
        \{0,(p-h^\vee)\varpi_1\}&(q=2n-1),\\
        \{0,(p-h)\varpi_n\}&(q=2n).
    \end{cases}
\end{equation*}
Hence, it suffices to show $S_{\bar{q}}\mathbf{L}_k(0)\not\simeq \mathbf{L}_k(0)$. 
This can be done by comparing the characters.
If $S_{\bar{q}}\mathbf{L}_k(0)\simeq \mathbf{L}_k(0)$ would hold, then their characters would coincide, that is, 
\begin{align*}
    \ch_{\mathbf{L}_k(0)}(z,q)=q^{\half\frac{p-q}{q}\bar{q}^2}z^{\frac{p-q}{q}\bar{q}}\ch_{\mathbf{L}_k(0)}(zq^{\bar{q}},q)
\end{align*}
by \eqref{effect of spectral flow twists}.
For $q=2n$, it is a contradiction as the powers of $z$ lie in $\Z$ in the left-hand side and in $\half+\Z$ in the right-hand side.
For $q=2n-1$, it is a contradiction again as it would imply that $\ch_{\mathbf{L}_k(0)}(z,q)$ contains the term $z^{-(p-q)}q^{\half q(p-q)}$. 
It is impossible since the lowest conformal weight among the vectors with $J^+_0$-eigenvalue $-(p-q)$ is realized by $(G^-_{-n})^{p-q}\mathbf{1}$ in $\subW$ and is $n(p-q)=\half(q+1)(p-q)>\half q(p-q)$.
This concludes the first part of the assertion.
The second follows from Proposition~\ref{spectral flow twist and IO} as
$$S_{2\bar{q}} \mathbf{L}_k(\lambda) \simeq S_{2\bar{q}}(\mathbf{L}_{k}(0)\boxtimes \mathbf{L}_{k}(\lambda))\simeq S_{2\bar{q}}\mathbf{L}_{k}(0)\boxtimes \mathbf{L}_{k}(\lambda)\simeq \mathbf{L}_{k}(\lambda).$$
(2) We have $S_\theta \ssubW\simeq \ssubW$ iff $\theta\in 2\bar{q}\Z$ by (1). Since we have 
$$\ssubW\simeq \bigoplus_{a\in \Z}\mathscr{C}_{k,a}^{D^+}\otimes \pi^x_{\sqrt{\frac{q}{p-q}}ax},$$
the assertion follows from Proposition \ref{period of spectral flow}.
(3) For $\lambda,\mu\in\Pr_\Z^k$, the weights of $\mathbf{L}_k(\lambda)\boxtimes\mathbf{L}_k(\mu)$ are in $\tfrac{1}{2}\Z$. Comparing with the weights of the irreducible $\ssubW$-modules forces the simple modules appearing in the decompositions of the fusion product to be of $\mathbf{L}_k(\nu)$, $\nu\in\Pr_\Z^k$. Hence, $\mathcal{K}^0(\ssubW)$ is a stable fusion subring of $\mathcal{K}(\ssubW)$.
In addition, 
\begin{align}\label{eq:symmetry}
    S_{\bar{q}}\mathbf{L}_k(\lambda)\simeq \mathbf{L}_k(\varsigma\circ\lambda),
\end{align}
for a certain involution $\varsigma$ such that $\varsigma\circ\lambda\in\Pr_\Z^k$.
The decomposition of the fusion ring follows from the periodicity of the spectral flow twist.
\endproof

\begin{conjecture}
For all $\lambda\in \Pr_\Z^k$ and $0\leq\theta\leq\bar{q}$, there exists $\mu\in\h^*$ such that 
$S_{\theta} \mathbf{L}_k(\lambda)\simeq H^0_f(L_k(\mu))$
as $\sWsub$-modules.
Moreover, we conjecture that the involution $\varsigma$ defined in \eqref{eq:symmetry} is $\sigma$, that is
$$S_{\bar{q}} \mathbf{L}_k(\lambda) \simeq \mathbf{L}_k(\sigma\circ\lambda).$$
\end{conjecture}

When $k$ is principal exceptional, i.e. $q=2n-1$, the identification $\Pr^k_{\Z}\simeq\rP_+^{p-h^\vee}$ as well as Example~\ref{exple} below support the following conjecture.
\begin{conjecture}\label{conj:fusion_rules}
    When $k=-h^\vee+\tfrac{p}{2n-1}$ is an exceptional level,
	$$ \mathcal{K}(L_{p-h^\vee}(\so_{2n+1}))\xrightarrow{\simeq}\mathcal{K}^0(\sWsub),\quad L_{p-h^\vee}(\lambda)\mapsto \mathbf{L}_k(\lambda).$$
\end{conjecture}

\begin{example}\label{exple}
The fusion ring of the subregular $\W$-algebra $\ssubW$ for $k=-h^\vee_++\frac{2n}{2n-1}$, i.e. $(p,q)=(2n,2n-1)$, can be computed through the decomposition 
    \begin{align}\label{SCE for fusion}
\W_{-h^\vee+\frac{p}{q}}^{D_+}\simeq L(\tfrac{1}{2},0)\otimes V_{2\sqrt{q}\Z}\oplus L(\tfrac{1}{2},\tfrac{1}{2})\otimes V_{\sqrt{q}+2\sqrt{q}\Z},
\end{align}
which is a simple current extension of $V_{2\sqrt{q}\Z}$ times the Virasoro minimal model $L(\tfrac{1}{2},0)$ of central charge $\tfrac{1}{2}$, see Theorem~\ref{unitary case1}.
The simple $L(\half,0)$-modules are $L(\half,h)$ ($h=0,\tfrac{1}{16}, \tfrac{1}{2}$). The fusion rules can be computed directly or found in \cite{dFMS}:
\begin{align*}
    \mathcal{K}(L(\tfrac{1}{2},0))&\simeq\mathcal{K}(L_{1}(\so_{2n+1})),\\ 
    L(\tfrac{1}{2},0)\mapsto L_1(0),\quad L(\tfrac{1}{2},\tfrac{1}{16})&\mapsto L_1(\varpi_n),\quad L(\tfrac{1}{2},\tfrac{1}{2})\mapsto L_1(\varpi_1).
\end{align*}
More precisely, the fusion table is as follows.
 	\begin{equation*}
				\begin{array}{c|ccc}
					\boxtimes&0&\varpi_n&\varpi_1\\\hline
					0&0&\varpi_n&\varpi_1\\
					\varpi_n&\varpi_n&0\oplus\varpi_1&\varpi_n\\
					\varpi_1&\varpi_1&\varpi_n&0
				\end{array}
			\end{equation*}
Let $N=2\sqrt{q}$. The simple $V_{N\Z}$-modules are $V_{a/N+N\Z}$ $(a\in \Z_{N^2})$ and satisfy the fusion rules $$V_{a/N+N\Z}\boxtimes V_{b/N+N\Z}\simeq V_{(a+b)/N+N\Z}.$$
Through the induction functor associated with the extension \eqref{SCE for fusion}, the simple modules for $\W_{-h^\vee_++\frac{2n}{2n-1}}^{D^+}$ are given by the image of the modules $M(a,b):=L(\tfrac{1}{2},a)\otimes V_{b/N+N\Z}$ which are local to the simple current $M(\tfrac{1}{2},2q)=L(\tfrac{1}{2},\tfrac{1}{2})\otimes V_{\sqrt{q}+2\sqrt{q}\Z}$. More explicitly, they consist of the equivalent classes 
$$M(0,b),\ M(\tfrac{1}{2},b)\quad (b\equiv 0),\quad M(\tfrac{1}{16},b)\quad (b\equiv1),\qquad (\text{mod }2)$$
with $0\leq b<4q$ subject to the relations 
$M(\tfrac{1}{2},2q)\boxtimes M(a,b)\sim M(a,b)$.
Identifying the simple $L_1(\so_{2n+1})$-modules as 
$$L_1(0)\leftrightarrow M(0,0),\quad L_1(\varpi_n)\leftrightarrow M(\tfrac{1}{16},q), \quad L_1(\varpi_1)\leftrightarrow M(\tfrac{1}{2},0),$$
and $S_{\theta}$ ($0\leq\theta<2q$) with the fusion product by $M(0,2\theta)$,
we obtain the isomorphism 
\begin{align*}
\begin{array}{ccc}
\mathcal{K}\left(\W_{-h^\vee_++\frac{2n}{2n-1}}^{D^+}\right)&\xrightarrow{\simeq}& \mathcal{K}\left(L_1(\so_{2n+1}) \right)\underset{\Z[\Z_2]}{\otimes}\Z[\Z_{2q}]\\
S_{\theta}\mathbf{L}_k(a)&\mapsto& L_1(a)\otimes [\theta]\ (a=0,\varpi_1,\varpi_n).
\end{array}
\end{align*}
\end{example}

Finally, let $\mathrm{KL}_k(\g)$ denote the category of $L_k(\g)$-modules generated by $L_k(\lambda)$ ($\lambda \in \Pr_\Z^k$) and $\mathrm{KL}_{k,f}(\g)$ the category of $\ssubW$-modules generated by $\mathbf{L}_k(\lambda)$ ($\lambda \in \Pr_\Z^k$). They are braided tensor categories by \cite{CHY} and \cite{H,HL} respectively. Note that the functor 
\begin{align*}
    H_f^0\colon \mathrm{KL}_k(\g)\rightarrow \mathrm{KL}_{k,f}(\g),\quad L_k(\lambda)\mapsto \mathbf{L}_k(\lambda)
\end{align*}
gives an equivalence of abelian categories.

\begin{conjecture}\label{BTC equiv}
The functor 
$$H_f^0\colon \mathrm{KL}_k(\g)\rightarrow \mathrm{KL}_{k,f}(\g)$$
is an equivalence of braided tensor categories.
\end{conjecture}

\begin{remark}
Conjecture \ref{BTC equiv} is expected to hold as long as $H_f^0$ is an equivalence of abelian categories, see \cite[Theorem 10.4]{ACF} for $\g$ being simply-laced under certain conditions.
In general, $\mathrm{KL}_k(\g)$ for principal admissible levels $k=-h^\vee+\frac{p}{q}$ is expected to have the same fusion rings as $\mathrm{KL}_{p-h^\vee}(\g)$ (i.e. setting $q=1$) \cite{C19}, Conjecture \ref{BTC equiv} can be regarded as a stronger version of Conjecture \ref{conj:fusion_rules}.
\end{remark}

\subsection{Unitary case}\label{sec:unitarity}
A vertex operator algebra $V$ is said \emph{unitary} if it admits a positive-definite invariant Hermitian form \cite{DL1}. 
In this case, $V$ is also unitary as a module over the Virasoro subalgebra generated by its conformal vector. In particular, the central charge is a positive real number.

We consider the exceptional levels $k$ when $\ssubW$ is unitary.
It follows from the decomposition
\begin{align}\label{lattice decomposition}
\ssubW\simeq 
\bigoplus_{a\in M}\mathscr{C}_k^+(a)\otimes V_{\sqrt{\frac{q}{p-q}}a+\mathcal{P}_+}
\end{align}
in Theorem \ref{branching rule} (2) and the unitarity of $V_L$ \cite{DL1} that $\ssubW$ is unitary only if $\mathscr{C}_k^+(0)$ is. 
Hence, the central charge of $\mathscr{C}_k^+(0)$ is positive real. This happens iff $(n,k)$ satisfies one of the following:
\begin{align*}
    \text{(i) } k+h^\vee=\frac{2n}{2n-1}, \frac{2n+1}{2n-1}, \frac{2n+1}{2n}\quad (n\geq 2),\qquad \text{(ii) } k+h^\vee=\frac{7}{4}\quad (n=2).
\end{align*}

To describe \eqref{lattice decomposition} for the case (i) more precisely, we use the simple Virasoro vertex algebra $L(c,0)$ of central charge $c$ and the $\Z_2$-orbifold $V_{\sqrt{2p}\Z}^+$ of the lattice vertex algebra $V_{L}$ for the lattice $L=\Z\alpha=\sqrt{2p}\Z$.
We denote by $L(c,h)$ the simple $L(c,0)$-module of highest weight $h$, and by $V_{\Z\alpha}^\pm$ and $V_{\alpha/2+\Z\alpha}^\pm$ the simple $V_{\Z\alpha}^+$-modules, see \cite{DN} for the complete classification of the simple $V_{\Z\alpha}^+$-modules. By \cite{Ab}, 
$\{V_{\Z\alpha}^\pm,V_{\alpha/2+\Z\alpha}^\pm\}$ 
are simple currents satisfying
\begin{align}\label{fusion of lattice orbifold}
(V_{\alpha/2+\Z\alpha}^+)^{\boxtimes n}\simeq 
\begin{array}{cc}
\begin{cases}
V_{\Z\alpha}^+\ (n\equiv 0), & V_{\alpha/2+\Z\alpha}^+\ (n\equiv 1),\\
V_{\Z\alpha}^-\ (n\equiv 2), & V_{\alpha/2+\Z\alpha}^-\ (n\equiv 3),
\end{cases}
\end{array}
\quad \text{mod } 4.
\end{align}

\begin{theorem}\label{unitary case1}
We have the following decomposition.
\begin{enumerate}[wide, labelindent=0pt, font=\normalfont]
\item $(p,q)=(2n,2n-1)$
\begin{align*}
\W_{-h^\vee+\frac{p}{q}}^{D_+}(n,1)\simeq L(\tfrac{1}{2},0)\otimes V_{2\sqrt{q}\Z}\oplus L(\tfrac{1}{2},\tfrac{1}{2})\otimes V_{\sqrt{q}+2\sqrt{q}\Z}.
\end{align*}
\item $(p,q)=(2n+1,2n-1)$
\begin{align*}
\W_{-h^\vee+\frac{p}{q}}^{D_+}(n,1)\simeq 
&\ V^+_{\sqrt{2p}\Z}\otimes V_{2\sqrt{2q}\Z}\oplus V^-_{\sqrt{2p}\Z}\otimes V_{\sqrt{2q}+2\sqrt{2q}\Z}\\
&\hspace{2mm}\oplus V^+_{\sqrt{\frac{p}{2}}+\sqrt{2p}\Z}\otimes V_{\sqrt{\frac{q}{2}}+ 2\sqrt{2q}\Z}\oplus V^-_{\sqrt{\frac{p}{2}}+\sqrt{2p}\Z}\otimes V_{-\sqrt{\frac{q}{2}}+2\sqrt{2q}\Z}.
\end{align*}
\item $(p,q)=(2n+1,2n)$
\begin{align*}
\W_{-h^\vee+\frac{p}{q}}^{D_+}(n,1)\simeq V_{\sqrt{q}\Z}.
\end{align*}
\end{enumerate}
\end{theorem}

The description of \eqref{lattice decomposition} for the case (ii) involves the simple current extension of $L(\tfrac{6}{7},0)$
$$\widehat{L}(\tfrac{6}{7},0):= L(\tfrac{6}{7},0)\oplus L(\tfrac{6}{7},5).$$
The classification of its simple modules and their fusion rules are obtained in \cite{LLY}.
In particular, $\widehat{L}(\tfrac{6}{7},0)$ acts on $L(\tfrac{6}{7},\frac{4}{3})$, giving two non-isomorphic simple module structures, which we denote by
$$\widehat{L}(\tfrac{6}{7},\tfrac{4}{3})^{+},\quad \widehat{L}(\tfrac{6}{7},\tfrac{4}{3})^{-}.$$
They are simple currents contragredient to each other and satisfy
\begin{align}\label{fusion of extended vir}
\widehat{L}(\tfrac{6}{7},\tfrac{4}{3})^{\pm}\boxtimes \widehat{L}(\tfrac{6}{7},\tfrac{4}{3})^{\pm}\simeq \widehat{L}(\tfrac{6}{7},\tfrac{4}{3})^{\mp},
\quad \widehat{L}(\tfrac{6}{7},\tfrac{4}{3})^{\pm}\boxtimes \widehat{L}(\tfrac{6}{7},\tfrac{4}{3})^{\mp}\simeq \widehat{L}(\tfrac{6}{7},0).
\end{align}

\begin{theorem}\label{unitary case2}
For $n=2$, $(p,q)=(7,4)$, we have the following decomposition
\begin{align*}
\W_{-5/4}^{D_+}(2,1)\simeq \widehat{L}(\tfrac{6}{7},0) \otimes V_{2\sqrt{3}\Z} \oplus \widehat{L}(\tfrac{6}{7},\tfrac{4}{3})^+\otimes V_{\frac{2}{\sqrt{3}}+2\sqrt{3}\Z} \oplus \widehat{L}(\tfrac{6}{7},\tfrac{4}{3})^- \otimes V_{-\frac{2}{\sqrt{3}}+2\sqrt{3}\Z}.
\end{align*}
\end{theorem} 

\begin{proof}[Proof of Theorem \ref{unitary case1}]
In order to identify the coset $\mathscr{C}=\mathscr{C}_k^+(0)$, we use the asymptotic datum of $V=\ssubW$, i.e. the asymptotic behavior of the normalized character 
\begin{align}\label{asymptotic datum}
\chi_{V}(\tau)\sim \mathbf{A}_V (-\mathbf{i}\tau)^{\frac{\mathbf{w}_V}{2}} \mathrm{e}^{\frac{\pi \mathbf{i}}{12 \tau} \mathbf{g}_V}\quad (\tau \downarrow 0)
\end{align}
where $\tau$ is the coordinate of the upper half plane with $q=\mathrm{e}^{2\pi \mathbf{i}\tau}$. 
(1)
By \cite[Proposition 4.10]{AvEM}, we have
$\mathbf{g}_V=3/2$. 
Under the decomposition \eqref{lattice decomposition}, $\mathbf{g}_V$ takes the maximum of the asymptotic data of the components appearing in the decomposition \eqref{lattice decomposition}, we have 
$\mathbf{g}_{\mathscr{C}}\leq \mathbf{g}_V-\mathbf{g}_{V_{\lambda+\mathcal{P}_+}}=\mathbf{g}_V-1=1/2$.
It follows that $\mathscr{C}$ is an extension of the Virasoro unitary minimal model $L(\half,0)$ by \cite[Lemma 2.8]{AvEM}.
Since $\mathscr{C}$ is $\Z$-graded, the classification of $L(\half,0)$-modules implies that $\mathscr{C}=L(\half,0)$ and then the assertion. 
(3) Since $k=-h^\vee+\tfrac{2n+1}{2n}$, we have $\mathbf{g}_\mathscr{C}=0$. Hence, $\mathscr{C}$ is trivial and thus the assertion follows.
(2) The decomposition \eqref{lattice decomposition} is 
\begin{align*}
\ssubW=\mathbf{L}_k(0)\simeq 
 \mathscr{C}\otimes V_{2\sqrt{2q}\Z}\oplus \mathscr{C}_1\otimes V_{\sqrt{\frac{q}{2}}+ 2\sqrt{2q}\Z} \oplus\mathscr{C}_2\otimes V_{\sqrt{2q}+2\sqrt{2q}\Z}\oplus \mathscr{C}_3\otimes V_{-\sqrt{\frac{q}{2}}+2\sqrt{2q}\Z}.
\end{align*} 
By \eqref{sub strong generating type}, $\mathscr{C}_{i}$ has conformal dimension $1$ for $i=2$ and $\frac{p}{2}$ for $i=1,3$ and its top space is one dimensional.
On the other hand, since $S_{\bar{q}}\mathbf{L}_k(0)=\mathbf{L}_k(2\varpi_1)$ is a simple current of conformal dimension $1$, the direct sum $\mathbf{L}_k(0)\oplus \mathbf{L}_k(2\varpi_1)$ has a structure of simple vertex algebra, which decomposes into 
\begin{align}\label{decomposition}
\mathbf{L}_k(0)\oplus \mathbf{L}_k(2\varpi_1)\simeq 
&\ (\mathscr{C}\oplus \mathscr{C}_2)\otimes V_{\sqrt{2q}\Z}\oplus(\mathscr{C}_1\oplus \mathscr{C}_3)\otimes V_{\sqrt{\frac{q}{2}}+\sqrt{2q}\Z}.
\end{align}
The $\Z_2$-orbifold $\widetilde{\mathscr{C}}:=\mathscr{C}\oplus \mathscr{C}_2$ is also a simple vertex algebra, which is of CFT type, and its weight one subspace is one-dimensional. Therefore, by the proof of \cite[Theorem 1.1]{DM}, it generates a rank one simple Heisenberg vertex algebra, say $\pi_{\widetilde{\mathscr{C}}}$, such that $ \widetilde{\mathscr{C}}$ is completely reducible as a $\pi_{\widetilde{\mathscr{C}}}$-module. Since $\pi_{\widetilde{\mathscr{C}}}\hookrightarrow \widetilde{\mathscr{C}}$ is a conformal embedding, it follows that $\widetilde{\mathscr{C}}\simeq V_{\sqrt{2N}\Z}$ for some even $N$. Therefore, \eqref{decomposition} is actually
\begin{align}\label{decomposition2}
\mathbf{L}_k(0)\oplus \mathbf{L}_k(2\varpi_1)\simeq 
&V_{\sqrt{2N}\Z} \otimes V_{\sqrt{2q}\Z}\oplus V_{\sqrt{N/2}\alpha+\sqrt{2N}\Z} \otimes V_{\sqrt{\frac{q}{2}}+\sqrt{2q}\Z}.
\end{align}
By \cite[Proposition 4.10]{AvEM}, we have
\begin{align*}
&\chi_{\mathbf{L}_k(0)}(\tau), \chi_{\mathbf{L}_k(2\varpi_1)}(\tau) \sim\tfrac{1}{2\sqrt{pq}} \mathrm{e}^{2\frac{\pi \mathbf{i}}{12 \tau}}.
\end{align*}
Comparing the asymptotic data in \eqref{decomposition2}, we obtain $N=p$. Since the nontrivial $\Z_2$-automorphism on $V_{\sqrt{2N}\Z}$ preserving the conformal weight is unique, we conclude
$$\mathscr{C}=V^+_{\sqrt{2p}\Z},\quad \mathscr{C}_2=V^-_{\sqrt{2p}\Z},\quad \mathscr{C}_1= V^+_{\sqrt{\frac{p}{2}}+\sqrt{2p}\Z},\quad \mathscr{C}_3= V^-_{\sqrt{\frac{p}{2}}+\sqrt{2p}\Z}$$
and thus the assertion.
\end{proof}

\begin{proof}[Proof of Theorem \ref{unitary case2}]
The decomposition as a $L(\frac{6}{7},0)\otimes V_{2\sqrt{3}\Z}$-module is obtained in the same way as in the proof of Theorem \ref{unitary case2}~(1).
By \cite[Proposition 4.10]{AvEM}, we have $\mathbf{g}_V=13/7$ and thus $\mathbf{g}_\mathscr{C}=6/7$.
As $\W_{-5/4}^{D_+}(2,1)$ is $\Z$-graded and of type at most $\W(1,2,(2)^2)$, $\W_{-5/4}^{D_+}(2,1)$ must decompose into the form
\begin{align*}
\W_{-5/4}^{D_+}(2,1)\simeq 
(L(\tfrac{6}{7},0)\oplus L(\tfrac{6}{7},5)^{\oplus a})\otimes V_{2\sqrt{3}\Z}
\oplus L(\tfrac{6}{7},\tfrac{4}{3})\otimes (V_{\frac{2}{\sqrt{3}}+2\sqrt{3}\Z} \oplus V_{\frac{4}{\sqrt{3}}+2\sqrt{3}\Z})
\end{align*}
for some $a$. Then $a=1$ follows from the asymptotic datum 
\begin{align*}
\begin{split}
&\chi_{V}(\tau)\sim\tfrac{1}{\sqrt{7}} \mathrm{sin}(\tfrac{\pi}{7}) \mathrm{e}^{\frac{\pi \mathbf{i}}{12 \tau} \frac{13}{7}},\quad \chi_{V_{\frac{p}{\sqrt{2m}}\Z}+\sqrt{2m}\Z}(\tau)\sim \tfrac{1}{\sqrt{2m}}\mathrm{e}^{\frac{\pi \mathbf{i}}{12 \tau}},\quad (p\in \Z_{2m}),\\
&\chi_{L(\frac{6}{7},0)}(\tau),\ \chi_{L(\frac{6}{7},5)}(\tau)\sim\tfrac{1}{\sqrt{21}} \mathrm{sin}(\tfrac{\pi}{7}) \mathrm{e}^{\frac{\pi \mathbf{i}}{12 \tau} \frac{6}{7}},\quad \chi_{L(\frac{6}{7},\frac{4}{3})}(\tau)\sim\tfrac{2}{\sqrt{21}} \mathrm{sin}(\tfrac{\pi}{7}) \mathrm{e}^{\frac{\pi \mathbf{i}}{12 \tau} \frac{6}{7}}.
\end{split}
\end{align*}
Since $\W_{-5/4}^{D_+}(2,1)$ is simple, so is the coset $\mathscr{C}$.
Therefore, $\mathscr{C}\simeq \widehat{L}(\tfrac{6}{7},0)$ as vertex algebras by the uniqueness of simple current extensions. 
It follows that the other multiplicity spaces $L(\tfrac{6}{7},\tfrac{4}{3})$ are $\widehat{L}(\tfrac{6}{7},\tfrac{4}{3})^\pm$ as $\widehat{L}(\tfrac{6}{7},0)$-modules. 
Since $\W_{-5/4}^{D_+}(2,1)$ is self-dual and $\widehat{L}(\tfrac{6}{7},\tfrac{4}{3})^\pm$ are contragredient to each other, both $\widehat{L}(\tfrac{6}{7},\tfrac{4}{3})^\pm$ must appear in the decomposition. Replacing $J^+$ with $-J^+$ if necessary, we obtain the desired decomposition. 
\end{proof} 

From the previous decompositions, we obtain the following observations.

\begin{corollary}\label{funny simple current extensions}\hspace{0mm}
\begin{enumerate}[wide, labelindent=0pt, font=\normalfont]
\item For $k=-h^\vee+\tfrac{2n}{2n-1}$, the simple current extension $\sWsub\oplus S_{q}\sWsub$ is isomorphic to a vertex superalgebra:
$$\sWsub\oplus S_{q}\sWsub\simeq \mathscr{F}\otimes V_{\sqrt{q}\Z}$$ 
where  $\mathscr{F}$ is the free fermion vertex superalgebra. Moreover,
\begin{align*}
\sWsub\simeq (\mathscr{F}\otimes V_{\sqrt{q}\Z})_{\bar{0}},\quad S_{q}\sWsub\simeq (\mathscr{F}\otimes V_{\sqrt{q}\Z})_{\bar{1}}
\end{align*}
where  $(\cdot)_{\bar{i}}$ for $i=0,1$ denote the even and odd subspaces.\\
\item For $k=-h^\vee+\tfrac{2n+1}{2n-1}$, $\sWsub$ is a simple current extension of $V^+_{\sqrt{2p}\Z}\otimes V_{2\sqrt{2q}\Z}$ of order four. Moreover, the simple current extension $\sWsub\oplus S_{q}\sWsub$ is isomorphic to a lattice vertex algebra:
\begin{align*}
\sWsub\oplus S_{q}\sWsub\simeq V_{\sqrt{2p}\Z}\otimes V_{\sqrt{2q}\Z}\oplus V_{\sqrt{\frac{p}{2}}+\sqrt{2p}\Z}\otimes V_{\sqrt{\frac{q}{2}}+\sqrt{2q}\Z}.
\end{align*}
\item $\W_{-5/4}^{D_+}(2,1)$ is a simple current extension of $\widehat{L}(\tfrac{6}{7},0) \otimes V_{2\sqrt{3}\Z}$ of order three.
\end{enumerate}
\end{corollary}

\proof
(1) Recall that $\mathscr{F}$ is generated by the odd field $\phi(z)$ of conformal weight $1/2$ satisfying the OPE $\phi(z)\phi(w)\sim 1/(z-w)$ and the conformal vector $\half\partial \phi\cdot \phi$ has central charge $\half$. 
The asymptotic data $\chi_\mathscr{F}(\tau)\sim e^{\pi \mathbf{i}/24 \tau}$ gives the decomposition 
$\mathscr{F}\simeq L(\tfrac{1}{2},0)\oplus L(\tfrac{1}{2},\tfrac{1}{2})$, which is a simple current (super) extension. Then the assertion follows from Theorem \ref{unitary case1} (1).
For (2), the first part is clear from \eqref{fusion of lattice orbifold} and the latter part is clear from \eqref{decomposition2}. 
(3) is clear from Theorem~\ref{unitary case2}.
\endproof

\begin{theorem}
The simple $\W$-algebra $\sWsub$ is
unitary if $(k,n)$ satisfies one of the following
$$k+h^\vee=\tfrac{2n}{2n-1},\ \tfrac{2n+1}{2n}\quad (n\geq 2),\qquad k+h^\vee=\tfrac{7}{4}\quad (n=2).$$
\end{theorem}

\proof
The lattice vertex superalgebras associated with non-degenerate integral lattices are unitary \cite{AL,DL1}, and so is the free fermion vertex algebra $\mathscr{F}$ \cite{AL}.
Then, the cases $k+h^\vee=\tfrac{2n}{2n-1},\tfrac{2n+1}{2n}$ follow from Corollary \ref{funny simple current extensions} (1) and Theorem \ref{unitary case1} (3), respectively. Next, we consider the case $k+h^\vee=\tfrac{7}{4}$, $(n=2)$. 
Let $V_M$ be the lattice vertex algebra associated with the positive-definite even lattice 
$$M=\left\{(0,0), \pm\left(\sqrt{2}\gamma,\tfrac{2}{\sqrt{3}}\right)\right\} +\left(\sqrt{2}E_6 \oplus 2\sqrt{3}\Z\right)$$
with $\gamma$ a certain representative in $E_6^*/E_6$ defined in \cite{LY}.
By \cite{LY}, we have an embedding 
$\bigotimes_{i=1}^{6} L(c_i,0) \hookrightarrow V_M$ where $c_1,\ldots,c_6$ are $\half, \frac{7}{10}, \frac{4}{5}, \frac{6}{7}, \frac{25}{28}, \frac{39}{28}$.
By the uniqueness of simple current extensions the coset $\Com( \bigotimes_{i=1}^{6} L(c_i,0), V_M)$ is isomorphic to $\W_{-5/4}^{D_+}(2,1)$ thanks to \cite[Theorem 4.1]{LY}.
The automorphism associated with the unitary structure on $V_L$ preserves $\bigotimes_{i=1}^{6} L(c_i,0)$ and thus the coset is also unitary \cite[Remark 2.3]{LY}.
Therefore, $\W_{-5/4}^{D_+}(2,1)$ is unitary.
\endproof

\begin{remark}
For the remaining case $\sWsub$ with $k+h^\vee=\tfrac{2n+1}{2n-1}$, one can show that the subalgebra 
$V:=\ V^+_{\sqrt{2p}\Z}\otimes V_{2\sqrt{2q}\Z}\oplus V^-_{\sqrt{2p}\Z}\otimes V_{\sqrt{2q}+2\sqrt{2q}\Z}$ is unitary by restricting the unitary structure on $V_{\sqrt{2p}\Z}\otimes V_{\sqrt{2q}\Z}$ and that $V^+_{\sqrt{\frac{p}{2}}+\sqrt{2p}\Z}\otimes V_{\sqrt{\frac{q}{2}}+ 2\sqrt{2q}\Z}\oplus V^-_{\sqrt{\frac{p}{2}}+\sqrt{2p}\Z}\otimes V_{-\sqrt{\frac{q}{2}}+2\sqrt{2q}\Z}$ is a unitary $V$-module. However, these unitary structures do not extend to $\sWsub$.
\end{remark}

\section{Rational case: the principal $\W$-superalgebra $\W_\ell(\osp_{2|2n},\fpr)$}\label{sec:rational_principal}

We now transport our rationality results to the principal $\W$-superalgebra $\ssprW$ using the Feigin--Semikhatov duality.

\subsection{Coset construction}
The Kazama--Suzuki type coset construction (Theorem \ref{Coset theorem}) preserves the rationality and lisse condition of the $\W$-superalgebras \cite[Corollary 5.19]{CGN}.
In particular, $\ssprW$ is rational and lisse at levels 
\begin{equation*}
    \ell+h^\vee_-=\frac{1}{2(k+h^\vee_+)}=\frac{q}{2p},\quad (p,q)=1,
    \begin{cases}
        p\geq 2n-1, & q=2n-1,\\
        p\geq 2n, & q=2n,
    \end{cases}
\end{equation*}
that are the dual levels $\ell$ corresponding to the exceptional levels $k$ studied in Sect.~\ref{sec:subreg_symplectic_walgebra}. 
This was already pointed out in \cite[Section 7.3]{CL2}.
An important byproduct of the coset construction is the decomposition \eqref{decompositoin of the KS} that we recall here:
\begin{align}\label{decompositoin of the KS:2}
\subW\otimes V_{\Z}
\simeq \bigoplus_{a\in \Z} \mathcal{E}(a)\otimes \pi^{J^+_\Delta}_{-a}.
\end{align}
The $\sprW$-modules $\mathcal{E}(a)$ are all obtained from the spectral flow twists of $\sprW$ itself.

\begin{proposition}\label{multiplicity are spectral flow twists}
We have the isomorphisms of $\Wsuper$-modules
\begin{align*}
    \mathcal{E}(a)\simeq S_{a J^-}\Wsuper,\quad (a\in \Z).
\end{align*}
\end{proposition}

\proof
As $S_xV_\Z$ is isomorphic to $V_\Z$s, we have an isomorphism
$$\subW\otimes S_xV_{\Z}\simeq \subW\otimes V_{\Z}$$
as $\subW\otimes V_{\Z}$-modules.
In another direction, since 
$$x=\widehat{J}^--\frac{1}{k+h^\vee_+}J^+_\Delta,$$
we have the isomorphisms of $\sprW\otimes \pi^{J^+_\Delta}$-modules
\begin{align*}
    \Wsub\otimes S_xV_{\Z}
    \simeq \bigoplus_{a\in \Z}S_{\widehat{J}^-}\mathcal{E}(a)\otimes S_{-\frac{1}{k+h^\vee_+}J^+_\Delta}\pi^{J^+_{\Delta}}_{-a}
    \simeq \bigoplus_{a\in \Z}S_{J^-}\mathcal{E}(a)\otimes \pi^{J^+_{\Delta}}_{-a-1}.
\end{align*}
Comparing with the decomposition \eqref{decompositoin of the KS:2}, we obtain $S_{J^-}\mathcal{E}(a)\simeq \mathcal{E}(a+1)$ as $\sprW$-modules. Since $\mathcal{E}(0)\simeq \sprW$, we obtain the statement.
\endproof

Taking the simple quotient of $\subW\otimes V_{\Z}$ amounts to replace $\subW$ with the simple quotient $\ssubW$ and the multiplicity space $\mathcal{E}(a)$ with a simple quotient, say $\mathcal{E}^s(a)$, which is a $\ssprW$-module in the decomposition \eqref{decompositoin of the KS:2} (see e.g. \cite[Proposition 5.1, Corollary 5.6]{CGN}). We show that this decomposition acquires a periodicity in $\mathcal{E}^s(a)$.

Consider first the decomposition of $\ssprW$. We introduce the following abelian groups
\begin{align}\label{list of groups}
    \mathcal{P}_-=\begin{cases}
         2\sqrt{(p-q)p}\Z\\
         \sqrt{(p-q)p}\Z
         \end{cases},\quad 
         L=\begin{cases}
        2\sqrt{pq}\Z\\
        \sqrt{pq}\Z
        \end{cases},\quad 
G=\begin{cases}
        \Z_{2p} & (q=2n-1)\\
        \Z_{p} &(q=2n).
    \end{cases}
\end{align}

\begin{proposition}\label{decomposition for principal super}
	For admissible levels $\ell=-h_-^\vee+\frac{q}{2p}$ with $q=2n-1,2n$, we have the following isomorphism of $\Com(\pi^{J^-},\sWsuper)\otimes\pi^{J^-}$-modules
	\begin{align*}
		\sWsuper\simeq 
		\bigoplus_{a\in M}\mathscr{C}_{\ell,a}^{D^-}\otimes V_{\sqrt{\frac{p}{p-q}}a+\mathcal{P}_-}.
	\end{align*}
\end{proposition}

\proof
By Theorem \ref{branching rule}, the multiplicity spaces $\mathscr{C}^{D^+}_{k,m}$ in Corollary~\ref{isom for the multiplicities} satisfy the isomorphisms $\mathscr{C}^{D^+}_{k,a}\simeq \mathscr{C}^{D^+}_{k,b}$ as $\Com(\pi^{J^+},\ssubW)$-modules iff 
\begin{align*}
    a-b\in 2(p-q)\Z\quad (q=2n-1),\qquad
    a-b\in (p-q)\Z\quad (q=2n).
\end{align*}
Therefore, the coset $\Com(\mathscr{C}^{D^-}_{\ell},\ssprW)$ is isomorphic to 
\begin{align*}
\bigoplus_{a\in 2(p-q)\Z}\pi^{J^-}_{a}\simeq V_{2\sqrt{(p-q)p}\Z},\qquad \bigoplus_{a\in (p-q)\Z}\pi^{J^-}_{a}\simeq V_{\sqrt{(p-q)p}\Z}
\end{align*}
as vertex superalgebras for $q=2n-1$ and $q=2n$, respectively.
The assertion follows.
\endproof

We deduce the decomposition for the simple quotient.

\begin{corollary}\label{decompositon of KS construction}
Let $\ell=-h_-^\vee+\frac{q}{2p}$ be an admissible level with $q=2n-1,2n$.
\begin{enumerate}[wide, labelindent=0pt, font=\normalfont]
\item There is an isomorphism of $\sWsuper$-modules $\mathcal{E}^s(a)\simeq\mathcal{E}^s(b)$ iff 
\begin{align*}
    a-b\in 2p\Z\ (q=2n-1),\quad a-b\in p\Z\ (q=2n).
\end{align*}
\item We have a decomposition 
\begin{align}\label{eq:coset_decomposition}
	\sWsub\otimes V_\Z\simeq
	\bigoplus_{a\in G}\mathcal{E}^s(a)\otimes V_{-\sqrt{\frac{q}{p}}a+L}.
\end{align}
\end{enumerate}
\end{corollary}

\proof
(1) For $a\in\Z$, Proposition~\ref{multiplicity are spectral flow twists} implies that the simple quotient $\mathcal{E}^s(a)$ is isomorphic to $S_{aJ^-}\ssprW$.
The decomposition in Proposition~\ref{decomposition for principal super} together with the Proposition \ref{period of spectral flow} imply the assertion since the period of the spectral flow twist $S_{J^+}$ is given by
\begin{align*}
&\left(\sqrt{\tfrac{p}{p-q}}\Z\right)^*/2\sqrt{(p-q)p}\Z \simeq \sqrt{\tfrac{p-q}{p}}\Z/2\sqrt{(p-q)p}\Z \simeq \Z_{2p} ,\\
&\left(\sqrt{\tfrac{p}{p-q}}\Z\right)^*/\sqrt{(p-q)p}\Z\simeq \sqrt{\tfrac{p-q}{p}}\Z/\sqrt{(p-q)p}\Z \simeq \Z_{p},
\end{align*}
where $\mathcal{L}^*$ denote the dual of integer lattice $\mathcal{L}$. 
(2) For $q=2n-1$ and $q=2n$, the coset ${\Com(\mathcal{E}^s(0),\ssubW\otimes V_\Z)}$ 
is isomorphic to 
\begin{align*}
\bigoplus_{a\in 2p\Z}\pi^{J^+_\Delta}_{-a}\simeq V_{2\sqrt{pq}\Z},\quad \bigoplus_{a\in p\Z}\pi^{J^+_\Delta}_{-a}\simeq V_{\sqrt{pq}\Z}
\end{align*}
as vertex algebras, respectively. The assertion follows from the decomposition~\eqref{decompositoin of the KS:2}.
\endproof

\subsection{Classification of modules}\label{sec:irred_mod_super}
Generalizing Corollary~\ref{decompositon of KS construction}, we construct all the simple $\ssprW$-modules.
For $\g=\so_{2n+1}$, we have $\rP/\rQ\simeq \Z_2$, and set $[\lambda]\in \{0,\half\}$ for $\lambda\in \rP$ accordingly.
Then the $\ssubW$-modules $\mathbf{L}_k(\lambda)$ ($\lambda\in \mathrm{Pr}_\Z^k$), defined in Sect.~\ref{sec:simple_modules_subreg}, have $J^+_0$-eigenvalues $[\lambda]+\Z$.
We decompose $\mathbf{L}_k(\lambda)$ into a sum of $\mathscr{C}_{k}^{D^+}\otimes V_{\mathcal{P}_+}$-modules
 \begin{align}\label{Decomposition of subregular modules}
    \mathbf{L}_k(\lambda)\simeq
\bigoplus_{a\in M}\mathscr{C}_{k,a}^{D^+}(\lambda)\otimes V_{\sqrt{\frac{q}{p-q}}([\lambda]+a)+\mathcal{P}_+}.
\end{align}
For $a\in M$, the $\mathscr{C}_{k}^{D^+}$-module $\mathscr{C}_{k,a}^{D^+}(\lambda)$ is also a $\mathscr{C}_{\ell}^{D^-}$-module, which we denote by $\mathscr{C}_{\ell,a}^{D^-}(\lambda)$.
By Corollary \ref{decompositon of KS construction}, $\mathbf{L}_k(\lambda)\otimes V_\Z$ is a $\ssprW\otimes V_{L}$-module. Let us introduce the simple $\ssprW$-module 
$$\mathbf{L}_\ell^-(\lambda):=    \mathrm{H}_{\varepsilon^{-1}[\lambda]}^+\left(\mathbf{L}_k(\lambda)\right).$$

\begin{proposition}
Suppose $\lambda \in \Pr_\Z^k$.
\begin{enumerate}[wide, labelindent=0pt, font=\normalfont]
\item We have an isomorphism of $\mathscr{C}_\ell^{D^-}\otimes V_L$-modules
\begin{align}\label{decomposition of super modules}
\mathbf{L}_\ell^-(\lambda)\simeq 
\bigoplus_{a\in M}\mathscr{C}_{\ell,a}^{D^-}(\lambda)\otimes V_{\frac{pa+q[\lambda]}{\sqrt{(p-q)p}}+\mathcal{P}_-}.
\end{align}
\item We have an isomorphism of $\sWsuper\otimes V_{L}$-modules
\begin{align}\label{decompositon of Kazama-Suzuki modules}
\mathbf{L}_k(\lambda)\otimes V_\Z \simeq
\bigoplus_{a\in G}S_{a} \mathbf{L}_\ell^-(\lambda)\otimes V_{-\frac{q(a-[\lambda])}{\sqrt{pq}}+L}.
\end{align}
\end{enumerate}
\end{proposition}

\proof 
(1) The isomorphism can be deduced from \eqref{Decomposition of subregular modules} and the definition of $\mathrm{H}_{\varepsilon^{-1}[\lambda]}^+$. 
(2) The decomposition of $\mathbf{L}_k(\lambda)\otimes V_\Z$ as a module over $\mathscr{C}_\ell^{D^-}\otimes V_{\mathcal{P}_-}\otimes V_L$ can be deduced by direct computation from Proposition \ref{decomposition for principal super} and Corollary \ref{decompositon of KS construction}.
The same argument as in the proof of Proposition \ref{multiplicity are spectral flow twists} is used to conclude that \eqref{decompositon of Kazama-Suzuki modules} is an isomorphism of $\ssprW\otimes V_L$-modules.
\endproof

The category $\ssubW\Mod$ is a semisimple balanced braided tensor category (more strongly a modular tensor category) \cite{H,HL} and so is the category $\ssprW\Mod$ by Corollary \ref{BTC bijection} (see also \cite{CKM1}).
The spectral flow twists $\mathcal{E}^s(a)\simeq S_{a}\ssprW$ are simple currents
and thus $\cA^+:=\ssubW\otimes V_\Z$ is a simple current extension of $\cA^-:=\ssprW\otimes V_{L}$ by \eqref{eq:coset_decomposition}.
Such extensions have already been studied and it is known that we can reconstruct the categories $\ssubW\Mod$ and $\ssprW\Mod$ from each other \cite{CKM1,CGNS,YY}.
We briefly recall how this works in our setting, following \cite{CGNS}.

Since the categories $\ssprW\Mod$ and $V_L\Mod$ are semisimple, the module category $\cA^-\Mod$ is naturally equivalent to the Deligne product $\ssprW\Mod\otimes V_L\Mod$.
Similarly, $\cA^+$-mod is equivalent to $\ssubW\Mod\otimes V_\Z\Mod$, but also to $\ssubW\Mod$ as $V_\Z$ is holomorphic.
By regarding $\cA^+$ as a (super)commutative associative algebra object in the braided tensor category $\cA^-\Mod$, we have the induction functor
\begin{align}\label{induction functor}
\begin{array}{cccc}
\ind\colon & \ssprW\Mod\otimes V_L\Mod &\rightarrow & \mathrm{Rep}(\cA^+)\\
& M\otimes N &\mapsto & \cA^+\boxtimes (M\otimes N).
\end{array}
\end{align}
Here $\mathrm{Rep}(\cA^+)$ is the tensor category consisting of module objects for $\cA^+$, i.e. {suitable} pairs $(X,\mu)$ of objects $X$ in $\ssprW\Mod\otimes V_L\Mod$ and homomorphisms
$\mu\colon \cA^+\boxtimes X\rightarrow X$. 
The category $\cA^+\Mod$ is naturally a tensor subcategory of $\mathrm{Rep}(\cA^+)$ consisting of {local modules} \cite{CKM1}.

Recall that if a braided tensor category {$\mathcal{C}$} has a set of simple currents $\mathcal{G}=\{U_g\}_{g\in G}$ parametrized by a finite abelian group $G$, i.e. $U_g\boxtimes U_h\simeq U_{gh}$, then {$\mathcal{C}$} has a grading decomposition by the dual group $\check{G}=\Hom_{\mathrm{group}}(G,\C^*)$:
\begin{align}\label{decompostion of the category}
    \mathcal{C}=\bigoplus_{\phi \in \check{G}} \mathcal{C}_\phi,\quad \mathcal{C}_\phi=\{M\in \mathcal{C}\mid \mathcal{M}_{U_g,M}=\phi(g) \id_{U_g\boxtimes M}\}.
\end{align}
Here $\mathcal{M}_{U_g,M}$ is the monodromy of $U_g$ and $M$. 
When $\mathcal{C}$ is realized as a module category of vertex superalgebras as in our case, the monodromy for a simple module $M$ is expressed as 
\begin{align}\label{formula for monodromy}
    \mathcal{M}_{U_g,M}=\mathrm{exp}(2\pi \ssqrt{-1}(\Delta(U_g\boxtimes M)-\Delta(U_g)-\Delta(M)) \id_{U_g\boxtimes M}
\end{align}
via the conformal dimension $\Delta(M)$ of $M$.
Then the category of local modules for $\mathcal{G}$ is $\mathcal{C}_{\phi=1}$, which we denote by $\mathcal{C}^{\mathcal{G}}$. 
Let us take the following sets of simple currents:
\begin{align*}
    \mathcal{G}=\left\{\mathcal{E}^s(a) \otimes V_{-\sqrt{\frac{q}{p}}a+L}\mid a\in G \right\},\quad  \mathcal{G}_L=\left\{V_{-\sqrt{\frac{q}{p}}a+L}\mid a\in G \right\},
\end{align*}
parametrized by the same group $G$ defined in \eqref{list of groups} and 
\begin{align}\label{description of centralizer}
\mathcal{G}_L'=\left\{V_{\lambda+L}\mid \lambda\in \sqrt{\tfrac{p}{q}}\Z/L\right\}
\end{align}
parametrized by $G_L':=\sqrt{\tfrac{p}{q}}\Z/L$, which is $\Z_{2q}$ (resp.\ $\Z_q$) for $q=2n-1$ (resp.\ $q=2n$). Then 
\begin{align}\label{mutual center}
    \mathcal{G}_L'=\mathrm{Irr}(V_L\Mod^{\mathcal{G}_L}),\quad \mathcal{G}_L=\mathrm{Irr}(V_L\Mod^{\mathcal{G}_L'})
\end{align}
and $\cA^+$-mod is naturally equivalent to $\mathrm{Rep}(\cA^+)^{\mathcal{G}}$.
The restrictions of \eqref{induction functor}
\begin{align*}
\begin{array}{ccc}
 \ssprW\Mod &\overset{\ind}{\longrightarrow} & \mathrm{Rep}(\cA^+)\\
 M &\mapsto & \cA^+\boxtimes (M\otimes V_L),
\end{array}
\quad 
\begin{array}{ccc}
 V_L\Mod &\overset{\ind}{\longrightarrow} & \mathrm{Rep}(\cA^+)\\
 N &\mapsto & \cA^+\boxtimes (\ssprW\otimes N),
\end{array}
\end{align*}
are embeddings of tensor categories. 
As the $V_L$-modules appearing in
$$ \cA^+\boxtimes (M\otimes V_L)\simeq \bigoplus_{\lambda\in G}(\mathcal{E}^s(a)\boxtimes M)\otimes V_{-\sqrt{\frac{q}{p}}a+L},$$
are exactly those in $\mathcal{G}_L$, \eqref{mutual center} implies the tensor equivalence 
\begin{align*}
    \ind\colon \ssprW\Mod \xrightarrow{\simeq} \mathrm{Rep}(\cA^+)^{\mathcal{G}_L'}.
\end{align*}
On the other hand, $\mathrm{Rep}(\cA^+)$ is generated as a tensor category by the subcategories $\mathrm{Rep}(\cA^+)^{\mathcal{G}}$ and $V_L$-mod whose intersection is  $V_L\Mod^{\mathcal{G}_L}$.
We have the isomorphisms of fusion rings
\begin{align*}
    \mathcal{K}(\mathrm{Rep}(\cA^+))^{\mathcal{G}}\simeq \mathcal{K}(\cA^+)\simeq \mathcal{K}(\ssubW),\quad \mathcal{K}(V_L)\simeq \Z[L^*/L],
\end{align*}
and $L^*/L\simeq \Z_{4pq}$ (resp.\ $\Z_{pq}$) for $q=2n-1$ (resp.\ $q=2n$).
Then we have the isomorphisms of fusion rings \cite[Theorem 2]{CGNS}
\begin{align}
    \mathcal{K}(\ssprW)
\nonumber \simeq \mathcal{K}(\mathrm{Rep}(\cA^+))^{\mathcal{G}_L'}
\nonumber &\simeq \left(\mathcal{K}(\mathrm{Rep}(\cA^+))^{\mathcal{G}}\underset{\mathcal{K}(V_L)^{\mathcal{G}_L}}{\otimes} \mathcal{K}(V_L)\right)^{\mathcal{G}_L'}\\
\label{final isom} &\simeq \left(\mathcal{K}(\ssubW)\underset{\Z[G_L']}{\otimes} \Z[L^*/L]\right)^{\mathcal{G}_L'}.
\end{align}
By construction, the isomorphism \eqref{final isom} describes the set of simple $\ssprW$-modules as the obvious $\Z$-basis.

\begin{theorem}\label{thm:classification_super_case}
For admissible levels $\ell=-h^\vee_-+\frac{q}{2p}$ with $q=2n-1, 2n$, the complete set of simple $\sWsuper$-modules is given by 
$$\Irr(\sWsuper)=\{S_{\theta}\mathbf{L}_\ell^-(\lambda)\mid \lambda \in \Pr_\Z^k,\ 0\leq \theta < p \}.$$
We have the isomorphism 
$S_{p}\mathbf{L}_\ell^-(\lambda)\simeq \mathbf{L}_\ell^-(\varsigma\circ \lambda)$
where $\varsigma$ is the involution defined in \eqref{eq:symmetry} in the proof of Theorem~\ref{branching rule} and the fusion products
\begin{align}\label{fusion for super side}
    \mathbf{L}_\ell^-(\lambda)\boxtimes \mathbf{L}_\ell^-(\mu)\simeq \bigoplus_{\nu\in \Pr_\Z^k}{N_{\lambda,\mu}^\nu}\, S_{a}\mathbf{L}_\ell^-(\nu).
\end{align}
For $N_{\lambda\mu}^\nu\neq0$ (see \eqref{eq:coeff_tensor_product_subreg}), $a=-1$ if $[\lambda],[\mu]=\half$ and $a=0$ otherwise.
\end{theorem}

\proof
We first show \eqref{fusion for super side}. Since $\ind(\mathbf{L}_\ell^-(\lambda)\otimes V_{\frac{q[\lambda]}{\sqrt{pq}}+L})=\mathbf{L}_k(\lambda)\otimes V_\Z$, the module $\mathbf{L}_\ell^-(\lambda)$ is identified with $\mathbf{L}_k(\lambda)\otimes [-\frac{q[\lambda]}{\sqrt{pq}}]$ in \eqref{final isom}. Then,
\begin{align*}
\mathbf{L}_\ell^-(\lambda)\boxtimes \mathbf{L}_\ell^-(\mu)
&=\left(\mathbf{L}_k(\lambda)\otimes [-\tfrac{q[\lambda]}{\sqrt{pq}}]\right)\boxtimes \left(\mathbf{L}_k(\mu)\otimes [-\tfrac{q[\mu]}{\sqrt{pq}}]\right)\\
&=\sum_{\nu\in \Pr_\Z^k}N_{\lambda,\mu}^\nu \left(\mathbf{L}_k(\nu) \otimes [-\tfrac{q([\lambda]+[\mu])}{\sqrt{pq}}]\right)\\
&=\sum_{\nu\in \Pr_\Z^k}N_{\lambda,\mu}^\nu\mathbf{L}_\ell^-(\nu) \boxtimes \left(\mathbf{L}_k(0)\otimes [-\tfrac{q([\lambda]+[\mu]-[\nu])}{\sqrt{pq}}]\right)\\
&=\sum_{\nu\in \Pr_\Z^k}N_{\lambda,\mu}^\nu\mathbf{L}_\ell^-(\nu) \boxtimes S_{*J^-}\mathbf{L}^-_\ell(0)=\sum_{\nu\in \Pr_\Z^k}N_{\lambda,\mu}^\nu S_{*J^-}\mathbf{L}_\ell^-(\nu).
\end{align*}
Here $*=-([\lambda]+[\mu]-[\nu])$, which is $-1$ if $[\lambda],[\mu]=\half$ and $a=0$ otherwise when $N_{\lambda,\mu}^\nu\neq0$ is satisfied.
Next, we describe $\Irr(\ssprW)$. 
Since $J^+=J^-+\frac{p-q}{p}J^+_\Delta$, we have $S_{J^+}=S_{J^-}\circ S_{\frac{p-q}{p}J^+_\Delta}$.
Then
\begin{align*}
S_{\theta J^+}\mathbf{L}_k(\lambda)\otimes V_\Z
&\simeq \bigoplus_{a\in M}S_{(\theta+a)J^-}\mathbf{L}_\ell^-(\lambda)\otimes V_{\frac{\theta(p-q)-q(a-[\lambda])}{\sqrt{pq}}+L}\\
&\simeq \bigoplus_{a\in M}S_{aJ^-}\mathbf{L}_\ell^-(\lambda)\otimes V_{\frac{\theta p-q(a-[\lambda])}{\sqrt{pq}}+L}\\
&\simeq (\mathbf{L}_k(\lambda)\otimes V_\Z)\boxtimes \ind(\W_\ell^-\otimes V_{\frac{\theta p}{\sqrt{pq}}+L}).
\end{align*}
Therefore, $\mathcal{K}(\mathrm{Rep}(\cA^+))$ is spanned by $\mathbf{L}_k(\lambda)\otimes [a]$ ($a\in L^*/L$). Using \eqref{formula for monodromy}, we show that $\mathbf{L}_k(\lambda)\otimes [a]$ lies in $\mathcal{K}(\mathrm{Rep}(\cA^+))^{\mathcal{G}_L'}$ iff $a\in q(\theta-[\lambda])/\sqrt{pq}$ for some $\theta\in \Z$. 
To obtain the $\Z$-basis, we need to find the representatives of the tensor products over $\Z[G_L']$ in \eqref{final isom}.
It follows from \eqref{description of centralizer} that the
elements $M \otimes [a]$ in $\mathcal{K}(\mathrm{Rep}(\cA^+))^{G_2}$ are identified as 
\begin{align*}
    M \otimes [a+\tfrac{bp}{\sqrt{pq}}]=S_{bJ^+}(M) \otimes [a].
\end{align*}
As a consequence, the elements
$$M(\lambda,\theta)=\mathbf{L}_k(\lambda)\otimes [\tfrac{q(\theta-[\lambda])}{\sqrt{pq}}],\quad (\lambda \in  \Pr_\Z^k,\ 0\leq \theta <p)$$
form a $\Z$-basis of $\mathcal{K}(\mathrm{Rep}(\cA^+))^{\mathcal{G}_L'}$ and for $\theta=p$, we have
$M(\lambda,p)\simeq M(\varsigma\circ \lambda,0)$.
To see the corresponding simple $\ssprW$-module more explicitly, it suffices to look at the decomposition of $M(\lambda,\theta)$ as $\ssprW\otimes V_L$-modules and take the multiplicity space of $V_L$.
We then observe that $S_{\theta}\mathbf{L}_\ell^-(\lambda)$ corresponds to $M(\lambda,p)$. 
Hence, $\{S_{\theta}\mathbf{L}_\ell^-(\lambda) \mid \lambda \in \Pr_\Z^k,\ 0\leq \theta <p\}$ gives the complete set of simple $\ssprW$-modules and $S_{p}\mathbf{L}_\ell^-(\lambda)\simeq\mathbf{L}_\ell^-(\varsigma\circ \lambda)$ as desired. 
\endproof

\section{Characters}\label{sec:characters}

In this section, we compute the characters of some modules for the $\W$-superalgebras $\subW$ and $\sprW$ by using resolutions by Wakimoto-type free field representations.

\subsection{Resolution: irrational case}
Recall that when the level $k\notin \Q$ is irrational, the $\W$-algebra $\subW$ is simple. We consider the modules $\weyl_{\mu,f}^{k,+}=H_{f}^0(\weyl_\mu^k)$ obtained by applying the BRST reduction to the Weyl modules $\weyl_\mu^k=U(\widehat{\g}_k)\otimes_{U(\g[\![t]\!])} L_\mu$ where $\g=\so_{2n+1}$ and $\mu\in \rP_+$.
By the Fiebig's equivalence \cite{Fie}, the dual of Bernstein--Gelfand--Gelfand resolution \cite{BGG} $\mathscr{C}_\mu$ of the simple $\g$-module $L_\mu$ gives rise to a resolution of $\weyl_\mu^k$ of the form
\begin{align*}
\mathscr{C}_\mu^k\colon 0\rightarrow \weyl^k_\mu\rightarrow C^{k,0}_{\mu}\rightarrow C^{k,1}_{\mu}\rightarrow \cdots \rightarrow 0,\quad C^{k,i}_{\mu}=\bigoplus_{\lenght{w}=i}\mathbb{W}^k_{\mu_w}
\end{align*}
where the sum for $w$ runs over the Weyl group $W$ and $\mu_w=w\circ\mu=w(\mu+\rho)-\rho$.
Indeed, the complex on the lowest conformal weight subspace $\mathscr{C}_\mu^k$ agrees with $\mathscr{C}_\mu$.
We apply the BRST reduction $H_f^0$ to $\mathscr{C}_\lambda^k$ and obtain a resolution 
\begin{align*}
\mathscr{C}_{\mu,f}^k\colon0\rightarrow \weyl_{\mu,f}^{k,+}\rightarrow C^{k,0}_{\mu,f}\rightarrow C^{k,1}_{\mu,f}\rightarrow \cdots \rightarrow 0,\quad C^{k,i}_{\mu,f}\simeq \bigoplus_{\lenght{w}=i}\mathbb{W}^{k,+}_{\mu_w}
\end{align*}
of $\subW$-modules, which is exact since $\weyl^k_\mu$ is acyclic, i.e. $H^n_f(\weyl_\mu^k)\simeq \delta_{n,0}\weyl_{\mu,f}^{k,+}$.
Finally, applying the functor $\mathrm{H}_{\varepsilon^{-1}(\mu,\varpi_1)}^+$ to $\mathscr{C}_{\mu,f}^k$, we obtain a resolution of the $\sprW$-module $V^{\ell,-}_\mu:=\mathrm{H}_{\varepsilon^{-1}(\mu,\varpi_1)}^+(\weyl_{\mu,f}^{k,+})$:
\begin{align*}
\mathscr{D}_\mu^\ell\colon 0\rightarrow V_\mu^{\ell,-}\rightarrow D^{k,0}_{\mu}\rightarrow D^{k,1}_{\mu}\rightarrow \cdots \rightarrow 0,\quad  D^{k,i}_{\mu}\simeq \bigoplus_{\lenght{w}=i}S_{\bullet} W_{-\tfrac{1}{2\varepsilon^2}\check{\mu}_w}^\ell
\end{align*}
with $\bullet={(\mu_w-\mu,\varpi_{1})}J^-$.
We are now in position to compute the following character formulae.

\begin{proposition}\label{prop:simple_mod_irrational}
For $\mu\in \rP_+$ and $k\notin \mathbb{Q}$, we have
\begin{align*}
\ch{\weyl_{\mu,f}^{k,+}}&=
\frac{q^{\Delta_{0;\mu;0}^k}z^{(\mu,\varpi_1)}}{(q;q)^n_\infty(zq^n,z^{-1}q^{1-n};q)_\infty}\displaystyle{\sum_{w\in W}}(-1)^{\lenght{w}}q^{-(\mu_w-\mu,x_\Gamma^+)}z^{(\mu_w-\mu,\varpi_1)},\\
\ch{{V}_{\mu}^{\ell,-}}&={q^{\Delta_{-\frac{1}{2\varepsilon^2}\check{\mu}}^\ell}z^{-\frac{(\check{\mu},\check{\varpi_0})}{2\varepsilon^2}}}\frac{(-zq^{n+\half},-z^{-1}q^{-n+\half};q)_\infty}{(q;q)^{n+1}_\infty}
\displaystyle{\sum_{w\in W}}\frac{(-1)^{\lenght{w}}q^{-(\mu_w-\mu,x^+_\Gamma+n\varpi_1)}}{1+z^{-1}q^{-(\mu_w-\mu+\rho,\varpi_1)}}.
\end{align*}
\end{proposition}

\proof
    The computation of the characters follows from the characters of the Wakimoto-type modules (Proposition~\ref{principal admissble weights}) together with the Eular-Poincar\'{e} principle
    \begin{equation*}
        \ch{\weyl_{\mu,f}^{k,+}}=\sum_{i\geq 0}(-1)^i\ch{C^{k,i}_{\mu,f}},\quad
        \ch{V_\mu^{\ell,-}}=\sum_{i\geq0}(-1)^i\ch{D^{k,i}_{\mu}}.
    \end{equation*}
    For the computation of $\ch{V_\mu^{\ell,-}}$, we use the character formula for the spectral flow twist $S_{\theta J^-}M$ of a $\subW$-module $M$ 
    \begin{equation}\label{spectral flow twist character on super side}
        \ch{S_{\theta J^-}M}(z,q)=q^{\tfrac{1}{2}\theta^2(1-\varepsilon^{-2})}z^{\theta(1-\varepsilon^{-2})}\ch{M}(zq^\theta,z),
    \end{equation}
    and the identity
    \begin{equation*}
        (-zq^{\theta+n+\half},-z^{-1}q^{-\theta-n+\half};q)_\infty=z^{-\theta}q^{-n\theta-\frac{\theta^2}{2}}(-zq^{n+\half},-z^{-1}q^{-n+\half};q)_\infty,
    \end{equation*}
    for $\theta\in\Z$ and $n\in\Z_{\geq0}$.
\endproof

\subsection{Resolution: admissible case}
Let us consider the exceptional admissible levels $k=-h^\vee+\frac{p}{q}$ as in \eqref{exceptional admissible weights}. Let $L_k(\mu)$ be the simple $L_k(\g)$-module with $\mu\in \Pr_\Z^k$. 
By \cite[Theorem 6.11]{A5}, there exists a complex of $L_k(\g)$-modules
\begin{align*}
\mathscr{C}_\mu^k\colon \cdots \rightarrow C^{k,-1}_{\mu}\rightarrow C^{k,0}_{\mu}\rightarrow C^{k,1}_{\mu}\rightarrow \cdots ,\quad C^{k,i}_{\mu}\simeq\bigoplus\mathbb{W}^k_{\mu_w}
\end{align*}
such that $H^n(\mathscr{C}_\mu^k)\simeq \delta_{n,0}L_k(\mu)$.
Here, the summation on $C^{k,i}_{\mu}$ for $w$ runs over the elements in the (affine) integral Weyl group $\widehat{W}(\mu)\subset \widehat{W}$ satisfying $\ell^{\frac{\infty}{2}}_\mu(w)=i$.
We apply $H_f^0$ to $\mathscr{C}_\mu^k$ and obtain a complex
\begin{align*}
\mathscr{C}_{\mu,f}^k\colon \cdots \rightarrow C^{k,-1}_{\mu,f}\rightarrow C^{k,0}_{\mu,f}\rightarrow C^{k,1}_{\mu,f}\rightarrow \cdots ,\quad C^{k,i}_{\mu,f}=\bigoplus\mathbb{W}^{k,+}_{\mu_w}.
\end{align*}
The complex $\mathscr{C}_\mu^k\otimes \bigwedge{}^{\frac{\infty}{2}+\bullet}(\g_{>0})$, which computes the cohomology $H_f^0(\mathscr{C}_\mu^k)$, is naturally a double complex. 
The convergence of the spectral sequences associated with double complexes implies that 
\begin{equation*}
    H^n(\mathscr{C}_{\mu,f}^k)\simeq \delta_{n,0}H_f^n(L_k(\mu))\simeq \mathbf{L}_k(\mu),
\end{equation*}
by the cohomology vanishing \cite{A3}. 
As previously, we apply the functor $\mathrm{H}_{\varepsilon^{-1}(\mu,\varpi_1)}^+$ to $\mathscr{C}_{\mu,f}^k$ and obtain a complex of $\sprW$-modules:
\begin{align*}
\mathscr{D}_{\mu,f}^k\colon \cdots \rightarrow D^{k,-1}_{\mu,f}\rightarrow D^{k,0}_{\mu,f}\rightarrow D^{k,1}_{\mu,f}\rightarrow \cdots,\quad  D^{k,i}_{\mu}\simeq \bigoplus S_{\bullet} W_{-\tfrac{1}{2\varepsilon^2}\check{\mu}_w}^\ell
\end{align*}
with $\bullet={(\mu_w-\mu,\varpi_{1})}J^-$. 
It satisfies the cohomology vanishing
$$H^n(\mathscr{D}_\mu^\ell)\simeq \delta_{n,0} \mathrm{H}_{\frac{1}{\varepsilon}(\mu,\varpi_1)}^+(\mathbf{L}_k(\mu)).$$
By Theorem \ref{equivalence of categories}, we have an isomorphism of $\ssprW$-modules
$$\mathrm{H}_{\frac{1}{\varepsilon}(\mu,\varpi_1)}^+(\mathbf{L}_k(\mu))\simeq S_{\bullet}\mathbf{L}_\ell^-(\mu),\quad \bullet=([\mu]-(\mu,\varpi_1))J^-$$.
We obtain the following character formulae for $\mathbf{L}_k(\mu)$ and $\mathbf{L}_\ell^-(\mu)$ ($\mu\in\Pr^k_\Z$).

\begin{proposition}\label{chracter formula 2}
For $k=-h^\vee_++\frac{p}{q}$ (principal) admissible with $q=2n-1$ and $\mu\in \Pr_\Z^k$, we have
\begin{align*}
\ch{\mathbf{L}_k(\mu)}&=
\frac{q^{\Delta_{0;\mu;0}^k}z^{(\mu,\varpi_1)}}{(q;q)^n_\infty(zq^n,z^{-1}q^{1-n};q)_\infty}
\displaystyle{\sum_{\substack{y\in W,\\ \lambda\in\check{Q},\\ w=yt_{q\lambda}}}}(-1)^{l(y)}q^{-(\mu_w-\mu,x^+_\Gamma)+q(\mu_w+\rho,\lambda)-\frac{pq}{2}|\lambda|^2}z^{(\mu_w-\mu,\varpi_1)}
\\
&=\frac{q^{-\frac{pq}{2}|\Lambda_{0,e}^+|^2}z^{-p(\Lambda_{0,e}^+,\varpi_1)}\displaystyle{\sum_{w\in W}}(-1)^{l(w)}\Theta^p_{\Lambda^+_{\mu,w}}(z,q^q;\varpi_1)}{(q;q)_\infty^n(zq^n,z^{-1}q^{1-n};q)_\infty},\\
\ch{S_{\bullet}\mathbf{L}_\ell^-(\mu)}&=
\frac{{q^{\Delta_{-\frac{1}{2\varepsilon^2}\check{\mu}}^\ell}z^{-\tfrac{(\check{\mu},\check{\varpi}_0)}{2\varepsilon^2}}}(-zq^{n+\half},-z^{-1}q^{-n+\half};q)_\infty}{(q;q)^{n+1}_\infty}\\
&\hspace{2cm}\times\displaystyle{\sum_{\substack{y\in W,\\ \lambda\in\check{Q},\\ w=yt_{q\lambda}}}}\frac{(-1)^{l(y)}q^{-(\mu_w-\mu,x^+_\Gamma+n\varpi_1)+q(\mu_w+\rho,\lambda)-\frac{pq}{2}|\lambda|^2}}{1+z^{-1}q^{-(\mu_w-\mu+\rho,\varpi_1)}}
\\
&=q^{-\frac{q}{2p}(\mu,\varpi_1)^2-\frac{pq}{2}|\Lambda_{0,e}|^2+n(\mu,\varpi_1)}z^{q(\Lambda_{\mu,e},\varpi_1)}\frac{\poch{-zq^{\half},-z^{-1}q^{\half}}}{\poch{q}^{n+1}}\\
&\hspace{2cm}\times \sum_{w\in W}(-1)^{l(w)}
\mathscr{A}_{-\Lambda_{\mu,w}}^q(z^{-1},q^{p(\Lambda_{\mu,e},\varpi_1)},q^p ;\varpi_1)
\end{align*}
where 
$$\Lambda^+_{\mu,w}=\frac{1}{p}w(\mu+\rho)-\frac{1}{q}x_{\Gamma}^+,\quad \Lambda_{\mu,w}=\frac{1}{p}w(\mu+\rho)-\frac{1}{q}(x_{\Gamma}^++n\varpi_1)$$
and
\begin{align*}
&\mathscr{A}_\lambda^s(z,w,q;\nu)=
\sum_{\eta\in \lambda+\check{Q}} 
\frac{z^{s(\eta,\nu)} q^{\half s|\eta^2|}}{1-zwq^{(\eta,\nu)}},\quad \Theta_\lambda^s(z,q;\nu)=\mathscr{A}_\lambda^s(z,0,q;\nu),\quad (\lambda,\nu\in\h^*)
\end{align*}
\end{proposition}

\proof
The proof is similar to the irrational case (Proposition~\ref{prop:simple_mod_irrational}).
\endproof
 
\begin{remark}\label{last remark}\hspace{0cm}
\begin{enumerate}[wide, labelindent=0pt]
    \item The character formulae for other simple modules can be deduced by applying \eqref{effect of spectral flow twists} and \eqref{spectral flow twist character on super side}.
    \item On the subregular side, we recover the character formula for the highest weight module $\mathbf{L}_k(\mu)=H^0_f(L_k(\mu))$ \cite{KRW03} with some extra factor corresponding to the shifted $x^+_\Gamma$ to get a symmetric grading.
    \item For the coprincipal admissible levels, i.e. $q=2n$, and for $\lambda\in \Pr_\Z^k$, we obtain similar formulae for $\ch{\mathbf{L}_k(\mu)}$ and $\ch{\mathbf{L}_\ell^-(\mu)}$ by changing the lattice $\check{Q}$ in the definition of $\Theta_\lambda(z,q)$.
    \item The function $\mathscr{A}_\lambda^s(z,w,q;\nu)$ is a generalization of the higher-level Appell-Lerch sums appearing in the character formulae of the simple modules of unitary $\mathcal{N}=2$ minimal models (see \cite{CLRW}).
    For a lattice $L$, consider the family of functions
    $$\mathscr{A}_{\lambda+L}^s(z,w,q;\nu)=\sum_{\eta\in \lambda+L} \frac{z^{s(\eta,\nu)} q^{\half s|\eta^2|}}{1-zwq^{(\eta,\nu)}}$$
    indexed by the weight $\lambda$ and the scalar $s$.
    Then $L=\Z=\check{Q}_{\sll_2}$ and $\lambda=0$ recovers the Appell-Lerch sums whereas $L=\check{Q}_{\so_{2n+1}}$ recovers our case.
\end{enumerate}
\end{remark}

\end{document}